\documentclass[9pt]{extarticle}

\usepackage[margin=1in]{geometry}

\usepackage{amssymb}
\usepackage{bm}
\usepackage{graphicx}
\usepackage[centertags]{amsmath}
\usepackage{amsfonts}
\usepackage{amsthm}
\usepackage{graphicx}
\usepackage{slashed}
\usepackage{float}

\usepackage[usenames,dvipsnames,svgnames,table]{xcolor}
\usepackage[colorlinks=true]{hyperref}
\hypersetup{linkcolor=BrickRed, urlcolor=green, citecolor=blue, linktoc=page}

\numberwithin{equation}{section}

\newtheorem*{proposition*}{Proposition}
\newtheorem*{theorem*}{Theorem}
\newtheorem*{conjecture*}{Conjecture}
\newtheorem*{claim*}{Claim}
\newtheorem*{lemma*}{Lemma}
\newtheorem*{corollary*}{Corollary}
\newtheorem{theorem}{Theorem}[section]
\newtheorem{proposition}[theorem]{Proposition}

\newtheorem*{definition*}{Definition}

\newtheorem*{assumption*}{\mathcal{A}ssumption}

\newtheorem*{remark*}{Remark}
\newtheorem{remark}{Remark}[section]

\newcommand{\snabla}{\slashed{\nabla}}

\numberwithin{equation}{section}
\begin{document}
\title{Semi-global constructions of spacetimes containing curvature singularities}
\author{Yannis Angelopoulos \thanks {The Division of Physics, Mathematics and Astronomy, Caltech,
1200 E California Blvd, Pasadena CA 91125, USA, \newline yannis@caltech.edu}}
\date{}

\maketitle
\begin{abstract}
We construct semi-global $(1+3)$-dimensional Lorentzian spacetimes satisfying the Einstein vacuum equations that contain curvature singularities that are propagated all the way up to future null infinity. Special cases of our constructions are semi-global spacetimes containing the interaction of two impulsive gravitational waves. The spacetimes that we construct can be considered as the semi-global analogues of the \textit{local} spacetimes constructed in \cite{weaknull} by Luk, and in \cite{iwaves1} and \cite{iwaves2} by Luk and Rodnianski. 
\end{abstract}

\section{Introduction}

We consider the characteristic initial value problem for the Einstein vacuum equations
\begin{equation}\label{einstein}
R_{\mu \nu } (g) = 0 .
\end{equation}
The initial data, that are given on two intersecting null hypersurfaces, require \textit{only} integrability in certain weighted norms for the connection coefficients hence allowing the curvature components to be of rather singular nature. 

The goal of the present article is to construct \textit{semi-global} spacetimes that contain curvature singularities. More specifically we construct spacetimes that are geodesically complete in certain directions, that contain a portion of future null infinity, and where there exists a curvature singularity that gets propagated all the way up to future null infinity. 

We present two main results: a) in the first of these results the curvature singularity that reaches future null infinity is a \textit{weak null singularity} (so future null infinity is incomplete in this case), that is in terms of the ingoing parameter $u$ (that takes values in a finite interval) the ingoing traceless second fundamental form $\underline{\hat{\chi}}$ is only in $L^1_u$ but not in any $L^p_u$ for any $p > 1$, and at the same time we impose very weak decay in terms of the outgoing parameter $v$ (that reaches infinity) for the outgoing traceless second fundamental form $\hat{\chi}$ in terms of BV-type norms, b) in the second of these results we impose stronger pointwise decay in $v$ in terms of the $L^2$ norm of $\hat{\chi}$ over spheres, while the singularity is an \textit{impulsive gravitational wave}, that is $\underline{\hat{\chi}}$ has a jump discontinuity on the initial ingoing null hypersurface. Below we state two informal versions of our main results.

The first result (as mentioned above) can be considered as a semi-global version of the construction of a weak null singularity that was given by Luk in \cite{weaknull}. In the infinite direction we close the estimate in a scale-invariant norm that can be considered critical from the point of view that if this norm is allowed to be large then a trapped surface should be expected to form (at this point it is worth pointing out that BV norms -- that are also scale-invariant -- were previously used in the work of Christodoulou \cite{DC93} in spherical symmetry).

\begin{theorem}
Assume that we are given data on two intersecting null hypersurfaces $\underline{H}_{v_0}$ and $H_{u_0}$ that intersect on a topological 2-sphere $S_{u_0 , v_0}$ such that:

i) on $\underline{H}_{v_0}$ where $u \in [u_0 , U)$ for some $U < \infty$, we assume that $\hat{\underline{\chi}} \in L^1_u$ and $\hat{\underline{\chi}} \notin L^p_u$ for any $p > 1$ -- in particular we can assume that 
$$ | \hat{\underline{\chi}} |_{\underline{H}_{v_0}} \sim \frac{1}{( U - u) \log^q \left( \frac{1}{U-u} \right)} , $$
for some $q > 1$, and we assume that the rest of the Ricci and Riemann curvature components are smooth,

ii) on $H_{u_0}$ where $v \in [v_0 , \infty )$, we assume that
$$ \int_{v_0}^{\infty} \frac{1}{v} \| \hat{\chi} \|_{L^2 (S_{u_0 , v })} \, dv + \int_{v_0}^{\infty} \frac{1}{v} \| \hat{\chi} \|_{L^2 (S_{u_0 , v })}^2 \, dv \leq C < \infty, $$
and 
$$ \int_{v_0}^{\infty} \| \omega \|_{L^2 (S_{u_0 , v })} \, dv + \int_{v_0}^{\infty} v\| \omega \|_{L^2 (S_{u_0 , v })}^2 \, dv \leq C < \infty , $$
and the rest of the Ricci and Riemann curvature components are smooth and decay in $v$ with appropriate rates.

Then for $U-u_0$ and $C$ appropriately small, there exists a unique $(1+3)$-dimensional Lorentzian spacetime $(\mathcal{M} , g)$ foliated by null hypersurfaces $\underline{H}_v$ and $H_u$ (ingoing and outgoing respectively), that satisfies the Einstein vacuum equations $R_{\mu \nu} (g) = 0$ for $u_0 \leq u < U$ and $v_0 \leq v < \infty$. Moreover, the spacetime cannot be extended past $u=U$ in the sense that the Christoffel symbols fail to be in $L^2_u$ in any neighbourhood of $u=U$. 

\end{theorem}

\begin{figure}[H]
\centering
\includegraphics[width=7cm]{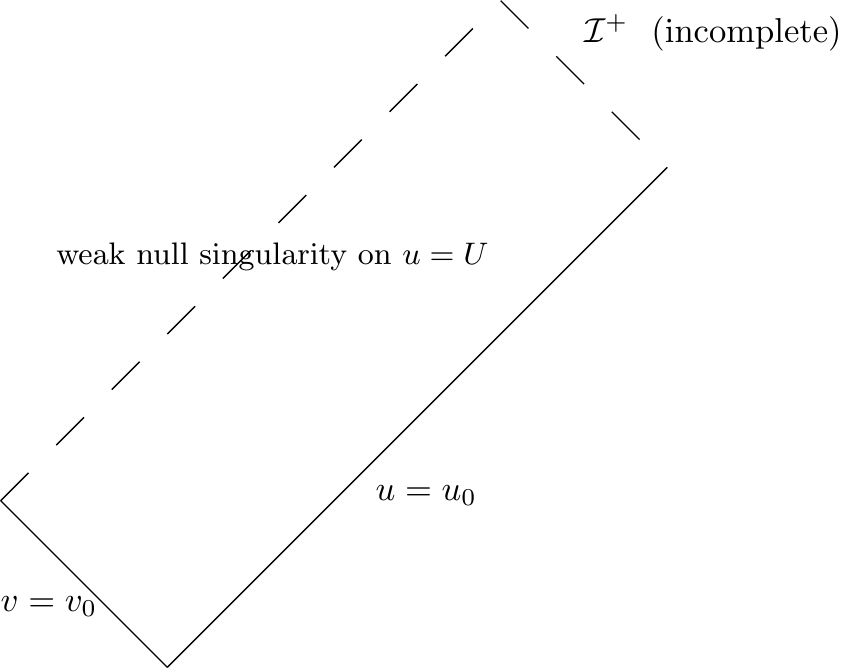}
\caption{The spacetime constructed in Theorem \ref{thm:main} that contains a weak null singularity on $u=U$.}
\end{figure}

The second result can be considered as a semi-global version of the local construction of impulsive gravitational wave spacetimes of \cite{iwaves1}, or of the local construction of colliding impulsive gravitational wave spacetimes of \cite{iwaves2}. 

\begin{theorem}
Assume that we are given data on two intersecting null hypersurfaces $\underline{H}_{v_0}$ and $H_{u_0}$ that intersect on a topological 2-sphere $S_{u_0 , v_0}$ such that:

i) on $\underline{H}_{v_0}$ where $u \in [u_0 , U)$ for some $U < \infty$, we assume that $\hat{\underline{\chi}}$ has a jump discontinuity across a topological 2-sphere $S_{u_j , v_0}$ for some $u_0 < u_j < U$, while the rest of the Ricci and Riemann curvature components are smooth,

ii) on $H_{u_0}$ where $v \in [v_0 , \infty )$, we assume that the data are smooth and satisfy certain decay estimates in $v$ with respect to the $L^2 (S_{u_0 , v} )$ norm.

Then for $U-u_0$ appropriately small, there exists a unique $(1+3)$-dimensional Lorentzian spacetime $(\mathcal{M} , g)$ foliated by null hypersurfaces $\underline{H}_v$ and $H_u$ (ingoing and outgoing respectively), that satisfies the Einstein vacuum equations $R_{\mu \nu} (g) = 0$ for $u_0 \leq u < U$ and $v_0 \leq v < \infty$. Moreover, the jump discontinuity for $\hat{\underline{\chi}}$ gets propagated along $H_{u_j}$ (the outgoing null hypersurface emanating from $S_{u_j , v_0}$) and the curvature components $\underline{\alpha}_{AB} = R \left( \frac{\partial}{\partial \theta^A} , \frac{1}{\Omega^2} \frac{\partial}{\partial u} ,  \frac{\partial}{\partial \theta^B} , \frac{1}{\Omega^2} \frac{\partial}{\partial u} \right)$, $A, B \in \{1,2\}$, where $\Omega (u,v)$ is a function to be specified later and $(\theta^1 , \theta^2)$ are coordinates on the 2-spheres $S_{u,v}$, are measures supported on $H_{u_j}$, while the spacetime is smooth outside $H_{u_j}$.  

\end{theorem}

Note that we could also place a jump discontinuity on a hypersurface $\underline{H}_{v_j}$ for some $v_0 < v_j < \infty$ such that $v_j - v_0$ is appropriately small, and apply the main result of \cite{iwaves2}, and then extend the spacetime all the way to future null infinity with the assumptions of Theorem \ref{thm:main2} on $H_{u_0}$, hence constructing a semi-global spacetime that contains the \textit{interaction} of two impulsive gravitational waves.

\begin{figure}[H]
\centering
\includegraphics[width=10cm]{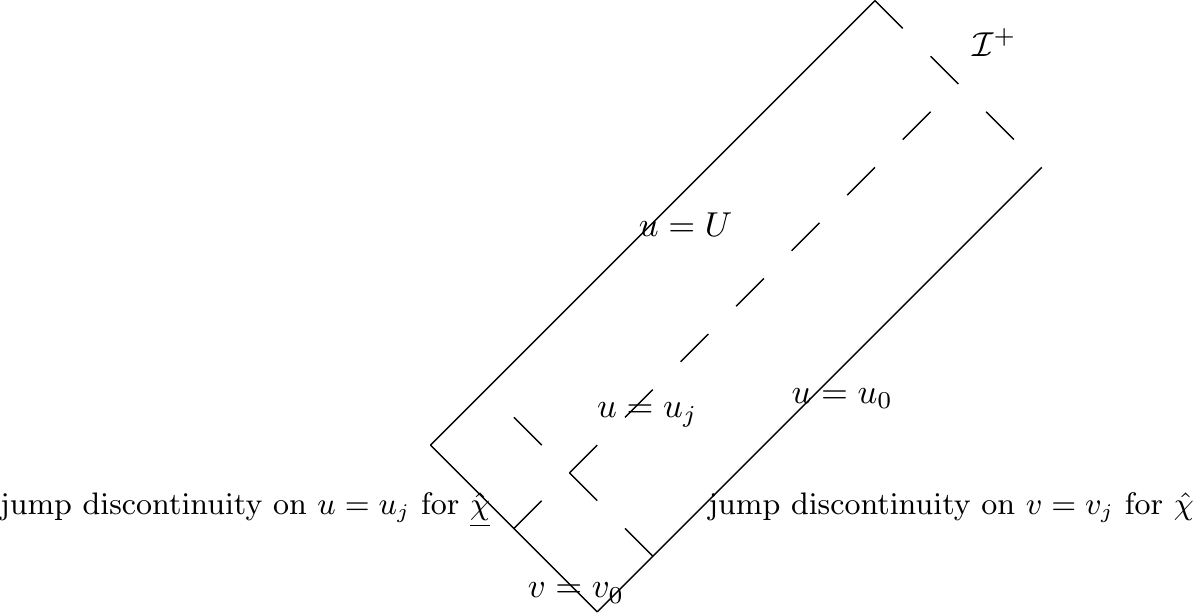}
\caption{The spacetime construced in Theorem \ref{thm:main2} in the special case that it contains the interaction of two impulsive gravitational waves, where one of them gets propagated all the way up to future null infinity.}
\end{figure}

\subsection{Historical remarks}
We give here a brief overview (that is anything but complete) of the previous work done concerning the construction of spacetimes containing curvature singularities.

We start with impulsive gravitational waves. The starting point for their study was the work of Brinkmann \cite{brinkmann} who introduced of the $pp$-wave spacetimes (that contained no curvature singularities though). Penrose in \cite{penrose} was the first to construct an impulsive gravitational wave spacetime, where the Riemann curvature tensor contained a delta singularity across a plane null hypersurface. On the other hand that spacetime is plane symmetric, non-globally hyperbolic, and non-asymptotically. The next important contribution in the area came from Khan and Penrose \cite{khanpenrose} who gave an explicit example of a spacetime representing the \textit{collision} of two impulsive gravitational waves. The metric of that spacetime has the following form:
$$ g = -2dudv +  dx^2 + dy^2 \mbox{  for $u \leq 0$, $v \leq 0$,} $$
$$ g = -2dudv + (1-v)^2 dx^2 + (1+v)^2 dy^2 \mbox{  for $u < 0$, $v > 0$,} $$
$$ g = -2dudv + (1-u)^2 dx^2 + (1+u)^2 dy^2 \mbox{  for $u >0$, $v < 0$,} $$
\begin{align*}
 g =& - \frac{2 (1-u^2 -v^2 )^{3/2}}{\sqrt{(1-u^2 ) (1-v^2 )}( uv + \sqrt{(1-u^2 ) (1-v^2 )})} dudv \\ & + (1-u^2 -v^2 ) \left( \frac{1-u \sqrt{1-v^2} - v \sqrt{1-u^2}}{  1+u \sqrt{1-v^2} + v \sqrt{1-u^2}} dx^2 +\frac{1+u \sqrt{1-v^2} + v \sqrt{1-u^2}}{  1-u \sqrt{1-v^2} - v \sqrt{1-u^2}} dy^2 \right) \mbox{  for $u , v > 0$}, 
 \end{align*}
and as the curvature blows up towards the spacelike hypersurface $u^2 + v^2 = 1$, the spacetime is neither globally nor even semi-globally geodesically complete. Further examples of colliding impulsive gravitational wave spacetimes were given by Szekeres \cite{szekeres}, Nutku-Halil \cite{nutkuhalil}, and Hauser-Ernst \cite{hauserernst}. We should also mention the work of Yurtsever \cite{yurtsever} who studied the interaction of ``almost plane waves", following the interaction for large but finite $v$ (in the coordinates described in the Theorems above). On the other hand Bell and Szekeres \cite{bellszekeres} constructed spacetimes of colliding electromagnetic waves (these can be seen as impulsive gravitational waves for the Einstein--Maxwell equations). The metric of these spacetimes is similar to the one of Khan and Penrose in the vacuum case, in the regions $\{ u,v \leq 0\}$, $\{ u < 0 , v > 0 \}$, $\{ u > 0 , v < 0 \}$, while for $\{ u , v > 0 \}$ there exists only a coordinate singularity beyond which the spacetime is extendible. Perturbations of such spacetimes have been studied by Chandrasekhar and Xanthopoulos, see \cite{chandxanth}. For a complete list of examples see the book \cite{griffiths}. Other examples of colliding impulsive gravitational wave spacetimes were given by Ferrari and Ib\'{a}\~{n}ez \cite{ferrariibanez1}, \cite{ferrariibanez2} and by Chandrasekhar and Xanthopoulos in \cite{chandxanth86}.

\begin{figure}[H]
\centering
\includegraphics[width=8cm]{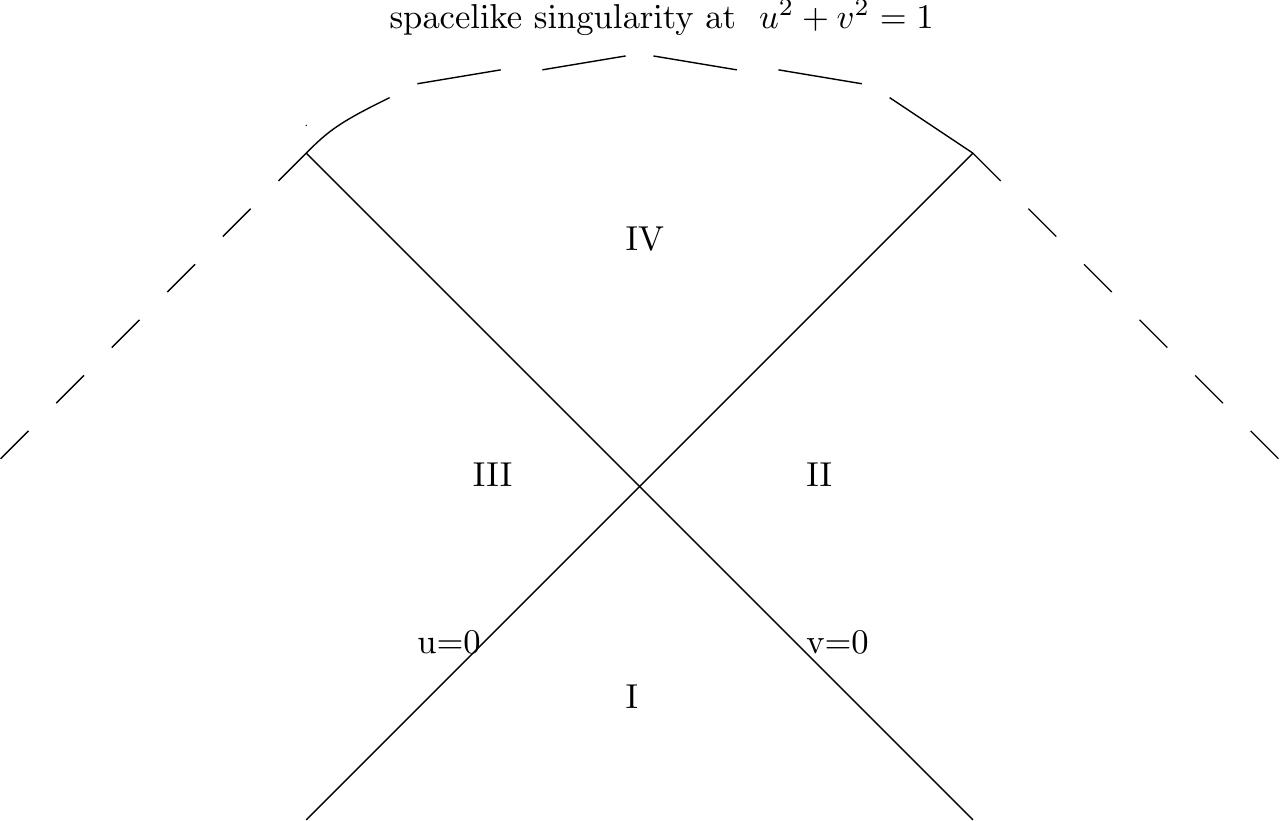}
\caption{The Khan--Penrose example where the hypersurface $u^2 + v^2 = 1$ is a spacelike singularity.}
\end{figure}

In all of the aforementioned examples the constructions of the impulsive gravitational wave spacetimes was done explicitly by assuming that the spacetime has certain symmetries. The first constructions of impulsive gravitational wave spacetimes \textit{without any symmetry assumptions} were done by Luk and Rodnianski in the groundbreaking works \cite{iwaves1} and \cite{iwaves2}. They studied the characteristic initial value problem for the Einstein vacuum equations in the double null foliation gauge, where the spacetime is foliated by null hypersurfaces intersecting on topological 2-spheres, imposing jump discontinuities for the second fundamental forms of the spheres. They showed the local existence and uniqueness of such solutions and they were able to follow the solution \textit{past} the interaction of the impulsive gravitational waves (as an aside we note that in \cite{iwaves2} they also obtained an alternative proof of Christodoulou's formation of trapped surfaces result \cite{DC09}). Recently they were also able to construct spacetimes modelling the interaction of null dust shells \cite{jonathanigor-null} using the results of \cite{iwaves1} and \cite{iwaves2}.

\begin{figure}[H]
\centering
\includegraphics[width=8cm]{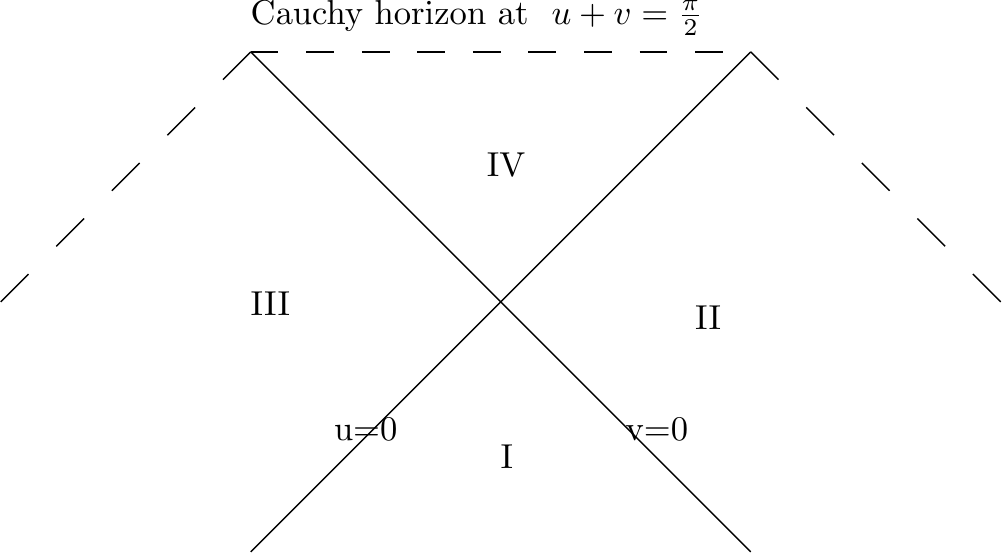}
\caption{The Bell-Szekeres example (a solution of the electrovacuum equations) where the hypersurfaces $u=0$ and $v=0$ contain singularities for the Maxwell components and where the hypersurface $u+v = \pi/2$ is a Cauchy horizon through which there is non-unique way to extend the spacetime.}
\end{figure}

On the other hand the study of weak null singularities -- which are null boundaries of Lorentzian spacetimes satisfying the Einstein vacuum equations where the metric is continuous but the Christoffel symbols blow up and are not in $L^2$ -- was initiated within the context of the Strong Cosmic Censorship conjecture (see the work \cite{dafermosluk} for a detailed version and a detailed history of this important problem of general relativity). The first construction of weak null singularities were given by Dafermos in the works \cite{MD05c} and \cite{MD03} within the context of spherically symmetric spacetimes satisfying the Einstein--Maxwell--scalar field equations. On the other hand, outside any symmetry assumptions, but in the class of analytic spacetimes, the first such constructions were given by Ori and Flanagan \cite{oriflanagan}. Without any symmetry assumptions and outside the class of analytic spacetimes, the first constructions were given in the important work of Luk \cite{weaknull} using several of the techniques of \cite{iwaves2}. To achieve such constructions he introduced a new class of weighted norms (that we also use in the current work). It's worth noting that Luk's construction is essentially a local one although it is intextendible in the sense that the Christoffel symbols blow up past its boundaries.

All the aforementioned results that are not given by explicit solutions are local constructions (with the exception of the ones of Dafermos). A natural problem is to try and extend these local constructions to global ones. In general relativity, due to the supercritical nature of the problem, a global construction of a geodesically complete spacetime is in principle a hard problem. The most famous example of a global construction is the monumental work of Christodoulou and Klainerman \cite{christab} that addresses and completely resolves the problem of the stability of the Minkowski spacetime (see also \cite{klainermannicolo}), which is essentially a small data problem and hence does not rely on the use of any conserved quantities. More recently there has been some progress on the stability of the Kerr black hole spacetimes (see \cite{klainerman17}, \cite{Dafermos2016}), which are also global constructions close to a known family of solutions. An easier problem is to construct a semi-global spacetime, i.e. one where the outgoing variable $v$ reaches infinity, hence constructing also part of future null infinity. In this area, let us mention the work of Li and Zhu \cite{lizhu}, which can be seen as a semi-global version of the construction given in \cite{lukimrn}. 

When a spacetime contains a curvature singularity, two natural questions are: a) where is this singularity supported, and b) if it creates other singularities in the spacetime. The topic of propagation of singularities for hyperbolic equations has a long history and we refer the reader to the works \cite{bony}, \cite{metivier}, \cite{majda}, \cite{melrose1}, \cite{melrose2}, and \cite{rauch}, for many different examples on local and global propagation of singularities (see also the book of Beals \cite{beals}). 

\subsection{Motivation}

Three main motivations for the present work are the following:

i) It is a natural question to extend the local constructions of impulsive gravitational wave spacetimes and of spacetimes containing weak null singularities to semi-global ones, constructing a spacetime that is geodesically complete in some directions, along with a portion of its null infinity. Note that in the case of Theorem \ref{thm:main} we are able to construct a spacetime whose null infinity is incomplete (our construction though is not related to the recent important work of Rodnianksi and Shlapentokh--Rothman \cite{igoryakov1}, \cite{igoryakov2} and has no connection to the Weak Cosmic Censorship conjecture).

ii) Whether the singularity creates other (possibly catastrophic) singularities for finite $v$. This is answered in the negative in our situation. The results of the main two Theorems should also be contrasted with the results of Penrose \cite{penrose65} on sandwich waves, where he notes that a spacetime that contains an impulsive gravitational wave supported on a plane cannot be globally hyperbolic, and the construction of an expanding impulsive gravitational wave from Penrose in \cite{penrose}, where two copies of Minkowski are glued on the two sides of a null cone (something that can be seen as a spacetime containing an impulsive gravitational wave supported on a 2-sphere) and the resulting spacetime contains topological singularities. Note also that in the case that we impose an initial jump over a 2-sphere for $\hat{\underline{\chi}}$ the size of the jump for all $v$ remains controlled by the initial jump. This is in contrast to the example of Rauch \cite{rauch} which imposes a jump over a plane in the data for a linear first order hyperbolic system and shows that the size of the jump grows towards the future. 

\begin{figure}[H]
\centering
\includegraphics[width=6cm]{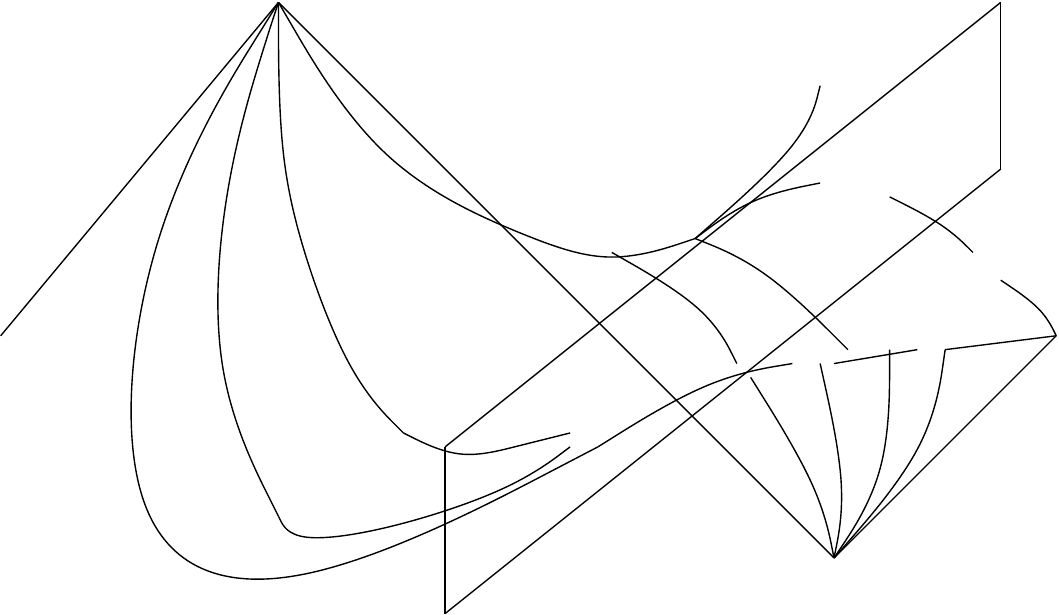}
\caption{The focusing property in the sandwich wave example of Penrose.}
\end{figure}

iii) Whether a possibly catastrophic singularity is created for finite $v$ due to the interaction of the initial singularities. This is also answered in the negative in our situation, where the singularity is initially supported and propagated along topological 2-spheres. Note that this is not the case in the Khan-Penrose example, where the singularity is propagated along topological 2-planes, and unlike in our result it creates another catastrophic curvature singularity. Whether this behaviour in the Khan-Penrose example is an artefact of plane symmetry, or it is caused by the fact that the singularity is supported on 2-dimensional surfaces of infinite area, is not known, the stability or instability of the Khan-Penrose spacetime remains an interesting and possibly hard open problem. 

iv) The construction of a global future geodesically complete spacetime from Cauchy (and not characteristic) data containing colliding impulsive gravitational waves (ongoing work). The semi-global problem studied here will be part of such a construction.

\subsection{Outline of the paper}

\quad \quad In section \ref{dnf} we introduce the double null foliation gauge and some notational conventions. In the next section \ref{ccc} we introduce the Ricci and Riemann curvature components through which we study the Einstein vacuum equations that we state in section \ref{eq}, which are the equations satisfied by the aforementioned components. 

In section \ref{dn} we define the norms that we employ throughout the paper, and which we use to precisely state and prove our main results. These results are stated in detail in section \ref{mainthm} along with an outline of their proofs.

In section \ref{cid} we solve the constraint equations in the setting of the characteristic initial value problem satisfying the assumptions of the two main Theorems. In section \ref{bi} we state and prove some basic inequalities that hold true under the assumptions of the two main Theorems.

In sections \ref{p1} and \ref{p2} we prove Theorems \ref{thm:main} and \ref{thm:main2} respectively, and finally in section \ref{aside} we state some variations of the two main Theorems and sketch their proofs.

\section{Acknowledgements}
I am deeply indebted to Jonathan Luk for suggesting the problem and for having several discussions about it.

\section{Double null foliation and coordinate system}\label{dnf}

Let $( \mathcal{M} , g )$ be the manifold with its associated metric that we are studying, where $\mathcal{M} = \mathcal{M}_2 \times \mathbb{S}^2$ for $\mathcal{M}_2$ a connected open subset of $\mathbb{R}^2$. Let $(u,v)$ be the standard Euclidean coordinates on $\mathcal{M}_2$ and $(\theta^1 , \theta^2 )$ be local coordinates on $\mathbb{S}^2$, so in the end we consider $(u,v,\theta^1 , \theta^2 )$ as local coordinates on $\mathcal{M}$. We define $\gamma$ to be a Riemannian metric on $\mathbb{S}^2$, $b$ a vector field on $\mathbb{S}^2$, and $\Omega$ a positive function defined on $\mathcal{M}_2$. With respect to the aforementioned quantities, the metric $g$ is defined as:
\begin{equation}\label{eq:metric}
g = -2\Omega^2 ( du \otimes dv+ dv \otimes du ) + \gamma_{AB} ( d\theta^A - b^A dv ) \otimes ( d\theta^B - b^B dv ) ,
\end{equation}
and is a Lorentzian metric on $\mathcal{M}$, hence with the above choices $(\mathcal{M} , g)$ is a smooth Lorentzian 4-manifold. Let $D$ denote the associated Levi-Civita connection,and we define the vector fields
$$ L = -2\Omega^2 Du \Rightarrow \quad L = \frac{\partial}{\partial v} + b^A \frac{\partial}{\partial \theta^A} , \quad \underline{L} = -2\Omega^2 Dv \Rightarrow \quad \underline{L} = \frac{\partial}{\partial u} , $$
which have the properties that
$$ Lv = \underline{L} u = 1 \mbox{  and  } \underline{L}v = Lu = 0 .$$
Since the metric $g$ has the form \eqref{eq:metric} we have that
\begingroup
\allowdisplaybreaks
\begin{align*}
& g \left( \frac{\partial}{\partial u} , \frac{\partial}{\partial u} \right) = 0 , \quad g \left( \frac{\partial}{\partial v} , \frac{\partial}{\partial v} \right) = \gamma_{AB} b^A b^B , \quad g \left( \frac{\partial}{\partial u } , \frac{\partial}{\partial v} \right) = g \left( \frac{\partial}{\partial v } , \frac{\partial}{\partial u} \right) = - 2 \Omega^2 , \\ & g \left( \frac{\partial}{\partial \theta^A} , \frac{\partial}{\partial \theta^B} \right) = g \left( \frac{\partial}{\partial \theta^B} , \frac{\partial}{\partial \theta^A} \right) = \gamma_{AB} = \gamma_{BA} , \quad g \left( \frac{\partial}{\partial u} , \frac{\partial}{\partial \theta^A} \right) = g \left( \frac{\partial}{\partial \theta^A} , \frac{\partial}{\partial u} \right) = 0, \\ & g \left( \frac{\partial}{\partial v} , \frac{\partial}{\partial \theta^A} \right) = g \left( \frac{\partial}{\partial \theta^A} , \frac{\partial}{\partial v} \right) = - \gamma_{AB} b^B , \quad g ( L , \underline{L} ) = -2 \Omega^2 , \mbox{  for $A , B \in \{1,2\}$.}
 \end{align*}
 \endgroup
The inverse metric $g^{-1}$ takes the form
$$ g^{-1} = - \frac{1}{2\Omega^2} ( L \otimes \underline{L} + \underline{L} \otimes L ) + \gamma^{-1} . $$
We denote the level sets of $u$ by $H_u$ and the level sets of $v$ by $\underline{H}_{v}$. The intersections of such level sets are topological 2-spheres denoted by $S_{u, v} = H_u \cap \underline{H}_{v}$. We require that initially
$$ \frac{\partial \theta^A}{\partial u} = 0 \mbox{  for $A \in \{1,2\}$ on $\underline{H}_{v_0}$,  } $$
$$ \frac{\partial \theta^A}{\partial v} = 0 \mbox{  for $A \in \{1,2\}$ on $H_{u_0}$, } $$
$$ u=0 \mbox{  on $H_{u_0}$,  } v = 0 \mbox{  on $\underline{H}_{v_0}$ . } $$
Note that according to the above definitions it can be checked that $u$ and $v$ as functions from $\mathcal{M}_2$ satisfy the eikonal equation in $\mathcal{M}$ (see Proposition 2.1 of\cite{dafermosluk}), hence
$$ g ( Du , Du  ) = (g^{-1})^{\alpha \beta} \partial_{\alpha} u \partial_{\beta} u = 0 \quad \quad g ( Dv , Dv ) = (g^{-1})^{\alpha \beta} \partial_{\alpha} v \partial_{\beta} v = 0 , $$
and $\theta^1 , \theta^2$ as functions on the sections $S_{u,v}$ satisfy $\underline{L} \theta^A = 0$ for $A \in \{1,2\}$.

We define the following vector fields which are geodesic:
$$  L' = \frac{1}{\Omega^2} \left( \frac{\partial}{\partial v} + b^A \frac{\partial}{\partial \theta^A} \right) , \quad \underline{L}' = \frac{1}{\Omega^2} \frac{\partial}{\partial u} \Rightarrow D_{L'} L' = D_{\underline{L}'} \underline{L}' = 0 . $$
Then we define the normalized null pair of vector fields
\begin{equation}\label{vf34}
e_4 \doteq \Omega^2 \underline{L}' = \frac{\partial}{\partial v} + b^A \frac{\partial}{\partial \theta^A} , \quad e_3 \doteq  L' = \frac{1}{\Omega^2} \frac{\partial}{\partial u} .
\end{equation}
We will work with the vector fields $(e_1 , e_2 , e_3 , e_4 )$ for $e_3$ and $e_4$ defined as above \eqref{vf34} and 
$$ e_A = \frac{\partial}{\partial \theta^A} \mbox{  for $A \in \{1,2\}$.} $$
From the above relations we can then deduce that
\begingroup
\allowdisplaybreaks
\begin{align*}
& g ( e_3 , e_3 ) = 0 , \quad g ( e_4 , e_4 ) = 0 , \quad g (e_3 , e_4 ) = g (e_4 , e_3 ) = -2 , \\ & g ( e_3 , e_A ) = g ( e_A , e_3 ) = 0 , \quad g ( e_4 , e_A ) = g ( e_A , e_4 ) = 0 \mbox{  for $A \in \{1,2\}$.}
\end{align*}
\endgroup
Finally we define the following operators and operations:
$$ ( \slashed{\mathrm{div}} F )_{A_1 \dots A_k } \doteq (\gamma^{-1} )^{BC} (\slashed{\nabla}_B F_{C A_1 \dots A_k}   , $$ 
$$ ( \slashed{\mathrm{curl}} F )_{A_1 \dots A_k} \doteq \slashed{\epsilon}^{BC} \slashed{\nabla}_B F_{C A_1 \dots A_k } , $$
$$ ( \slashed{\mathrm{tr}} F )_{A_1 \dots A_{k-1}} \doteq \gamma^{BC} F_{BC A_1 \dots A_{k-1}} , $$
$$ \left( F^1 \widehat{\otimes} F^2 \right)_{AB} \doteq F^1_A F^2_B + F^1_B F^2_A - \gamma_{AB} ( F^1 \cdot F^2 ) \mbox{  for $A, B \in \{1,2\}$}, $$
$$ ( \slashed{\nabla} \widehat{\otimes} h )_{AB} \doteq \slashed{\nabla}_A h_B + \slashed{\nabla}_B h_A - \gamma_{AB} \slashed{\mathrm{div}} h \mbox{  for $A, B \in \{1,2\}$}, $$
$$ F^1 \cdot F^2 = ( \gamma^{-1} )^{AB} F^1_A F^2_B \mbox{  for $A,B \in \{1,2\}$, } $$
$$ G^1 \cdot G^2 \doteq ( \gamma^{-1} )^{AC} (\gamma^{-1} )^{BD} G^1_{AB} G^2_{CD} \mbox{  for $A,B,C,D \in \{1,2\}$, } $$
$$ ( F^1 \widehat{\otimes} F^2 )_{AB} \doteq F^1_A F^2_B + F^1_B F^2_A - \gamma_{AB} ( F^1 \cdot F^2 ) \mbox{  for $A, B \in \{1,2\}$}, $$
$$ G^1 \wedge G^2 \doteq \slashed{\epsilon}^{AB} (\gamma^{-1})^{CD} G^1_{AC} G^2_{BD} \mbox{  for $A,B,C,D \in \{1,2\}$}, $$
$$ \mbox{*} h_A \doteq \gamma_{AC} \slashed{\epsilon}^{CB} h_B \mbox{  for $A,B,C \in \{1,2\}$}, $$
$$ \mbox{*} G_{AB} = \gamma_{BD} \slashed{\epsilon}^{DC} G_{AC} \mbox{  for $A,B,C,D \in \{1,2\}$}, $$
where $F^1$, $F^2$, $h$ are 1-forms, $G^1$, $G^2$, $G$ are symmetric 2-forms, $F$ is a totally symmetric tensor, and $\slashed{\epsilon}$ is the volume form associated to $\gamma$.

\section{Connection and curvature coefficients}\label{ccc}
We will study the Einstein equations \eqref{einstein} using the double null frame $(e_3 , e_4 )$ and the orthonormal frame $(e_1 , e_2 )$ tangent to the 2-spheres $S_{u , v}$, always following the convention $A , B \in \{1,2\}$, by recasting \eqref{einstein} as a system of equations for the Ricci (or connection) coefficients
\begingroup
\allowdisplaybreaks
\begin{align*}
& \chi_{AB} \doteq g (D_{e_A} e_4 , e_B ) , \quad \underline{\chi}_{AB} \doteq g (D_{e_A} e_3 , e_B ) , \\ & \eta_A \doteq -\frac{1}{2} g (D_{e_3}  e_A , e_4 ) , \quad \underline{\eta}_A \doteq - \frac{1}{2} g ( D_{e_4} e_A , e_3 ) , \\ & \omega \doteq - \frac{1}{4} g (D_{e_4} e_3 , e_4 ) , \quad \underline{\omega} \doteq - \frac{1}{4} g ( D_{e_3} e_4 , e_3 ) , \\ & \zeta_A \doteq \frac{1}{2} g (D_{e_A} e_4 , e_3 ) ,
\end{align*}
\endgroup
and the null curvature components
\begingroup
\allowdisplaybreaks
\begin{align*}
& \alpha_{AB} \doteq R ( e_A , e_4 , e_B , e_4 ) , \quad \underline{\alpha}_{AB} \doteq R ( e_A , e_3 , e_B , e_3 ) , \\ & \beta_A \doteq \frac{1}{2} R (e_A , e_4 , e_3 , e_4 ) , \quad \underline{\beta}_A \doteq \frac{1}{2} R ( e_A , e_3 , e_3 , e_4 ) , \\ & \rho \doteq \frac{1}{4} R ( e_4 , e_3 , e_4 , e_3 ) , \quad \sigma \doteq  \frac{1}{4} \mbox{*} R ( e_4 , e_3 , e_4 , e_3 ) ,
\end{align*}
\endgroup
where $\mbox{*}R$ is the Hodge dual of the Riemann curvature tensor $R$ which is defined as
$$ \mbox{*} R_{abcd} \doteq \epsilon_{ab\lambda \nu} R^{\lambda \nu cd} .$$ 
We also define the traceless parts of $\chi$ and $\underline{\chi}$ by
$$ \hat{\chi}_{AB} \doteq \chi_{AB} - \frac{1}{2} ( \slashed{\mathrm{tr}} \chi ) \gamma_{AB} , \quad \hat{\underline{\chi}}_{AB} \doteq \chi_{AB} - \frac{1}{2} ( \slashed{\mathrm{tr}} \underline{\chi} ) \gamma_{AB} \mbox{  for $A, B \in \{1,2\}$} . $$
Note that
$$ g ( D_{e_3} e_4 , e_3 ) = - g ( D_{e_3} e_3 , e_4 ) = 0 , $$
as $e_3$ is geodesic, which implies that for our choice of frame we have that
$$ \underline{\omega} = 0 . $$
Also note that as
$$ [ e_4 , e_A ] = - \frac{\partial b^B}{\partial \theta^A} \frac{\partial}{\partial \theta^B} \mbox{  for $A \in \{1,2\}$, } $$
we also have for $A \in \{1,2\}$ that
\begingroup
\allowdisplaybreaks
\begin{align*}
\underline{\eta}_A = & - \frac{1}{2} g ( D_{e_4} e_A , e_3 ) \\ = & - \frac{1}{2} g ( D_{e_A} e_4 , e_3 ) - \frac{1}{2} g ( [e_4 , e_A ] , e_3 ) \\ = & - \frac{1}{2} g ( D_{e_A} e_4 , e_3 ) = - \zeta_A \\ & \Rightarrow \underline{\eta}_A = - \zeta_A .
\end{align*}
\endgroup
On the other hand we have that
\begingroup
\allowdisplaybreaks
\begin{align*}
\eta_A = & - \frac{1}{2} g ( D_{e_3} e_A , e_4 ) \\ = & - \frac{1}{2} g ( D_{e_A} e_3 , e_4 ) - \frac{1}{2} g ( [ e_3 , e_A ] , e_4 ) \\ = & \zeta_A + \frac{1}{2} [ e_A ( \Omega^{-2} ) ] \Omega^2 g ( e_3 , e_4 ) \\ = & - \underline{\eta}_A - [ e_A ( \Omega^{-2} ) ] \Omega^2 \\ = & - \underline{\eta}_A + 2 e_A ( \log \Omega ) \Rightarrow 
\end{align*}
\endgroup
\begin{equation}\label{eq:Ometa}
 e_A ( \log \Omega ) = \frac{1}{2} ( \eta_A + \underline{\eta}_A ) . 
\end{equation}  
Note that for $\omega$ we have that:
\begingroup
\allowdisplaybreaks
\begin{align*}
\omega = & - \frac{1}{4} g (D_{e_4} e_3 , e_4 ) = \frac{1}{4} g ( D_{e_4} e_4 , e_3 ) \\ = & \frac{1}{4} g \left( D_{e_4} ( \Omega^2 \Omega^{-2} e_4 ) , e_3 \right) \\ = & \frac{1}{4} g \left( [ e_4 ( \Omega^2 ) ] ( \Omega^{-2} e_4 ) , e_3 \right) + \frac{1}{4} g \left( \Omega^2 D_{e_4} ( \Omega^{-2} e_4 ) , e_3 \right) \\ = & \frac{1}{4} [ e_4 ( \Omega^{2} ) ] \Omega^{-2} g (e_4 , e_3 ) \\ = & - \frac{1}{2} [ e_4 ( \Omega^{-2} ) ] \Omega^{-2} =  - e_4 ( \log \Omega) \Rightarrow
\end{align*}
\endgroup
\begin{equation}\label{eq:omOm}
\omega = - e_4 ( \log \Omega ) .
\end{equation}
The last equation and equation \eqref{eq:Ometa} imply that
\begin{equation}\label{eq:upuOm}
\frac{\partial ( \log \Omega )}{\partial v} = - \omega - \frac{1}{2} b^A ( \eta_A + \underline{\eta}_A ) .
\end{equation}
We have that
$$ [ e_3 , e_4 ] = \frac{1}{\Omega^2} \frac{\partial b^A}{\partial u } \frac{\partial}{\partial \theta^A} - b^A \left[ \frac{\partial}{\partial \theta^A} \left( \frac{1}{\Omega^2} \right) \right] \frac{\partial}{\partial u} - \left[ \frac{\partial}{\partial \underline{u}} \left( \frac{1}{\Omega^2} \right) \right] \frac{\partial}{\partial u}  , $$
which implies that
$$ g ( [ e_3 , e_4 ] , e_A ) = \frac{1}{\Omega^2}\frac{\partial b^B}{\partial u} \gamma_{BA} . $$
Then we have that
\begingroup
\allowdisplaybreaks
\begin{align*}
\frac{1}{\Omega^2} \frac{\partial b^B}{\partial u} \gamma_{BA} = & g ( [ e_3 , e_4 ] , e_A ) \\ = & g ( D_{e_3} e_4 , e_A ) - g ( D_{e_4} e_3 , e_A ) \\ = & g ( D_{e_3} e_4 , e_A ) + g ( D_{e_4} e_A , e_3 ) \\ = & - g ( D_{e_3} e_A , e_4 ) - 2 \underline{\eta}_A \\ = & 2 \eta_A - 2 \underline{\eta}_A \Rightarrow
\end{align*}
\endgroup
\begin{equation}\label{eq:b}
\frac{\partial b^A}{\partial u} = \Omega^2 \gamma^{AB} ( 2 \eta_B - 2 \underline{\eta}_B ) \mbox{  for any $A \in \{1,2 \}$.}
\end{equation}
For $\gamma$ we have the following two equations:
$$ \slashed{\mathcal{L}}_L \gamma_{AB} = 2 \chi_{AB} , \quad \slashed{\mathcal{L}}_{\underline{L}} \gamma_{AB} = 2 \Omega^2 \underline{\chi}_{AB} \mbox{  for $A, B \in \{1,2\}$,} $$
where $\slashed{\mathcal{L}}_L$ and $\slashed{\mathcal{L}}_{\underline{L}}$ denote the restriction of the Lie derivatives $\mathcal{L}_L$ and $\mathcal{L}_{\underline{L}}$ respectively to $TS_{u,v}$. These last two equations imply that
\begin{equation}\label{eq:lgamma}
\left( \frac{\partial}{\partial \underline{u}} + b^C \frac{\partial}{\partial \theta^C} \right) \gamma_{AB} + \gamma_{AC} \frac{\partial b^C}{\partial \theta^B} + \gamma_{BC} \frac{\partial b^C}{\partial \theta^A} = 2 \chi_{AB} ,
\end{equation}
and
\begin{equation}\label{eq:ulgamma}
\frac{\partial \gamma_{AB}}{\partial u} = 2 \Omega^2 \underline{\chi}_{AB} .
\end{equation}

\section{The equations}\label{eq}

In the rest of the article we denote by $\slashed{\nabla}$ the Levi-Civita connection induced by $\gamma$ on the sections $S_{u , v}$, and by $\slashed{\nabla}_3$ and $\slashed{\nabla}_4$ the projections onto $S_{u , v}$ of the covariant derivatives $D_{e_3}$ and $D_{e_4}$ respectively. We now turn to the equations satisfied by the Ricci coefficients and the null curvature components under the assumption that our spacetime $(\mathcal{M} , g)$ satisfies the Einstein equations \eqref{einstein}. Since we chose our gauge so that $\underline{\omega} = 0$, the Ricci coefficients satisfy the following equations:
\begin{equation}\label{eq:uchi3}
\slashed{\nabla}_3 \hat{\underline{\chi}} + \slashed{\mathrm{tr}} \underline{\chi} \hat{\underline{\chi}} = - \underline{\alpha} ,
\end{equation}
\begin{equation}\label{eq:uchi4}
\slashed{\nabla}_4 \hat{\underline{\chi}} + \frac{1}{2} \slashed{\mathrm{tr}} \chi \hat{\underline{\chi}} = \slashed{\nabla} \hat{\otimes} \underline{\eta} + 2\omega \hat{\underline{\chi}} - \frac{1}{2} \slashed{\mathrm{tr}} \underline{\chi} \hat{\chi} + \underline{\eta} \hat{\otimes} \underline{\eta} , 
\end{equation}
\begin{equation}\label{eq:truchi3}
\slashed{\nabla}_3 \slashed{\mathrm{tr}} \underline{\chi} + \frac{1}{2} ( \slashed{\mathrm{tr}} \underline{\chi} )^2 = - | \hat{\underline{\chi}} |^2 ,
\end{equation}
\begin{equation}\label{eq:truchi4}
\slashed{\nabla}_4 \slashed{\mathrm{tr}} \underline{\chi} + \frac{1}{2} \slashed{\mathrm{tr}} \chi \slashed{\mathrm{tr}} \underline{\chi} = 2 \omega \slashed{\mathrm{tr}} \underline{\chi} + 2 \rho - \hat{\chi} \cdot \hat{\underline{\chi}} + 2 \mathrm{div} \underline{\eta} + 2 | \underline{\eta} |^2 ,
\end{equation}
\begin{equation}\label{eq:chi3}
\slashed{\nabla}_3 \hat{\chi} + \frac{1}{2} \slashed{\mathrm{tr}} \underline{\chi} \hat{\chi} = \slashed{\nabla} \hat{\otimes} \eta - \frac{1}{2} \slashed{\mathrm{tr}} \chi \hat{\underline{\chi}} + \eta \hat{\otimes} \eta , 
\end{equation}
\begin{equation}\label{eq:chi4}
\slashed{\nabla}_4 \hat{\chi} + \slashed{\mathrm{tr}} \chi \hat{\chi} = - 2 \omega \hat{\chi} - \alpha ,
\end{equation}
\begin{equation}\label{eq:trchi3}
\slashed{\nabla}_3 \slashed{\mathrm{tr}} \chi + \frac{1}{2} \slashed{\mathrm{tr}} \underline{\chi} \slashed{\mathrm{tr}} \chi = 2 \rho -\hat{\chi} \cdot \hat{\underline{\chi}} + 2 \slashed{\mathrm{div}} \eta + 2 | \eta |^2 ,
\end{equation}
\begin{equation}\label{eq:trchi4}
\slashed{\nabla}_4 \slashed{\mathrm{tr}} \chi + \frac{1}{2} ( \slashed{\mathrm{tr}} \chi )^2 = -2 \omega \slashed{\mathrm{tr}} \chi - | \hat{\chi} |^2 ,
\end{equation}
\begin{equation}\label{eq:ueta3}
\slashed{\nabla}_3 \underline{\eta} + \mathrm{tr} \underline{\chi} \underline{\eta} = - \underline{\chi} \cdot \eta  - \underline{\hat{\chi}} \cdot \underline{\eta} + \underline{\beta} ,
\end{equation}
\begin{equation}\label{eq:eta4}
\slashed{\nabla}_4 \eta  + \mathrm{tr} \chi \eta = \chi \cdot  \underline{\eta} - \hat{\chi} \cdot \eta - \beta ,
\end{equation}
\begin{equation}\label{eq:omega3}
\slashed{\nabla}_3 \omega = \zeta \cdot (\eta - \underline{\eta} ) - \eta \cdot \underline{\eta} + \frac{1}{2} \rho .
\end{equation}
Note that as $\slashed{\nabla}_3 \underline{\eta} = - \slashed{\nabla}_3 \eta$, we have that:
\begin{equation}\label{eq:eta3}
\slashed{\nabla}_3 \eta + \slashed{\mathrm{tr}} \underline{\chi} \eta = \underline{\chi} \cdot  \underline{\eta} - \underline{\hat{\chi}} \cdot \eta -  \underline{\beta} .
\end{equation}
The Ricci coefficients satisfy also the following Codazzi equations:
\begin{equation}\label{eq:coduchi}
\slashed{\mathrm{div}} \hat{\underline{\chi}} = \frac{1}{2} \slashed{\nabla} \left( \slashed{\mathrm{tr}} \underline{\chi} + \frac{2}{\underline{u}} \right) + \underline{\beta} + \frac{1}{2} ( \eta - \underline{\eta} ) \cdot \left( \hat{\underline{\chi}} - \frac{1}{2} \slashed{\mathrm{tr}} \underline{\chi} \right) ,
\end{equation}
\begin{equation}\label{eq:codchi}
\slashed{\mathrm{div}} \hat{\chi} = \frac{1}{2} \slashed{\nabla} \left( \slashed{\mathrm{tr}} \chi - \frac{2}{\underline{u}} \right) - \beta - \frac{1}{2} ( \eta - \underline{\eta} ) \cdot \left( \hat{\chi} - \frac{1}{2} \slashed{\mathrm{tr}} \chi \right) ,
\end{equation}
\begin{equation}\label{eq:codeta}
\slashed{\mathrm{curl}} \eta = - \slashed{\mathrm{curl}} \underline{\eta} = \sigma + \frac{1}{2} \hat{\underline{\chi}} \wedge \hat{\chi} .
\end{equation}
The null curvature components satisfy the following equations:
\begin{equation}\label{eq:ubeta3}
\slashed{\nabla}_3 \underline{\beta} + 2 \slashed{\mathrm{tr}} \underline{\chi} \underline{\beta} =- \slashed{\mathrm{div}} \underline{\alpha} + \underline{\eta} \cdot \underline{\alpha} , 
\end{equation}
\begin{equation}\label{eq:ubeta4}
\slashed{\nabla}_4 \underline{\beta} + \slashed{\mathrm{tr}} \chi \underline{\beta} =- \slashed{\nabla} \rho +^{*} \slashed{\nabla} \sigma + 2 \omega \underline{\beta} + 2 \hat{\underline{\chi}} \cdot \beta - 3 ( \underline{\eta} \rho -^{*} \underline{\eta} \sigma )  ,
\end{equation}
\begin{equation}\label{eq:beta3}
\slashed{\nabla}_3 \beta + \slashed{\mathrm{tr}} \underline{\chi} \beta = \slashed{\nabla} \rho +^{*} \slashed{\nabla} \sigma + 2 \hat{\chi} \cdot \underline{\beta} + 3 ( \eta \rho +^{*} \eta \sigma ) ,
\end{equation}
\begin{equation}\label{eq:beta4}
\slashed{\nabla}_4 \beta + 2 \slashed{\mathrm{tr}} \chi \beta = \slashed{\mathrm{div}} \alpha - 2 \omega \beta + \eta \alpha ,
\end{equation}
\begin{equation}\label{eq:rho3}
\slashed{\nabla}_3 \rho + \frac{3}{2} \slashed{\mathrm{tr}} \underline{\chi} \rho = - \slashed{\mathrm{div}} \underline{\beta} - \frac{1}{2} \hat{\chi} \cdot \underline{\alpha} + \zeta \cdot \underline{\beta} - 2 \eta \cdot \underline{\beta} ,  
\end{equation}
\begin{equation}\label{eq:rho4}
\slashed{\nabla}_4 \rho + \frac{3}{2} \slashed{\mathrm{tr}} \chi \rho = \slashed{\mathrm{div}} \beta - \frac{1}{2} \hat{\underline{\chi}} \cdot \alpha + \zeta \cdot \beta + 2 \underline{\eta} \cdot \beta ,
\end{equation}
\begin{equation}\label{eq:sigma3}
\slashed{\nabla}_3 \sigma + \frac{3}{2} \slashed{\mathrm{tr}} \underline{\chi} \sigma = - \slashed{\mathrm{div}}^{*} \underline{\beta} + \frac{1}{2} \hat{\chi} \cdot^{*} \underline{\alpha} - \zeta \cdot^{*} \underline{\beta} - 2 \eta \cdot^{*} \underline{\beta} ,
\end{equation}
\begin{equation}\label{eq:sigma4}
\slashed{\nabla}_4 \sigma + \frac{3}{2} \slashed{\mathrm{tr}} \chi \sigma = - \slashed{\mathrm{div}}^{*} \beta + \frac{1}{2} \hat{\underline{\chi}} \cdot^{*} \alpha - \zeta \cdot^{*} \beta - 2 \underline{\eta} \cdot^{*} \beta ,
\end{equation}
\begin{equation}\label{eq:alpha3}
\slashed{\nabla}_3 \alpha + \frac{1}{2} \slashed{\mathrm{tr}} \underline{\chi} \alpha = \slashed{\nabla} \hat{\otimes} \beta - 3 ( \hat{\chi} \rho +^{*} \hat{\chi} \sigma ) + ( \zeta + 4 \eta ) \hat{\otimes} \beta ,
\end{equation}
\begin{equation}\label{eq:ualpha4}
\slashed{\nabla}_4 \underline{\alpha} + \frac{1}{2} \slashed{\mathrm{tr}} \chi \underline{\alpha} = - \slashed{\nabla} \hat{\otimes} \underline{\beta} + 4 \omega \underline{\alpha} - 3 ( \hat{\underline{\chi}} \rho -^{*} \hat{\underline{\chi}} \sigma ) + ( \zeta - 4 \underline{\eta} ) \hat{\otimes} \underline{\beta} .
\end{equation}
Moreover we define the renormalized quantities
$$\check{\rho} = \rho - \frac{1}{2} \hat{\chi} \cdot \underline{\hat{\chi}}, \quad \check{\sigma} \doteq \sigma + \frac{1}{2} \hat{\underline{\chi}} \wedge \hat{\chi} , $$
which satisfiy the following equations
\begin{equation}\label{eq:crho3}
\slashed{\nabla}_3 \check{\rho} + \frac{3}{2} \slashed{\mathrm{tr}} \underline{\chi} \check{\rho} = - \slashed{\mathrm{div}} \underline{\beta} + ( \zeta - 2\eta ) \cdot \underline{\beta} - \frac{1}{2} \hat{\underline{\chi}} \cdot [ ( \slashed{\nabla} \hat{\otimes} \eta ) + ( \eta \hat{\otimes} \eta ) ] + \frac{1}{4} tr \chi | \hat{\underline{\chi}} |^2 ,
\end{equation}
\begin{equation}\label{eq:crho4}
\slashed{\nabla}_4 \check{\rho} + \frac{3}{2} \slashed{\mathrm{tr}} \chi \check{\rho} = \slashed{\mathrm{div}} \beta + ( \zeta + 2 \underline{\eta} ) \cdot \beta - \frac{1}{2} \hat{\chi} [ ( \slashed{\nabla} \hat{\otimes} \underline{\eta} ) + ( \underline{\eta} \hat{\otimes} \underline{\eta} ) ] + \frac{1}{4} tr \underline{\chi} | \hat{\chi} |^2 ,
\end{equation}
\begin{equation}\label{eq:csigma3}
\slashed{\nabla}_3 \check{\sigma} + \frac{3}{2} \slashed{\mathrm{tr}} \underline{\chi} \check{\sigma} = - \slashed{\mathrm{div}}^{*} \underline{\beta} + ( \zeta - 2 \eta ) \wedge \underline{\beta} + \frac{1}{2} \hat{\underline{\chi}} \wedge [ ( \slashed{\nabla} \hat{\otimes} \eta ) + ( \eta \hat{\otimes} \eta ) ] ,
\end{equation}
\begin{equation}\label{eq:csigma4}
\slashed{\nabla}_4 \check{\sigma} + \frac{3}{2} \slashed{\mathrm{tr}} \chi \check{\sigma} = - \slashed{\mathrm{div}}^{*} \beta - ( \zeta + 2 \underline{\eta} ) \wedge \beta + \frac{1}{2} [ ( \slashed{\nabla} \hat{\otimes} \underline{\eta} ) + ( \underline{\eta} \hat{\otimes} \underline{\eta} ) ] \wedge \hat{\chi} . 
\end{equation}
Finally we note that the Gaussian curvature of the surfaces $S$ has the form:
\begin{equation}\label{eq:gauss}
K = - \rho + \frac{1}{2} \hat{\chi} \cdot \hat{\underline{\chi}} - \frac{1}{4} \slashed{\mathrm{tr}} \chi \slashed{\mathrm{tr}} \underline{\chi} ,
\end{equation}
and satisfies the following equations:
\begin{equation}\label{eq:gauss3}
\slashed{\nabla}_3 K + \slashed{\mathrm{tr}} \underline{\chi} K = \slashed{\mathrm{div}} \underline{\beta} - ( \zeta - 2\eta ) \cdot \underline{\beta} + \frac{1}{2} \underline{\hat{\chi}} \cdot ( \slashed{\nabla} \widehat{\otimes} \eta + \eta \widehat{\otimes} \eta ) - \frac{1}{2} \slashed{\mathrm{tr}} \underline{\chi} \slashed{\mathrm{div}} \eta - \frac{1}{2} \slashed{\mathrm{tr}} \underline{\chi} | \eta |^2 ,
\end{equation}
\begin{equation}\label{eq:gauss4}
\slashed{\nabla}_4 K + \slashed{\mathrm{tr}} \chi K = -\slashed{\mathrm{div}} \beta - ( \zeta + 2\underline{\eta} ) \cdot \beta + \frac{1}{2} \hat{\chi} \cdot ( \slashed{\nabla} \widehat{\otimes} \underline{\eta} + \underline{\eta} \widehat{\otimes} \underline{\eta} ) - \frac{1}{2} \slashed{\mathrm{tr}} \chi \slashed{\mathrm{div}} \underline{\eta} - \frac{1}{2} \slashed{\mathrm{tr}} \chi | \underline{\eta} |^2 .
\end{equation}
In terms of $\check{\sigma}$ and the Gaussian curvature $K$ equations \eqref{eq:ubeta4} and \eqref{eq:beta3} take the form:
\begin{equation}\label{eq:ubeta44}
\begin{split}
\slashed{\nabla}_4 \underline{\beta} + \slashed{\mathrm{tr}} \chi \underline{\beta} = & \slashed{\nabla} K + \mbox{*} \slashed{\nabla} \check{\sigma} + 2 \omega \underline{\beta} + 2\hat{\underline{\chi}} \cdot \beta + 3 ( \underline{\eta} K + \mbox{*} \underline{\eta} \check{\sigma} ) - \frac{1}{2} [ \slashed{\nabla} ( \hat{\chi} \cdot \hat{\underline{\chi}} ) + \mbox{*} \slashed{\nabla} ( \hat{\chi} \wedge \hat{\underline{\chi}} ) ] \\ & + \frac{1}{4} [ ( \slashed{\nabla} \slashed{\mathrm{tr}} \chi ) \slashed{\mathrm{tr}} \underline{\chi} + \slashed{\mathrm{tr}} \chi ( \slashed{\nabla} \slashed{\mathrm{tr}} \underline{\chi} ) ] - \frac{3}{2} [ \underline{\eta} ( \hat{\chi} \cdot \hat{\underline{\chi}}) - \mbox{*} \underline{\eta} ( \hat{\chi} \wedge \hat{\underline{\chi}} ) ] + \frac{3}{4} \underline{\eta} \slashed{\mathrm{tr}} \chi \slashed{\mathrm{tr}} \underline{\chi} ,
\end{split} 
\end{equation}
\begin{equation}\label{eq:beta33}
\begin{split}
\slashed{\nabla}_3 \beta + \slashed{\mathrm{tr}} \underline{\chi} \beta = & - \slashed{\nabla} K + \mbox{*} \slashed{\nabla} \check{\sigma} + 2 \hat{\chi} \cdot \underline{\beta} - 3 ( \eta K - \mbox{*} \eta \check{\sigma} ) + \frac{1}{2} [ \slashed{\nabla} ( \hat{\chi} \cdot \hat{\underline{\chi}} ) + \mbox{*} \slashed{\nabla} ( \hat{\chi} \wedge \hat{\underline{\chi}} ) ] \\ & - \frac{1}{4}  [ ( \slashed{\nabla} \slashed{\mathrm{tr}} \chi ) \slashed{\mathrm{tr}} \underline{\chi} + \slashed{\mathrm{tr}} \chi ( \slashed{\nabla} \slashed{\mathrm{tr}} \underline{\chi} ) ] + \frac{3}{2} [ \eta ( \hat{\chi} \cdot \hat{\underline{\chi}}) + \mbox{*} \eta ( \hat{\chi} \wedge \hat{\underline{\chi}} ) ] -  \frac{3}{4} \eta \slashed{\mathrm{tr}} \chi \slashed{\mathrm{tr}} \underline{\chi} .
\end{split}
\end{equation}
It will be more convenient to work with the renormalized quantities:
$$ \slashed{\mathrm{tr}} \chi - \frac{2}{v} , \quad \slashed{\mathrm{tr}} \underline{\chi} + \frac{2}{v} , \quad K - \frac{1}{v^2} . $$
The last quantity can be expressed as:
$$ K - \frac{1}{v^2} = - \check{\rho} - \frac{1}{4} \left( \slashed{\mathrm{tr}} \chi - \frac{2}{v} \right) \left( \slashed{\mathrm{tr}} \underline{\chi} + \frac{2}{v} \right) + \frac{1}{2v} \left( \slashed{\mathrm{tr}} \chi - \frac{2}{v} \right) - \frac{1}{2v} \left( \slashed{\mathrm{tr}}  \underline{\chi} + \frac{2}{v} \right) .$$
The renormalized quantity $K - \frac{1}{v^2}$ satisfies the following equations:
\begin{equation}\label{eq:rgauss3}
\slashed{\nabla}_3 \left( K - \frac{1}{v^2} \right) + \slashed{\mathrm{tr}} \underline{\chi} \left( K - \frac{1}{v^2} \right) = \slashed{\mathrm{div}} \underline{\beta} -\frac{1}{v^2} \slashed{\mathrm{tr}} \underline{\chi}  - ( \zeta - 2\eta ) \cdot \underline{\beta} + \frac{1}{2} \underline{\hat{\chi}} \cdot ( \slashed{\nabla} \widehat{\otimes} \eta + \eta \widehat{\otimes} \eta ) - \frac{1}{2} \slashed{\mathrm{tr}} \underline{\chi} \slashed{\mathrm{div}} \eta - \frac{1}{2} \slashed{\mathrm{tr}} \underline{\chi} | \eta |^2 ,
\end{equation}
\begin{equation}\label{eq:rgauss4}
\slashed{\nabla}_4 \left( K - \frac{1}{v^2} \right) + \frac{3}{2} \slashed{\mathrm{tr}} \chi \left( K - \frac{1}{v^2} \right) = - \slashed{\mathrm{div}} \beta + \frac{1}{2} \slashed{\mathrm{tr}} \chi \left( K - \frac{1}{v^2} - \slashed{\mathrm{div}} \underline{\eta} \right) - ( \zeta + 2\underline{\eta} ) \cdot \beta + \frac{1}{2} \hat{\chi} \cdot ( \slashed{\nabla} \widehat{\otimes} \underline{\eta} + \underline{\eta} \widehat{\otimes} \underline{\eta} )- \frac{1}{2} \slashed{\mathrm{tr}} \chi | \underline{\eta} |^2 .
\end{equation}
Note that the equations (that we will use) for $\slashed{\mathrm{tr}}  \chi - \frac{2}{v}$ and $\slashed{\mathrm{tr}}  \underline{\chi} + \frac{2}{v}$ are the following:
\begin{equation}\label{eq:trchir3}
\slashed{\nabla}_3 \left( \slashed{\mathrm{tr}}  \chi - \frac{2}{v} \right) + \slashed{\mathrm{tr}}  \underline{\chi} \left( \mathrm{tr} \chi - \frac{2}{v} \right) = -2 \left( K - \frac{1}{v^2} \right) + 2 \slashed{\mathrm{div}} \eta - \frac{1}{v} \left( \slashed{\mathrm{tr}}  \underline{\chi} + \frac{2}{v} \right) + | \eta |^2 , 
\end{equation} 
\begin{equation}\label{eq:trchir4}
\slashed{\nabla}_4 \left( \slashed{\mathrm{tr}}  \chi - \frac{2}{v} \right) + \slashed{\mathrm{tr}} \chi \left( \slashed{\mathrm{tr}}  \chi - \frac{2}{v} \right) = \frac{1}{2} \left( \slashed{\mathrm{tr}}  \chi - \frac{2}{v} \right)^2 - 2 \omega \slashed{\mathrm{tr}} \chi - | \hat{\chi} |^2 , 
\end{equation}
\begin{equation}\label{eq:truchir3}
\slashed{\nabla}_3 \left( \slashed{\mathrm{tr}} \underline{\chi} + \frac{2}{v} \right) + \frac{1}{2} \slashed{\mathrm{tr}} \underline{\chi} \left( \slashed{\mathrm{tr}} \underline{\chi} + \frac{2}{v} \right) = \frac{1}{v} \slashed{\mathrm{tr}} \underline{\chi} - | \underline{\hat{\chi}} |^2 ,
\end{equation}
\begin{equation}\label{eq:truchir4}
\slashed{\nabla}_4 \left( \slashed{\mathrm{tr}} \underline{\chi} + \frac{2}{v} \right) + \slashed{\mathrm{tr}} \chi \left( \slashed{\mathrm{tr}} \underline{\chi} + \frac{2}{v} \right) = \frac{2}{v} \left( \slashed{\mathrm{tr}} \chi - \frac{2}{v} \right) - 2 \left( K - \frac{1}{v^2} \right) + 2 \slashed{\mathrm{div}} \underline{\eta} + 2 | \underline{\eta} |^2 .
\end{equation}

\section{Definition of norms}\label{dn}
For an integral on some $S_{u , v}$ we consider a coordinate chart $\{ M_i \}_i $ on it and a partition of unity $\{ p_{M_i} \}_i$ subordinate to this chart and we have for some function $f$ that:
$$ \int_{S_{u,v}} f \doteq \sum_i \int_{-\infty}^{\infty} \int_{-\infty}^{\infty} f p_{M_i} \sqrt{\det \gamma} \, d\theta^1 d\theta^2 . $$ 
Then on the hypersurfaces $H_u$ and $\underline{H}_{v}$ integration can be defined as follows using the form of the metric $g$ for any $u_1 < u_2$ and any $v_1 < v_2$:
$$ \int_{H_u (v_1 , v_2 )} f \doteq 2\sum_i \int_{v_1}^{v_2} \int_{-\infty}^{\infty} \int_{-\infty}^{\infty} f p_{M_i} \Omega \sqrt{\det \gamma} \, d\theta^1 d\theta^2 dv , $$
$$ \int_{\underline{H}_{v} (u_1 , u_2 )} f \doteq 2 \sum_i \int_{u_1}^{u_2} \int_{-\infty}^{\infty} \int_{-\infty}^{\infty} f p_{M_i} \Omega \sqrt{\det \gamma} \, d\theta^1 d\theta^2 d u . $$
Finally for a spacetime region $D_{(u_1 , u_2 ) , (v_1 , v_2 )}$ we have that:
\begin{align*}
 \int_{D_{(u_1 , u_2 ) , (v_1 , v_2 )}} f \doteq & \sum_i \int_{u_1}^{u_2} \int_{v_1}^{v_2} \int_{-\infty}^{\infty} \int_{-\infty}^{\infty} f p_{M_i} \sqrt{- \det g} \, d\theta^1 d\theta^2 dv du \\ = & 2 \sum_i \int_{u_1}^{u_2} \int_{v_1}^{v_2} \int_{-\infty}^{\infty} \int_{-\infty}^{\infty} f p_{M_i} \Omega^2 \sqrt{\det \gamma} \, d\theta^1 d\theta^2 dv du . 
 \end{align*}
We note here that by $D_{u, v}$ we will denote the region $D_{(u_0 , u ) , (v_0 , v)}$.

The $L^p$ norms of arbitrary tensor fields $f$ on the $S$-sections and on the null hypersurfaces $H_u$ and $\underline{H}_{v}$ are defined as follows:
\begin{equation}\label{norm:lps}
\| f\|_{L^p (S_{u,\underline{u}})} = \left( \int_{S_{u,v}} \langle f, f \rangle_{\gamma}^{p/2} \, d\mu_{S} \right)^{1/p} \mbox{  for $1 \leq p < \infty$,}
\end{equation}
\begin{equation}\label{norm:lpu}
\| f \|_{L^p (H_u)} = \left( \int_{v_0}^{\infty}\int_{S_{u,v}} \langle f , f \rangle_{\gamma}^{p/2} \, d\mu_{S} dv \right)^{1/p} \mbox{  for $1 \leq p < \infty$,}
\end{equation}
\begin{equation}\label{norm:lpubar}
\| f \|_{L^p (\underline{H}_{v})} = \left( \int_{u_0}^{U} \int_{S_{u,v}} \langle f , f \rangle_{\gamma}^{p/2} \, d\mu_{S} du\right)^{1/p} \mbox{  for $1 \leq p < \infty$,}
\end{equation}
\begin{equation}\label{norm:linfty}
\| f \|_{L^{\infty} (S_{u,v} )} = \sup_{\theta \in S_{u,v}} \langle f , f \rangle_{\gamma}^{1/2} ( \theta ) . 
\end{equation}
Finally we define we following mixed norms (again for arbitrary tensor fields $f$) for some $U \in (u_0 , \infty)$ , $V \in (v_0 , \infty ]$:
\begin{equation}\label{norm:pqr1}
\| f \|_{L^p_u L^q_v L^r_S } \doteq \left( \int_{u_0}^U \left( \int_{v_0}^V \left( \int_{S_{u,v}} \langle f , f \rangle_{\gamma}^{r/2} d\mu_S \right)^{q/r} \, dv \right)^{p/q} \, du \right)^{1/p} ,
\end{equation}
\begin{equation}\label{norm:pqr2}
\| f \|_{L^p_v L^q_u L^r_S } \doteq \left( \int_{v_0}^V \left( \int_{u_0}^U \left( \int_{S_{u,v}} \langle f , f \rangle_{\gamma}^{r/2} d\mu_S \right)^{q/r} \, du \right)^{p/q} \, dv \right)^{1/p} .
\end{equation}

\section{The main theorems}\label{mainthm}
The main result of the paper is the following theorem where we make use of scaling invariant norms on $\hat{\chi}$ on the hypersurfaces $H_u$ while on the $\underline{H}_v$ hypersurfaces we work with weighted norms that allow us to consider $\hat{\underline{\chi}}$ belonging only to $L^1_u$ (these are the same norms as in \cite{weaknull}).
\begin{theorem}\label{thm:main}
Let $U \in ( u_0 , \infty )$ and $V = \infty$. For given initial data on $H_{u_0}$ for $v_0 \leq v < \infty$ and on $\underline{H}_{\underline{u}_0}$ for $u_0 \leq u \leq U$ where $U - u_0 \leq \epsilon$ and $\epsilon$ is appropriately small, we assume that
$$ c_1 \leq \left. \det \gamma \right|_{S_{u , v_0}} \leq c_2 , $$
$$ \left. \det \gamma \right|_{S_{u_0 , v}} \geq c_3 v^4 , $$
$$ \sum_{k \leq 3} \left( \left| \left. \left( \frac{\partial}{\partial \theta} \right)^k \gamma \right|_{S_{u , v_0}} \right| \leq c_4 , \quad \left| \left. \left( \frac{\partial}{\partial \theta} \right)^k \gamma \right|_{S_{u_0 , v}} \right| \right) \leq c_5 v^2 , $$
\begin{align*}
 \mathcal{I}_1 \doteq \sum_{i=0}^3 \sum_{\psi_0 \in \{ \eta , \underline{\eta} \}} \Big( \| v_0 ( v_0 \slashed{\nabla} )^i \psi_0 \|_{L^{\infty}_u L^2 (S_{u ,v_0 } )} + & \| v ( v \slashed{\nabla} )^i \psi_0 \|_{L^{\infty}_v L^2 (S_{u_0 ,v } )} \\ + &  \| ( v \slashed{\nabla} )^4 \psi_0 \|_{L^2_v L^2 (S_{u_0 , v } )} + \| v_0 ( v_0 \slashed{\nabla} )^4 \psi_0 \|_{L^2_u L^2 (S_{u , v_0 } )}  \Big) \leq \epsilon_1 , 
 \end{align*}
\begin{align*}
\mathcal{I}_2 \doteq \sum_{i=0}^4 \Bigg[ \sum_{\psi_1 \in \{ \hat{\chi} , \slashed{\mathrm{tr}} \chi - \frac{2}{v} \} }   \Bigg( \left\| \frac{1}{\sqrt{v}} ( v \slashed{\nabla} )^i \psi_1 \right\|_{L^2_v  L^2 (S_{u_0,v} )} +& \left\| \frac{1}{v} ( v \slashed{\nabla} )^i \psi_1 \right\|_{L^1_v L^2 (S_{u_0 , v} ) } \Bigg) \\ & + \| v^{1/2} ( v \slashed{\nabla} )^i \omega \|_{L^2_v  L^2 (S_{u_0 ,v} ) } + \| (v \slashed{\nabla} )^i \omega \|_{L^1_v L^2 (S_{u_0 , v} )} \Bigg] \leq \epsilon_2 , 
\end{align*}
$$ \mathcal{I}_3 \doteq \sum_{i=0}^4 \left( \| \mathrm{w} (u ) (v_0 \slashed{\nabla} )^i \underline{\hat{\chi}} \|_{L^2_u L^2 (S_{u , v_0 } )} +  \left\| \mathrm{w} (u )  v_0 ( v_0 \slashed{\nabla} )^i  \left( \slashed{\mathrm{tr}} \underline{\hat{\chi}} + \frac{2}{v_0} \right) \right\|_{L^2_u L^2 (S_{u , v_0 } )} \right)  \leq \epsilon_3 , $$
\begin{align*}
 \mathcal{I}_4 \doteq \sum_{i=0}^3  \Bigg(  \| v^{3/2} ( v \slashed{\nabla} )^i \beta \|_{L^2_v L^2 (S_{u_0 , v} ) } + &  \sum_{\Psi \in \{ K - \frac{1}{v^2} , \check{\sigma} \}} \| v_0^{3/2} ( v_0 \slashed{\nabla} )^i \Psi_1 \|_{L^2_u L^2 (S_{u , v_0 } )}  \\ + & \| \mathrm{w} (u) v_0 ( v_0 \slashed{\nabla} )^i \underline{\beta} \|_{L^2_u L^2 (S_{u, v_0} )}   + \sum_{\Psi \in \{ K - \frac{1}{v^2} , \check{\sigma} \}} \| \mathrm{w} (u_0 ) v ( v \slashed{\nabla} )^i \Psi_2 \|_{L^2_v L^2 (S_{u_0 , v } ) } \Bigg) \leq \epsilon_4 , 
 \end{align*}
for constants $c_1$, $c_2$, $c_3$, $c_4$, $c_5$ and $\epsilon_1 , \epsilon_2 , \epsilon_3 , \epsilon_4 > 0$. Note that in the norms above when $u$ is fixed, $v \in [v_0 , \infty)$ in the $L^{\infty}_v$ norms and the interval of integration is $[v_0 , \infty)$ for the $L^p_v$ norms for $p < \infty$, while  when $v$ is fixed, $u \in [ u_0 , U ]$ for the $L^{\infty}_u$ norms and the interval of integration is $[u_0 , U]$ for the $L^p_u$ norms for $p < \infty$.

Then we have that:
$$ \mathcal{O}_1 \doteq \sum_{i=0}^3 \sum_{\psi_0 \in \{ \eta , \underline{\eta} \}} \Big( \| v ( v \slashed{\nabla} )^i \psi_0 \|_{L^{\infty}_u L^{\infty}_v L^2 (S_{u ,v } )} +   \| ( v \slashed{\nabla} )^4 \psi_0 \|_{L^2_v L^{\infty}_u L^2 (S_{u , v } )} + \| v ( v \slashed{\nabla} )^4 \psi_0 \|_{L^2_u L^{\infty}_v L^2 (S_{u , v } )}  \Big) \leq C  , $$
\begin{align*}
\mathcal{O}_2 \doteq \sum_{i=0}^4 \Bigg[ \sum_{\psi_1 \in \{ \hat{\chi} , \slashed{\mathrm{tr}} \chi - \frac{2}{v} \} }   &\Bigg( \left\| \frac{1}{\sqrt{v}} ( v \slashed{\nabla} )^i \psi_1 \right\|_{L^2_v L^{\infty}_u  L^2 (S_{u,v} )} + \left\| \frac{1}{v} ( v \slashed{\nabla} )^i \psi_1 \right\|_{L^1_v L^{\infty}_u L^{\infty} (S_{u , v} ) } \Bigg) \\ & + \| v^{1/2} ( v \slashed{\nabla} )^i \omega \|_{L^2_v  L^{\infty}_u L^2 (S_{u ,v} ) } + \| (v \slashed{\nabla} )^i \omega \|_{L^1_v L^{\infty}_u L^2 (S_{u , v} )} \Bigg]   \leq C , 
\end{align*}
$$ \mathcal{O}_3 \doteq \sum_{i=0}^4 \left( \| \mathrm{w} (u ) (v \slashed{\nabla} )^i \underline{\hat{\chi}} \|_{L^2_u L^{\infty}_v L^2 (S_{u , v } )} + \left\| \mathrm{w} (u )  v ( v \slashed{\nabla} )^i  \left( \slashed{\mathrm{tr}} \underline{\hat{\chi}} + \frac{2}{v} \right) \right\|_{L^2_u L^{\infty}_v L^2 (S_{u , v } )} \right)  \leq C , $$
\begin{align*}
 \mathcal{O}_4 \doteq \sum_{i=0}^3  \Bigg(  \| v^{3/2} ( v \slashed{\nabla} )^i \beta \|_{L^{\infty}_u L^2_v L^2 (S_{u , v} ) } + & \sum_{\Psi \in \{ K - \frac{1}{v^2} , \check{\sigma} \}} \| v^{3/2} ( v \slashed{\nabla} )^i \Psi_1 \|_{L^{\infty}_v L^2_u L^2 (S_{u , v } )}   \\ + & \| \mathrm{w} (u) v ( v \slashed{\nabla} )^i \underline{\beta} \|_{L^{\infty}_v L^2_u L^2 (S_{u, v} )}   + \sum_{\Psi \in \{ K - \frac{1}{v^2} , \check{\sigma} \}} \| \mathrm{w} (u ) v ( v \slashed{\nabla} )^i \Psi_2 \|_{L^{\infty}_u L^2_v L^2 (S_{u , v } ) } \Bigg) \leq C ,
 \end{align*}
where $C \doteq C ( c_1 , c_2 , c_3 , c_4 , c_5 , \epsilon_1 , \epsilon_2 , \epsilon_3 , \epsilon_4 )$, where $\mathrm{w} : [ u_0 , U ) \rightarrow \mathbb{R}_{\geq 0}$ is a smooth non-negative decreasing function for which we have that
$$ \int_{u_0}^U \frac{1}{\mathrm{w}^2 (u)} \, du \leq \epsilon . $$ 
Note that in the above norms $v \in [v_0 , \infty)$ in the $L^{\infty}_v$ norms and the interval of integration is $[v_0 , \infty)$ for the $L^p_v$ norms for $p < \infty$, while  and $u \in [ u_0 , U ]$ for the $L^{\infty}_u$ norms and the interval of integration is $[u_0 , U]$ for the $L^p_u$ norms for $p < \infty$.
\end{theorem}

We can also obtain another semi-classical construction by assuming stronger $L^{\infty} (S)$ bounds on $\hat{\chi}$, while relaxing the decay assumptions on $\eta$ and $\underline{\eta}$. On the other hand though we have to assume that the potential singularity on the the initial $\underline{H}_{v_0}$ hypersurface is weaker than the one that can be imposed with the data of Theorem \ref{thm:main}. For Theorem \ref{thm:main} we can have a weak null singularity on $H_U$, but for Theorem \ref{thm:main2} we can only have an impulsive gravitational wave propagating along some $H_{u_c}$ for some $u_0 < u_c < U$.

\begin{theorem}\label{thm:main2}
Let $U \in ( u_0 , \infty )$ and $V = \infty$, and fix some $\delta \in \left( \frac{1}{3} ,1 \right)$. For given initial data on $H_{u_0}$ for $v_0 \leq v < \infty$ and on $\underline{H}_{\underline{u}_0}$ for $u_0 \leq u \leq U$ where $U - u_0 \leq \epsilon$ and $\epsilon$ is small enough, we assume that
$$ c_1 \leq \left. \det \gamma \right|_{S_{u , v_0}} \leq c_2 , $$
$$ \left. \det \gamma \right|_{S_{u_0 , v}} \geq c_3 v^4 , $$
$$  \sum_{k \leq 3} \left( \left| \left. \left( \frac{\partial}{\partial \theta} \right)^k \gamma \right|_{S_{u , v_0}} \right| \leq c_4 , \quad \left| \left. \left( \frac{\partial}{\partial \theta} \right)^k \gamma \right|_{S_{u_0 , v}} \right| \right) \leq c_5 v^2 , $$
\begin{align*}
 \mathcal{I}_1^2 \doteq \sum_{i=0}^3 \sum_{\psi_0 \in \{ \eta , \underline{\eta} \}} \Big( \| v_0^{\delta } ( v_0 \slashed{\nabla} )^i \psi_0 \|_{L^{\infty}_u L^2 (S_{u ,v_0 } )} + & \| v^{\delta } ( v \slashed{\nabla} )^i \psi_0 \|_{L^{\infty}_v L^2 (S_{u_0 ,v } )} \\ + &  \left\| \frac{1}{v^{1/2}} ( v \slashed{\nabla} )^4 \psi_0 \right\|_{L^2_v L^2 (S_{u_0 , v } )} + \| v_0^{\delta } ( v_0 \slashed{\nabla} )^4 \psi_0 \|_{L^2_u L^2 (S_{u , v_0 } )}  \Big) \leq \epsilon_1 , 
 \end{align*}
$$ \mathcal{I}_2^2 \doteq \sum_{i=0}^4 \Bigg( \sum_{\psi_1 \in \{ \hat{\chi} , \slashed{\mathrm{tr}} \chi - \frac{2}{v} \} }   \left\| v^{\delta} ( v \slashed{\nabla} )^i \psi_1 \right\|_{L^{\infty}_v  L^2 (S_{u_0,v} )}  + \| v^{1+\delta} ( v \slashed{\nabla} )^i \omega \|_{L^{\infty}_v  L^2 (S_{u_0 ,v} ) } \Bigg) \leq \epsilon_2 , $$
$$ \mathcal{I}_3^2 \doteq \sum_{i=0}^4 \left( \left\| \widetilde{\mathrm{w}} (u ) \frac{1}{v_0^{1-\delta}} (v_0 \slashed{\nabla} )^i \underline{\hat{\chi}} \right\|_{L^2_u L^2 (S_{u , v_0 } )} +  \left\| \widetilde{\mathrm{w}} (u )  v_0^{\delta} ( v_0 \slashed{\nabla} )^i  \left( \slashed{\mathrm{tr}} \underline{\hat{\chi}} + \frac{2}{v_0} \right) \right\|_{L^2_u L^2 (S_{u , v_0 } )} \right)  \leq \epsilon_3 , $$
\begin{align*}
 \mathcal{I}_4^2 \doteq \sum_{i=0}^3  \Bigg( \| v^{3/2} ( v \slashed{\nabla} )^i \alpha \|_{L^2_v L^2 (S_{u_0 , v} ) }  + & \| v^{3/2} ( v \slashed{\nabla} )^i \beta \|_{L^2_v L^2 (S_{u_0 , v} ) } \\ + & \| v_0^{3/2} ( v_0 \slashed{\nabla} )^i \beta \|_{L^2_u L^2 (S_{u , v_0 } )} +\sum_{\Psi \in \{ \check{\rho} , \check{\sigma} \}} \| v_0^{3/2} ( v_0 \slashed{\nabla} )^i \Psi_1 \|_{L^2_u L^2 (S_{u , v_0 } )}  \\ & + \| \widetilde{\mathrm{w}} (u) v_0^{1/2} ( v_0 \slashed{\nabla} )^i \underline{\beta} \|_{L^2_u L^2 (S_{u, v_0} )}   + \sum_{\Psi \in \{ \check{\rho} , \check{\sigma} \}} \| \widetilde{\mathrm{w}} (u_0 ) v^{1/2} ( v \slashed{\nabla} )^i \Psi_2 \|_{L^2_v L^2 (S_{u_0 , v } ) } \Bigg) \leq \epsilon_4 , 
 \end{align*}
for constants $c_1$, $c_2$, $c_3$, $c_4$, $c_5$ and $\epsilon_1 , \epsilon_2 , \epsilon_3 , \epsilon_4 > 0$, where $\widetilde{\mathrm{w}} : [ u_0 , U ) \rightarrow \mathbb{R}_{\geq 0}$ is a smooth non-negative decreasing and \textbf{bounded} function for which we have that
$$ \int_{u_0}^U \frac{1}{\widetilde{\mathrm{w}}^2 (u)} \, du \leq \epsilon . $$
Note that in the norms above when $u$ is fixed, $v \in [v_0 , \infty)$ in the $L^{\infty}_v$ norms and the interval of integration is $[v_0 , \infty)$ for the $L^p_v$ norms for $p < \infty$, while  when $v$ is fixed, $u \in [ u_0 , U ]$ for the $L^{\infty}_u$ norms and the interval of integration is $[u_0 , U]$ for the $L^p_u$ norms for $p < \infty$.

Then we have that:
$$ \mathcal{O}_1^2 \doteq \sum_{i=0}^3 \sum_{\psi_0 \in \{ \eta , \underline{\eta} \}} \Big( \| v^{\delta} ( v \slashed{\nabla} )^i \psi_0 \|_{L^{\infty}_{u,v} L^2 (S_{u ,v } )}  +   \left\| \frac{1}{v^{1/2}} ( v \slashed{\nabla} )^4 \psi_0 \right\|_{L^2_v L^{\infty}_u L^2 (S_{u , v } )} + \| v^{\delta} ( v \slashed{\nabla} )^4 \psi_0 \|_{L^2_u L^{\infty}_v L^2 (S_{u , v } )}  \Big) \leq C  , $$
$$ \mathcal{O}_2^2 \doteq \sum_{i=0}^4 \Bigg( \sum_{\psi_1 \in \{ \hat{\chi} , \slashed{\mathrm{tr}} \chi - \frac{2}{v} \} }    \left\| v^{\delta} ( v \slashed{\nabla} )^i \psi_1 \right\|_{L^{\infty}_{u,v}   L^2 (S_{u,v} )}  + \| v^{1 + \delta} ( v \slashed{\nabla} )^i \omega \|_{  L^{\infty}_{u,v} L^2 (S_{u ,v} ) }  \Bigg)   \leq C , $$
$$ \mathcal{O}_3^2 \doteq \sum_{i=0}^4 \left( \left\| \widetilde{\mathrm{w}} (u ) \frac{1}{v^{1-\delta}} (v \slashed{\nabla} )^i \underline{\hat{\chi}} \right\|_{L^2_u L^{\infty}_v L^2 (S_{u , v } )} + \left\| \widetilde{\mathrm{w}} (u )  v^{\delta} ( v \slashed{\nabla} )^i  \left( \slashed{\mathrm{tr}} \underline{\hat{\chi}} + \frac{2}{v} \right) \right\|_{L^2_u L^{\infty}_v L^2 (S_{u , v } )} \right)  \leq C , $$
\begin{align*}
 \mathcal{O}_4^2 \doteq \sum_{i=0}^3  \Bigg(  \| v^{3/2} ( v \slashed{\nabla} )^i \alpha \|_{L^{\infty}_u L^2_v L^2 (S_{u , v} )}+&  \| v^{3/2} ( v \slashed{\nabla} )^i \beta \|_{L^{\infty}_u L^2_v L^2 (S_{u , v} ) } \\ + &  \| v^{3/2} ( v \slashed{\nabla} )^i \beta \|_{L^{\infty}_v L^2_u L^2 (S_{u , v } )} + \sum_{\Psi \in \{ \check{\rho} , \check{\sigma} \}} \| v^{3/2} ( v \slashed{\nabla} )^i \Psi_1 \|_{L^{\infty}_v L^2_u L^2 (S_{u , v } )}    \\ & + \| \widetilde{\mathrm{w}} (u) v^{1/2} ( v \slashed{\nabla} )^i \underline{\beta} \|_{L^{\infty}_v L^2_u L^2 (S_{u, v} )}   + \sum_{\Psi \in \{ \check{\rho} , \check{\sigma} \}} \| \widetilde{\mathrm{w}} (u ) v^{1/2} ( v \slashed{\nabla} )^i \Psi_2 \|_{L^{\infty}_u L^2_v L^2 (S_{u , v } ) } \Bigg) \leq C ,
 \end{align*}
where $C \doteq C ( c_1 , c_2 , c_3 , c_4 , c_5 , \epsilon_1 , \epsilon_2 , \epsilon_3 , \epsilon_4 )$. Note that in the above norms $v \in [v_0 , \infty)$ in the $L^{\infty}_v$ norms and the interval of integration is $[v_0 , \infty)$ for the $L^p_v$ norms for $p < \infty$, while  and $u \in [ u_0 , U ]$ for the $L^{\infty}_u$ norms and the interval of integration is $[u_0 , U]$ for the $L^p_u$ norms for $p < \infty$.
\end{theorem}

The main challenge in both Theorems is to identify the proper decay in $v$ for all quantities, and try to close all estimates for the different decay rates. 

In the first Theorem \ref{thm:main}, we impose only integrable decay in $v$ for $\hat{\chi}$ (which can be seen as being $1-\delta'$, for any $\delta' > 0$, worse than the expected optimal decay) in the sense of a scaling invariant norm, and we do something similar for $\omega$, while for the rest of the Christoffel symbols we impose the expected optimal decay in $v$ -- even for the Christoffel symbols initially determined on $\underline{H}_{v_0}$ where the decay in $v$ is measured in norms weighted in $u$. Due to the fact that we want to cover the case of a weak null singularity we are forced to work with $\beta$, $\underline{\beta}$, $K$ and $\check{\sigma}$ (so not only avoid $\underline{\alpha}$ but $\check{\rho}$ as well), on which we impose suboptimal (compared to the one expected for asymptotically flat spacetimes) decay in $v$. Moreover, we use the $\snabla_3$ equation for $\hat{\chi}$ (that is \eqref{eq:chi3}) which we integrate in $v$ with appropriate weights to close the estimates in the scaling invariant norms (as an aside let us mention that we do not need to assume the same for $\slashed{\mathrm{tr}} \chi$ as it turns out to be essentially better behaved than $\hat{\chi}$, a fine point has to do with closing the estimates at the highest derivative level where the elliptic structure is used simultaneously for both $\hat{\chi}$ and $\slashed{\mathrm{tr}} \chi$, and not separately as in previous works). Lastly we note that we need the expected optimal decay in $v$ for the rest of the Christoffel symbols due to the use of $K$ in the energy estimates for the curvature components.

In the second Theorem \ref{thm:main2} we allow suboptimal decay for all Christoffel symbols, but because of the last observation of the previous paragraph (i.e. the use of $K$), we cannot use the same machinery as before. Instead we need to use $\alpha$ and $\check{\rho}$ on $H_u$ hypersurfaces instead of $K$. This is also the reason that we have to impose pointwise (in $L^2 (S)$ norms) decay in $v$ for $\hat{\chi}$ and we cannot only assume finiteness in the scaling invariant norms that were used in Theorem \ref{thm:main} (we note though that it might be possible to use only the scaling invariant norms for $\hat{\chi}$ by using $\mu$ and $\underline{\mu}$ -- see the proof of Theorem \ref{thm:main} for their definitions --  instead of $K$).

\begin{remark}
With appropriate modifications, Theorems \ref{thm:main} and \ref{thm:main2} should also hold for $S_{u,v}$ being smooth surfaces of finite area. This variation of Theorem \ref{thm:main2} will be used in upcoming work on the construction of global future geodesically complete spacetimes containing the interaction of two impulsive gravitational waves.
\end{remark}

\begin{remark}
As mentioned earlier if we allow $\hat{\chi}$ to be large measured in the norms from Theorem \ref{thm:main}, this should lead to trapped surface formation, i.e. if for some $V < \infty$ we have that
$$ \int_{v_0}^{V} \frac{1}{v} \| \hat{\chi} \|_{L^2 (S_{u_0 , v })} \, dv + \int_{v_0}^{V} \frac{1}{v} \| \hat{\chi} \|_{L^2 (S_{u_0 , v })}^2 \, dv \geq \widetilde{C} , $$
for some $\widetilde{C}$ large enough, then we expect $S_{U,V}$ to be trapped.
\end{remark}

\begin{remark}
Let us also note that using a construction similar to the one of Li and Yu \cite{liyu} and using Theorem \ref{thm:main2} it should be possible to construct a global future geodesically complete spacetime containing an outgoing impulsive gravitational wave from Cauchy data.
\end{remark}

\section{Characteristic initial data}\label{cid}
The goal of this section is to construct initial data for Theorem \ref{thm:main2}, the data for Theorem \ref{thm:main} follow in a similar way. The main issue is for $\hat{\chi}$ to satisfy
$$ | \hat{\chi} | ( u_0 , v ) \sim \frac{1}{v^{1+\epsilon}} \mbox{  on $H_{u_0}$ for $v \in [v_0 , \infty)$} , $$
for some $\epsilon > 0$, and for $\hat{\underline{\chi}}$ to satsify
$$ | \hat{\underline{\chi}} | ( u , v_0 ) \sim \frac{1}{\mathrm{w}^2 (u)} \mbox{  on $\underline{H}_{v_0}$ for $u \in [ u_0 , U ]$.} $$ 

As it is well known (see \cite{DC09}) we need to specify on $S_{u_0 , \underline{u}_0}$ $\gamma_{AB}$, $\zeta_A$, $ \slashed{\mathrm{tr}} \chi$ and $\slashed{\mathrm{tr}} \underline{\chi}$, and the conformal factor $\Phi$ for
$$ \gamma_{AB} \doteq \Phi^2 \widetilde{\gamma}_{AB} , $$
on $H_{u_0}$ and $\underline{H}_{v_0}$, where $\widetilde{\gamma}$ is the standard round metric of radius $v_0$ on $S_{u_0 , v_0}$.

For the data on $\underline{H}_{v_0}$ the construction is given in \cite{weaknull} in the relevant section and we will not repeat it here.

On $H_{u_0}$ we set
$$ \left. \widetilde{\gamma}_{AB} \right|_{H_{u_0}} = \frac{\mathrm{m}_{AB}}{\left( 1+ \frac{1}{4} | \vartheta |^2 \right)^2} , $$
for $\mathrm{m}$ a $2\times2$ matrix with determinant equal to 1, which we can ensure by writing $\mathrm{m} \doteq e^{\widetilde{\Phi}}$ for $\widetilde{\Phi}$ a $2\times2$ matrix with eigenvalues that add to 0. On $H_{u_0}$ we require $\widetilde{\Phi}$ to satisfy the following bounds:
$$ \sum_{i \leq I} \left| \frac{\partial^i}{\partial \vartheta^i} \frac{\partial \widetilde{\Phi}}{\partial v} \right| (u_0 , v)  \lesssim \frac{1}{v^{1+\epsilon'}} , \quad \sum_{i \leq I} \left|  \frac{\partial^i }{\partial \vartheta^i} \widetilde{\Phi}\right| (u_0 , v) \lesssim 1 \mbox{  for $v \in [ v_0 , \infty)$,} $$
for $I$ big enough and for some $\epsilon' > 0$. More precisely we also require
$$ \left| \frac{\partial \widetilde{\Phi}}{\partial v} \right| (u_0 , v) \sim \frac{1}{v^{1+\epsilon'}} \mbox{  for $v \in [ v_0 , \infty)$,} $$
which implies that
$$ | \hat{\chi} |^2_{\gamma} = \frac{1}{4} \widetilde{\gamma}^{AC} \widetilde{\gamma}^{BD} \frac{\partial \widetilde{\gamma}_{AB}}{\partial v} \frac{\partial \widetilde{\gamma}_{CD}}{\partial v}  \sim \frac{1}{v^{2+2\epsilon}} \mbox{  for $v \in [ v_0 , \infty)$.} $$

On $H_{u_0}$ we further set $\Omega =  1$ and $b^A = 0$, $A \in \{1,2\}$. Then this implies that
$$ \slashed{\mathcal{L}}_{\partial / \partial_v}  \gamma_{AB} = 2 \chi_{AB} \mbox{  and  } \slashed{\mathcal{L}}_{\partial / \partial_v} \slashed{\mathrm{tr}} \chi = - \frac{1}{2} ( \slashed{\mathrm{tr}} \chi )^2 - | \hat{\chi} |^2 \mbox{  on $H_{u_0}$.} $$
The second equation implies the following equation for $\Phi$:
\begin{equation}\label{eq:Phic}
 \frac{\partial^2 \Phi}{\partial v^2} + \frac{1}{8} \widetilde{\gamma}^{AC} \widetilde{\gamma}^{BD} \frac{\partial \widetilde{\gamma}_{AB}}{\partial v} \frac{\partial \widetilde{\gamma}_{CD}}{\partial v} \cdot \Phi = 0 ,
 \end{equation}
with initial data
$$ \left. \Phi \right|_{S_{u_0, v_0}} = 1 , \quad \left. \frac{\partial \Phi}{\partial v} \right|_{S_{u_0,v_0}} = \frac{1}{2} \slashed{\mathrm{tr}} \chi = \frac{1}{v_0} . $$
Finally we set data for $\zeta$:
$$ \sum_{i \leq I-1} \left. \left| \frac{\partial^i \zeta}{\partial \vartheta^i} \right|^2_{\gamma} \right|_{S_{u_0 , v_0}} \lesssim \frac{1}{v_0^4} , $$
and we note that $\hat{\chi}$ and $\slashed{\mathrm{tr}} \chi$ can be expressed in terms of $\Phi$ as:
$$ \hat{\chi}_{AB} = \frac{1}{2} \Phi^2 \frac{\partial \widetilde{\gamma}_{AB}}{\partial v} , \quad \slashed{\mathrm{tr}} \chi = \frac{2}{\Phi} \frac{\partial \Phi}{\partial v} . $$

Using equation \eqref{eq:Phic} and the data on $S_{u_0 , v_0}$ for $\widetilde{\Phi}$ we get that:
$$ \Phi (u_0 , v) = \frac{v}{v_0} + O ( v^{1-\epsilon'} ) , $$
$$ \frac{\partial \Phi}{\partial v} = \frac{1}{v_0} + O (v_0^{-1-\epsilon'} ) , $$
which implies that
$$ \sum_{i=0}^I |( v \slashed{\nabla} )^i \hat{\chi} | ( u_0 , v) \lesssim \frac{1}{v^{1+\epsilon}} , $$
$$ \sum_{i=0}^I \left| ( v \slashed{\nabla} )^i \left( \slashed{\mathrm{tr}} \chi ( u_0 ,v) - \frac{2}{v} \right) \right| \lesssim \frac{1}{v^{1+\epsilon}} , $$
for all $v \in [v_0 , \infty )$, from the results of Chapter 2 of \cite{DC09} due to the assumptions on $\widetilde{\Phi}$. Moreover from the same results and equation 
$$ \mathcal{L}_{\partial / \partial v} \zeta + \slashed{\mathrm{tr}} \chi \zeta = \slashed{\mathrm{div}} \chi - \slashed{\nabla} \slashed{\mathrm{tr}} \chi , $$
which holds on $H_{u_0}$ according to our assumptions, we get that
$$ \sum_{i=0}^{I-1} \left| ( v \slashed{\nabla} )^i \zeta \right| \lesssim \frac{1}{v^{1+\epsilon'}} , $$
while from the implied estimate for $\gamma$ and $\gamma^{-1}$ we get the following estimate for the Gaussian curvature:
$$ \sum_{i \leq I-2} \left| ( v \slashed{\nabla} )^i K \right| \lesssim \frac{1}{v^{1+\epsilon'}} , $$
for all $v \in [ v_0 , \infty)$. Finally we can also get the following bounds for $\slashed{\mathrm{tr}} \underline{\chi}$:
$$ \sum_{i \leq I-2} \left| ( v \slashed{\nabla} )^i \left( \slashed{\mathrm{tr}} \underline{\chi} + \frac{2}{v} \right) \right|  \lesssim \frac{1}{v^{1+\epsilon'}} , $$
for all $v \in [ v_0 , \infty)$ by making use of the equation
$$ \mathcal{L}_{\partial / \partial v} \slashed{\mathrm{tr}} \underline{\chi} + \slashed{\mathrm{tr}} \chi \slashed{\mathrm{tr}} \underline{\chi} = -2 K -2 \slashed{\mathrm{div}} \zeta + 2 | \zeta |^2 , $$
which holds on $H_{u_0}$ according to our assumptions.

\section{Basic inequalities}\label{bi}
We will present some basic inequalities that we will use throughout the proofs of Theorems \ref{thm:main} and \ref{thm:main2}. For all of them we make use of the assumptions of Theorems \ref{thm:main} and \ref{thm:main2}. We start with a basic transport estimate in the $v$ direction.
\begin{proposition}\label{est:basic1}
Assume that the bootstrap assumptions \eqref{bootstraps} or \eqref{bootstraps2} hold true for all $\epsilon < \sigma$ for some $\sigma > 0$ small enough. Then we have that for $2 \leq p \leq \infty$ and any tensor field $F$ of arbitrary rank tangential to the spheres $S_{u,v}$ that satisfies
\begin{equation}\label{eq:gen}
\slashed{\nabla}_4 F + \lambda_0 \slashed{\mathrm{tr}} \chi F = G ,
\end{equation}
 the following hold for any $0 \leq v_0 \leq v$, and any $u \geq u_0 \geq 0$:
\begin{equation}\label{est:b4p}
v^{\lambda_1} \| F \|_{L^p (S_{u,v} )}  \lesssim \| F \|_{L^p (S_{u,v_0} )} + \int_{v_0}^{v} (v' )^{\lambda_1} \| G \|_{L^p (S_{u,v'} )} \, dv' ,  
\end{equation}
for $0 \leq \lambda_1 \leq 2 \left( \lambda_0 - \frac{1}{p} \right)$. 
\end{proposition}
\begin{proof}
We differentiate the quantity we want to bound in $v$ and we have that:
\begin{align*}
\frac{d}{dv} \int_{S_{u,v}} v^{\lambda_1 p} | F |^p = & \int_{S_{u,v}} \left( \lambda_1 p v^{\lambda_1 p - 1} | F |^p + p v^{\lambda_1} | F |^{p-2} \langle F , e_4 (F ) \rangle + \slashed{\mathrm{tr}} \chi v^{\lambda_1 p} | F |^p \right) \\ = & \int_{S_{u,v}} \left( \lambda_1 p v^{\lambda_1 p - 1} | F |^p + ( 1 - p \lambda_0 ) \slashed{\mathrm{tr}} \chi v^{\lambda_1 p} | F |^p  \right) + \int_{S_{u,v}} p v^{\lambda_1 p} | F |^{p-2} \langle F , G \rangle , 
\end{align*}
and this implies that
\begin{align*}
\int_{S_{u,v_2}} v_2^{\lambda_1 p} | F |^p \lesssim & \int_{S_{u,v_1}} v_1^{\lambda_1 p} | F |^p + \int_{v_1}^{v_2} \int_{S_{u,v}} \left| \slashed{\mathrm{tr}} \chi - \frac{2}{v} \right| | F |^p \\ & + \int_{v_1}^{v_2} \int_{S_{u,v}} \frac{C(p, \lambda_0 , \lambda_1 )}{v} v^{\lambda_1 p} | F |^p + \int_{v_1}^{v_2}  \int_{S_{u,v}} p v^{\lambda_1 p} | F |^{p-2} \langle F , G \rangle ,
\end{align*}
and the desired estimate follows since $C(p, \lambda_0 , \lambda_1 ) \leq 0$ if $\lambda_1 \leq 2 \left( \lambda_0 - \frac{1}{p} \right)$ and by applying H\"{o}lder's inequality to the last term.
\end{proof}
Next we state a transport estimate in the $u$ direction. We omit the proof as it is just an application of the fundamental theorem of calculus.
\begin{proposition}\label{est:basic2}
Assume that the bootstrap assumptions \eqref{bootstraps} or \eqref{bootstraps2} hold true for all $\epsilon < \sigma$ for some $\sigma > 0$ small enough. Then for every $2 \leq p \leq \infty$ we have that
\begin{equation}\label{est:b3p}
\| f \|_{L^p (S_{u_2 , v })} \lesssim \| f \|_{L^p (S_{u_1 , v })} + \int_{u_1}^{u_2} \| \slashed{\nabla}_3 f \|_{L^p (S_{u' , v } )} \, du' ,
\end{equation}
for all $u_0 \leq u_1 \leq u_2 \leq \infty$, $v \in [ v_0 , \infty)$. 
\end{proposition}
Moreover we will need the following Sobolev inequalities, the proofs of them can be found in \cite{iwaves2}.
\begin{proposition}\label{sobolev}
Assume that the bootstrap assumptions \eqref{bootstraps} or \eqref{bootstraps2} hold true. Then we have for all $\epsilon \leq \delta$ where $\delta$ is sufficiently small, that the following inequalities hold for all $u_0 \leq u \leq U$ and all $v_0 \leq v < \infty$:
\begin{equation}\label{sobolev1}
\| f \|_{L^4 (S_{u,v})} \leq C \frac{1}{\sqrt{v}} \sum_{k=0}^1 \| ( v \slashed{\nabla} )^k f \|_{L^2 (S_{u,v} )} ,
\end{equation}
\begin{equation}\label{sobolev2}
\| f \|_{L^{\infty} (S_{u,v})} \leq C \left( \| f \|_{L^2 (S_{u,v})} +  \frac{1}{\sqrt{v}} \| ( v \slashed{\nabla} ) f \|_{L^4 (S_{u,v} )} \right) ,
\end{equation}
\begin{equation}\label{sobolev3}
\| f \|_{L^{\infty} (S_{u,v})} \leq C \sum_{k=0}^2  \frac{1}{v} \| ( v\slashed{\nabla} )^k f \|_{L^2 (S_{u,v} )} ,
\end{equation}
for some constant $C$.
\end{proposition}
We will also use the following two Lemmas that contain estimates for div-curl systems (see \cite{jonathanan} for a proof).
\begin{proposition}\label{ellipticcurl}
We consider an 1-form $f$ on $S_{u , v}$ with metric $\gamma$ for some $u$, $v$, that satisfies the following equations:
$$ \slashed{\mathrm{div}} f = F, \quad \slashed{\mathrm{curl}} f = G . $$
Assuming that
$$ \sum_{i=0}^1 \| v  ( v \slashed{\nabla} )^i K \|_{L^2 (S_{u , v} )} < \infty , $$
we then have that:
\begin{equation}\label{est:ellipticcurl}
\| ( v \slashed{\nabla} )^k f \|_{L^2 (S_{u , v} )} \lesssim  \left( \sum_{i=0}^1 \| v ( v \slashed{\nabla} )^i K \|_{L^2 (S_{u , v} )} \right) \left[ \sum_{m=0}^{k-1} \left(  \| v ( v \slashed{\nabla} )^m F \|_{L^2 (S_{u , v} )}  + \| v ( \underline{u} \slashed{\nabla} )^m G \|_{L^2 (S_{u , v} )} \right)  + \| f \|_{L^2 (S_{u ,v} )} \right] .
\end{equation}
\end{proposition}

\begin{proposition}\label{elliptic}
For a symmetric traceless 2-tensor $f$ on $S_{u , v}$ with metric $\gamma$ for some $u$ and $v$ satisfying the equation
$$ \slashed{\mathrm{div}} f  = F $$
and assuming that
$$ \sum_{i=0}^1 \| v ( v \slashed{\nabla} )^i K \|_{L^2 (S_{u , v} )} < \infty , $$
 we then have that:
\begin{equation}\label{est:elliptic}
\|( v \slashed{\nabla} )^k f \|_{L^2 (S_{u , v } )} \lesssim \left( \sum_{i=0}^1 \| v ( v \slashed{\nabla} )^i K \|_{L^2 (S_{u , v} )} \right) ( \sum_{m=0}^{k-1} \| v ( v \slashed{\nabla} )^m F \|_{L^2 (S_{u , v} )} + \| f \|_{L^2 (S_{u , v} )} ) .
\end{equation} 

\end{proposition}
The following two Propositions contain two basic computations through integration by parts, see \cite{iwaves2} for a proof.
\begin{proposition}\label{intbyparts}
Let $f_1$ and $f_2$ be two $m$-tensorfields. We have for any $u_1 < u_2$ and any $v_1 < v_2$ that
\begin{equation}\label{ibp1}
\int_{D_{u,v}} ( f_1 \slashed{\nabla}_4 f_2 + f_2 \slashed{\nabla}_4 f_1 ) = \int_{\underline{H}_{v_2} (u_1 , u_2)} f_1 f_2 - \int_{\underline{H}_{v_1} (u_1 , u_2)} f_1 f_2 + \int_{D_{u,v}} \left( 2\omega - \slashed{\mathrm{tr}} \chi \right) f_1 f_2 ,
\end{equation}
and that
\begin{equation}\label{ibp2}
\int_{D_{u,v}} ( f_1 \slashed{\nabla}_3 f_2 + f_2 \slashed{\nabla}_3 f_1 ) = \int_{H_{u_2} (v_1 , v_2 )} f_1 f_2 - \int_{H_{u_1} (v_1 , v_2 )} f_1 f_2 - \int_{D_{u,v}}  \slashed{\mathrm{tr}} \underline{\chi} f_1 f_2 .
\end{equation}
\end{proposition}
\begin{proposition}\label{intbyparts2}
Let $f$ be an $m$-tensorfield and $g$ be an $(m-1)$-tensorfield. We have that
\begin{equation}\label{ibp3}
\int_{D_{u,v}} ( f^{\alpha_1 \cdots \alpha_m} \slashed{\nabla}_{\alpha_m} g_{\alpha_1 \cdots \alpha_{m-1}} + \slashed{\nabla}^{\alpha_m} f_{\alpha_1 \cdots \alpha_m} g^{\alpha_1 \cdots \alpha_{m-1}} ) = - \int_{D_{u,v}} (\eta +\underline{\eta} ) f g .
\end{equation}
\end{proposition}
A consequence of estimate \eqref{ibp1} from Proposition \ref{intbyparts} is the following.
\begin{proposition}\label{intbyparts3}
Assume that the bootstrap assumptions \eqref{bootstraps} or \eqref{bootstraps2} hold true. Assume that $f$ satisfies
$$ \slashed{\nabla}_4 f + q \slashed{\mathrm{tr}} \chi f = F , $$
for some $q \geq \frac{1}{2}$. Then we have that 
\begin{equation}\label{ibp4}
\int_{\underline{H}_v} v^p | f |^2 \lesssim \int_{\underline{H}_{v_0}} v_0^p | f |^2 + \int_{D_{u,v}} v^p \left| \langle f , F \rangle \right| ,
\end{equation}
for all $v \in [ v_0 , \infty)$, and for 
$$ 0 \leq p \leq 4q-2 . $$
\end{proposition}
\begin{proof}
We make use of estimate \eqref{ibp1} and we have that
$$ \int_{D_{u,v}} v^p \langle f , \slashed{\nabla}_4 f \rangle = \frac{1}{2} \int_{\underline{H}_{v_2}} v_2^p | f |^2 - \frac{1}{2} \int_{\underline{H}_{v_1}} v_1^p | f |^2 + \int_{D_{u,v}} v^p \left( \omega - \frac{1}{2} \slashed{\mathrm{tr}} \chi - \frac{p}{2v} \right) | f |^2 , $$
which implies by making use of the assumption on $\slashed{\nabla}_4 f$ that
\begin{align*}
\frac{1}{2} \int_{\underline{H}_{v_2}} v_2^p | f |^2 = & \frac{1}{2}\int_{\underline{H}_{v_1}} v_1^p | f |^2 - \int_{D_{u,v}} v^p \left( \omega + \left( q - \frac{1}{2} \right) \slashed{\mathrm{tr}} \chi - \frac{p}{2v} \right) | f |^2 - \int_{D_{u,v}} v^p \langle f , F \rangle \\ = & \frac{1}{2}\int_{\underline{H}_{v_1}} v_1^p | f |^2 - \int_{D_{u,v}} v^p \left[ \omega + \left( q - \frac{1}{2} \right) \left( \slashed{\mathrm{tr}} \chi - \frac{2}{v} \right) + \left( q - \frac{1}{2} \right) \left( 1 - \frac{p}{4q-2} \right)  \frac{2}{v} \right] | f |^2 - \int_{D_{u,v}} v^p \langle f , F \rangle \\ \leq & \frac{1}{2}\int_{\underline{H}_{v_1}} v_1^p | f |^2 + C \int_{D_{u,v}} v^p \left( | \omega | + \left| \slashed{\mathrm{tr}} \chi - \frac{2}{v} \right| \right) | f |^2 + \int_{D_{u,v}} v^p \left| \langle f , F \rangle \right| ,
\end{align*}
for $C$ a positive constant depending on $p$ and $q$. Using the bootstrap assumptions from Theorem \ref{thm:main} we have that
\begin{align*}
C \int_{D_{u,v}} v^p \left( | \omega | + \left| \slashed{\mathrm{tr}} \chi - \frac{2}{v} \right| \right) | f |^2 \lesssim & \int_{D_{u,v}} \sum_{k=0}^2 \left( \| ( v \slashed{\nabla} )^k \omega \|_{L^1_v L^{\infty}_u L^2 (S_{u,v})} +  \left\| ( v \slashed{\nabla} )^k \left( \slashed{\mathrm{tr}} \chi - \frac{2}{v} \right) \right\|_{L^1_v L^{\infty}_u L^2 (S_{u,v})} \right) | f|^2 \\ \leq & \epsilon' \sup_{v \in [v_1 , v_2 ]} \int_{\underline{H}_v} | f |^2 , 
\end{align*}
for some $\epsilon' > 0$ small enough (this is the point where the bootstrap assumptions where used, also introducing integrable decay in $v$ -- note that here we used the bootstrap assumptions of Theorem \ref{thm:main} although the same estimates are implied by the bootstrap assumptions of Theorem \ref{thm:main2}), so that the last term can be absorbed by the left-hand side of the previous inequality after taking the supremum in $v$.
\end{proof}
We also have the following Proposition (see \cite{iwaves1} for a proof) for the commutators $[\slashed{\nabla}_3 , \slashed{\nabla}]$ and $[\slashed{\nabla}_4 , \slashed{\nabla}]$.
\begin{proposition}\label{commute}
Assume that
$$ \slashed{\nabla}_3 f = F_0 , \quad \slashed{\nabla}_3 ( \slashed{\nabla}^k f ) = F_k,  \quad \slashed{\nabla}_4 \widetilde{g} = G_0, \quad \slashed{\nabla}_4 (\slashed{\nabla}^k \widetilde{g} ) = G_k . $$
Then we have that
\begin{equation*}
\begin{split}
F_k + \frac{k}{2} \slashed{\mathrm{tr}} \underline{\chi} ( \slashed{\nabla}^k f ) \sim & \sum_{\psi \in \{ \eta , \underline{\eta} \}} \sum_{k_1 + k_2 + k_3 = k}  \slashed{\nabla}^{k_1} \psi^{k_2}  ( \slashed{\nabla}^{k_3} F_0 ) + \sum_{\psi \in \{ \eta , \underline{\eta} \} , \psi_{\underline{H}} \in \{ \underline{\hat{\chi}} , \slashed{\mathrm{tr}} \underline{\chi} \}} \sum_{k_1 + k_2 + k_3 + k_4 = k} (  \slashed{\nabla}^{k_1} \psi )^{k_2}  ( \slashed{\nabla}^{k_3} \psi_{\underline{H}} )  ( \slashed{\nabla}^{k_4} f ) \\ & + \sum_{\psi \in \{ \eta , \underline{\eta} \}} \sum_{k_1 + k_2 + k_3 + k_4 = k-1} ( \slashed{\nabla}^{k_1} \psi )^{k_2}  ( \slashed{\nabla}^{k_3} \underline{\beta} ) ( \slashed{\nabla}^{k_4} f ) ,
\end{split}
\end{equation*}
and
\begin{equation*}
\begin{split}
G_k + \frac{k}{2} \slashed{\mathrm{tr}} \chi ( \slashed{\nabla}^k \widetilde{g} ) \sim & \sum_{\psi \in \{ \eta , \underline{\eta} \}} \sum_{k_1 + k_2 + k_3 = k}  ( \slashed{\nabla}^{k_1} \psi )^{k_2}  ( \slashed{\nabla}^{k_3} G_0 ) + \sum_{\psi \in \{ \eta , \underline{\eta} \} , \psi_H \in \{ \hat{\chi} , \slashed{\mathrm{tr}} \chi - \frac{2}{v} \} } \sum_{k_1 + k_2 + k_3 + k_4 = k} (\slashed{\nabla}^{k_1} \psi )^{k_2}  ( \slashed{\nabla}^{k_3} \psi_H )   ( \slashed{\nabla}^{k_4} \widetilde{g} ) \\ & + \sum_{\psi \in \{ \eta , \underline{\eta} \}} \sum_{k_1 + k_2 + k_3 + k_4 = k-1} ( \slashed{\nabla}^{k_1}\psi )^{k_2}  ( \slashed{\nabla}^{k_3} \beta ) ( \slashed{\nabla}^{k_4} \widetilde{g} ) .
\end{split}
\end{equation*}
\end{proposition}
Note that in the previous Proposition we used the notation
$$ ( \slashed{\nabla}^{k_1} f )^{k_2} = \sum_{l_1 + \cdots + l_{k_2} = k_1 } ( \slashed{\nabla}^{l_1} \psi ) \cdots ( \slashed{\nabla}^{l_{k_2}} \psi ) , $$
and if $k_2 = 0$ then the above term is equal to 0. We also note that as $\beta$ and $\underline{\beta}$ can be written as:
$$ \beta = \snabla \chi + \psi \chi , \quad \underline{\beta} = \snabla \underline{\chi} + \psi \underline{\chi} , $$
for $\psi \in \{ \eta , \underline{\eta} \}$ then we also have the following for the quantities of Proposition \ref{commute}:
$$ F_k + \frac{k}{2} \slashed{\mathrm{tr}} \underline{\chi} ( \snabla^k f ) \sim \sum_{\psi \in \{ \eta , \underline{\eta} \}} \sum_{k_1 + k_2 + k_3 = k}  \slashed{\nabla}^{k_1} \psi^{k_2}  ( \slashed{\nabla}^{k_3} F_0 ) + \sum_{\psi \in \{ \eta , \underline{\eta} \} , \psi_{\underline{H}} \in \{ \underline{\hat{\chi}} , \slashed{\mathrm{tr}} \underline{\chi} \}} \sum_{k_1 + k_2 + k_3 + k_4 = k} (  \slashed{\nabla}^{k_1} \psi )^{k_2}  ( \slashed{\nabla}^{k_3} \psi_{\underline{H}} )  ( \slashed{\nabla}^{k_4} f ) , $$
and
$$ G_k + \frac{k}{2} \slashed{\mathrm{tr}} \chi ( \snabla^k \widetilde{g} ) \sim \sum_{\psi \in \{ \eta , \underline{\eta} \}} \sum_{k_1 + k_2 + k_3 = k}  ( \slashed{\nabla}^{k_1} \psi )^{k_2}  ( \slashed{\nabla}^{k_3} G_0 ) + \sum_{\psi \in \{ \eta , \underline{\eta} \} , \psi_H \in \{ \hat{\chi} , \slashed{\mathrm{tr}} \chi - \frac{2}{v} \} } \sum_{k_1 + k_2 + k_3 + k_4 = k} (\slashed{\nabla}^{k_1} \psi )^{k_2}  ( \slashed{\nabla}^{k_3} \psi_H )   ( \slashed{\nabla}^{k_4} \widetilde{g} ) . $$

\section{Proof of Theorem \ref{thm:main}}\label{p1}

The proof will involve a bootstrap argument. Assume that
\begin{equation}\label{bootstraps}
\sum_{i=1}^4 \mathcal{O}_i + v^2 \left\| ( v \slashed{\nabla} )^m \left( K - \frac{1}{v^2} \right) \right\|_{L^2 (S_{u ,v} )} \leq \widetilde{C} \sum_{i=1}^4 \epsilon_i ,
\end{equation}
for some constant $\widetilde{C}$. The goal will be to improve the constant $\widetilde{C}$ (using the smallness of $\epsilon$) and close the bootstrap assumption.

\textbf{Metric components, $\Omega$, $b$, and $\gamma$:} We start by presenting some auxiliary estimates for the metric curvature components. 

We integrate equation \eqref{eq:upuOm} in $v$ and we have as $\Omega = 1$ initially that
\begingroup
\allowdisplaybreaks
\begin{align*}
\log \Omega (u , v, \theta^1 , \theta^2 ) = & - \int_{v_0}^{v} \omega \, d v' \\ & - \frac{1}{2} \int_{v_0}^{v} b^A ( \eta_A + \underline{\eta}_A ) \, d v' .
\end{align*}
\endgroup
For the first term of the right hand side above we have that:
$$ - \int_{v_0}^{v} \omega \, d v' \leq \sum_{i=0}^2 \int_{v_0}^{v}  \frac{1}{v'} \|( v' \snabla )^i \omega \|_{L^2 (S_{u,v'} )} \, d v' \leq \frac{1}{v_0} \sum_{i=0}^2 \| ( v' \snabla )^i \omega \|_{L^1_{v'} L^2 (S_{u,v'} )} , $$
where we used Sobolev's inequality \eqref{sobolev3}, and then the last term can be bounded by the bootstrap assumptions on $\omega$. For the second term we make the bootstrap assumptions that 
$$ \sum_{A \in \{1,2\}} | b^A | \leq \breve{\epsilon} ,$$  
for some $\breve{\epsilon} > 0$. We have that:
\begin{align*}
- \frac{1}{2} \int_{v_0}^{v} b^A ( \eta_A + \underline{\eta}_A ) \, dv' \lesssim&  \sup | b^A | \sum_{k=0}^2 \int_{v_0}^{v} \frac{1}{(v' )^{1+\delta}} ( \| ( v' \snabla )^k \eta \|_{L^2 (S_{u , v'})} + \| ( v' \snabla )^k \underline{\eta} \|_{L^2 (S_{u , v'})} ) \, dv' \\ \lesssim & \frac{1}{v_0^{\delta}} \sup | b^A | \sup_v \sum_{k=0}^2 \left( \| ( v' \snabla )^k \eta \|_{L^2 (S_{u , v})} + \| ( v' \snabla )^k \underline{\eta} \|_{L^2 (S_{u , v})} \right), 
\end{align*}
where we used Sobolev's inequality \eqref{sobolev3} and where the last terms can be bounded by the bootstrap assumptions.
The above estimates imply that by the choice of $v_0$ and $\epsilon$ small enough we have that
$$ -\epsilon \lesssim \log \Omega \lesssim \epsilon , $$
and using again the smallness of $\epsilon$ and the monotonicity of the $\log$ function, we get that:
$$ \frac{1}{2} \leq \Omega \leq 2 . $$

On the other hand for $b$ we use equation \eqref{eq:b} and integrate it in $u$ and we have that:
\begingroup
\allowdisplaybreaks
\begin{align*}
b^A (u , v , \theta^1 , \theta^2 ) - b^A (u_0 , v , \theta^1 , \theta^2 ) = & \int_{u_0}^u \Omega^2 \gamma^{AB} ( 2 \eta_B - 2 \underline{\eta}_B ) \, du' \\ \lesssim & \sum_{k=0}^2 \int_{u_0}^u \frac{1}{v^{3+\delta}} ( \| (v \snabla )^k \eta \|_{L^2(S_{u , v})} + \| \underline{\eta} \|_{L^{\infty} (S_{u , \underline{u}'})} ) \, du'  \mbox{  for $A \in \{1,2\}$,}
\end{align*}
\endgroup 
and the last term can be bounded by the bootstrap assumptions of $\eta$ and $\underline{\eta}$, the estimate we just showed for $\Omega$ (again as we work within a bootstrap argument), and that 
$$ | \gamma^{AB} | ( u , v , \theta^1 , \theta^2 ) \leq \frac{\breve{C}}{v^2} , $$
that we will show below (for some constant $\breve{C}$). The desired estimate for $b$ now follows from the smallness of $\epsilon$ and the fact that $b^A (u_0 , v , \theta^1 , \theta^2 ) = 0$.

Equation \eqref{eq:ulgamma} implies that:
\begin{equation}\label{eq:detgamma}
\frac{\partial}{\partial u} \log ( \det \gamma ) = \Omega^2 \slashed{\mathrm{tr}} \underline{\chi} .
\end{equation}
This gives us that
$$ \det \gamma ( u , v , \theta^1 , \theta^2 ) = \det \gamma ( u_0 , v , \theta^1 , \theta^2 ) \cdot e^{\int_{u_0}^u \Omega^2 \slashed{\mathrm{tr}} \underline{\chi} \, du' } , $$
and since by the boundedness estimate for $\Omega$ and the bootstrap assumptions for $\slashed{\mathrm{tr}} \underline{\chi}$ we have that
$$ \int_{u_0}^u \Omega^2 \slashed{\mathrm{tr}} \underline{\chi} \, du' \leq 4 \left( \int_{u_0}^u \frac{1}{\mathrm{w}^2 (u' )} \, du' \right)^{1/2} \left( \int_{u_0}^u \mathrm{w}^2 (u' ) \| \slashed{\mathrm{tr}}  \underline{\chi} \|_{L^2 ( S_{ u' , \underline{u}} )}^2  \, du' \right)^{1/2}  , $$
we get in the end that
$$ | \det \gamma ( u , \underline{u} , \theta^1 , \theta^2 ) - \det \gamma ( u_0 , \underline{u} , \theta^1 , \theta^2 ) | \lesssim C \epsilon , $$
and since $\epsilon$ is small enough this implies the desired estimate for $\det \gamma$. The lower bound for $\det \gamma$ implies that $\gamma$ is bounded from below as well. On the other hand by integrating \eqref{eq:ulgamma} in $u$  we have that:
$$ | \gamma_{AB} ( u , v , \theta^1 , \theta^2 ) - \gamma_{AB} ( u_0 , v, \theta^1 , \theta^2 ) | \leq 2 \int_{u_0}^u | \Omega^2 \underline{\chi}_{AB} | \, du' \leq 8 \left( \int_{u_0}^u \frac{1}{\mathrm{w}^2 (u' )} \, du' \right)^{1/2} \sum_{k=0}^2 \left( \int_{u_0}^u \frac{1}{v} \mathrm{w}^2 (u' ) \| ( v \snabla )^k \underline{\chi} \|_{L^2 (S_{u' , \underline{u} })}^2  \, du' \right)^{1/2} , $$
where we used again the boundedness of $\Omega$ and the bootstrap assumptions for $\underline{\hat{\chi}}$ and $\slashed{\mathrm{tr}} \underline{\chi}$, which in turn imply the desired estimate for $\gamma$.

\textbf{Riemann curvature components, 1. $\beta$, $K - \frac{1}{v^2}$, and $\check{\sigma}$:} We use the energy inequalities of Propositions \ref{intbyparts} and \ref{intbyparts3} where we use the $\slashed{\nabla}_3$ equation \eqref{eq:beta3} for $\beta$ and the $\slashed{\nabla}_4$ equations \eqref{eq:gauss4}, \eqref{eq:csigma4} for $K$ and $\check{\sigma}$ and we have that:
\begin{align}\label{boot:riemann1}
\int_{H_{u_2}} v^3 | \beta |^2 + \int_{\underline{H}_{v_2}} v_2^3  \left( \left| K - \frac{1}{v^2} \right|^2 + | \check{\sigma}  |^2 \right) \lesssim & \int_{H_{u_2}} v^3 | \beta |^2 + \int_{\underline{H}_{v_1}} v_1^3 \left( \left|  K - \frac{1}{v^2} \right|^2 + | \check{\sigma} |^2 \right) \\ & + \int_{D_{u,v}} v^3 \left( \langle \beta , F_{\beta} \rangle +  \left\langle K - \frac{1}{v^2} , F_{K - \frac{1}{v^2}} \right\rangle + \langle \sigma , F_{\check{\sigma}} \rangle \right) ,
\end{align}
where 
$$
\slashed{\nabla}_3 \beta + \slashed{\mathrm{tr}} \underline{\chi} \beta - \slashed{\nabla} \left( K - \frac{1}{v^2} \right) +^{*} \slashed{\nabla} \sigma = F_{\beta} , \quad  \slashed{\nabla}_4 \left( K - \frac{1}{v^2} \right) + \frac{3}{2}\slashed{\mathrm{tr}} \chi \left( K - \frac{1}{v^2} \right) + \slashed{\mathrm{div}} \beta = F_{K - \frac{1}{v^2}} , \quad \slashed{\nabla}_4 \check{\sigma}  +\frac{3}{2} \slashed{\mathrm{tr}} \chi \check{\sigma} +^{*}\slashed{\mathrm{div}}\beta = F_{\check{\sigma}} . $$
Note that we used the bootstrap assumption on $\slashed{\mathrm{tr}} \underline{\chi}$ so that we have:
$$ \int_{D_{u,v}} \slashed{\mathrm{tr}} \underline{\chi} | \beta |^2 \lesssim  \sup_{u \in [ u_1 , u_2 ]} \int_{H_u} v^2 | \beta |^2 , $$
and this term can be absorbed by the left hand side after taking supremums. On the other hand the terms involving $\slashed{\mathrm{tr}} \chi$ for $\rho$ and $\sigma$ have the right sign after using equations \eqref{eq:rho4} and \eqref{eq:sigma4} and the bootstrap assumption on $\slashed{\mathrm{tr}} \chi$ (taking $v_0$ large enough if needed). In this way we can ignore all terms involving $\slashed{\mathrm{tr}} \chi$ and $\slashed{\mathrm{tr}} \underline{\chi}$ coming both from the identities \eqref{ibp1} and \eqref{ibp2} and the equations for $\beta$, $K - \frac{1}{v^2}$ and $\check{\sigma}$. So for the nonlinear terms we start with $F_{\beta}$ and we have that:
$$
\int_{D_{u,v}} v^3 \langle \beta ,  F_{\beta} \rangle  =   I_1 + II_1 + III_1 + IV_1 + V_1 + VI_1 + VII_1 + VIII_1 + IX_1  ,    $$
where 
$$ \mbox{$I_1$ involves the term $\frac{6}{v^2} \eta$,} $$ 
$$ \mbox{$II_1$ involves the term $2 \hat{\chi} \cdot \underline{\beta}$,} $$
$$ \mbox{ $III_1$ involves the terms $ \eta \left( K - \frac{1}{v^2} \right)$ and $\eta \check{\sigma}$,} $$
$$ \mbox{ $IV_1$ involves the terms $ \slashed{\nabla} ( \hat{\chi} \cdot \hat{\underline{\chi}} )$ and $ \slashed{\nabla} ( \hat{\chi} \wedge \hat{\underline{\chi}} )$,} $$
$$ \mbox{ $V_1$ involves the terms $\eta \hat{\chi} \cdot \hat{\underline{\chi}}$ and $\eta \hat{\chi} \wedge \hat{\underline{\chi}}$,} $$
$$ \mbox{ $VI_1$ involves the terms $\slashed{\nabla} \left(  \slashed{\mathrm{tr}} \chi - \frac{2}{v} \right) \cdot \left( \slashed{\mathrm{tr}} \underline{\chi} + \frac{2}{v} \right)$ and $\slashed{\nabla} \left(  \slashed{\mathrm{tr}} \underline{\chi} + \frac{2}{v} \right) \cdot \left( \slashed{\mathrm{tr}} \chi - \frac{2}{v} \right)$,} $$
$$ \mbox{ $VII_1$ involves the terms $\slashed{\nabla} \left(  \slashed{\mathrm{tr}} \chi - \frac{2}{v} \right) \cdot \frac{2}{v} $ and $\slashed{\nabla} \left(  \slashed{\mathrm{tr}} \underline{\chi} + \frac{2}{v} \right) \cdot \frac{2}{v}$,} $$
$$ \mbox{ $VIII_1$ involves the terms $\eta \left(  \slashed{\mathrm{tr}} \chi - \frac{2}{v} \right) \cdot \frac{2}{v}$ and $\eta \left(  \slashed{\mathrm{tr}} \underline{\chi} + \frac{2}{v} \right) \cdot \frac{2}{v}$,} $$
$$ \mbox{ and $IX_1$ involves the term $\eta \left( \slashed{\mathrm{tr}} \chi - \frac{2}{v} \right) \cdot \left( \slashed{\mathrm{tr}} \underline{\chi} + \frac{2}{v} \right)$.} $$
We examine all terms one by one. For the $I_1$ term we have that
\begin{align*}
I_1 = & \int_{D_{u,v}} \left\langle \beta , - \frac{6}{v^2} \eta \right\rangle \\ \lesssim & \epsilon \sup_u \int_{H_u} v^3 | \beta |^2 + \frac{1}{\epsilon} \int_{D_{u,v}} v^3 \frac{1}{v^4} | \eta |^2 ,
\end{align*}
and the first term of the last expression can be absorbed by the left-hand side of \eqref{boot:riemann1}, while for the second term (which is linear in $\eta$) we use the bootstrap assumption (which guarantees integrability in $v$ as $\| \eta \|_{L^{2} (S_{u,v} )} \lesssim \frac{1}{v}$ so the overall $v$-weight is of the order of $\frac{1}{v^3}$) and choose $U$ such that $U-u_0$ is small enough to make the term appropriately small. 

For the $II_1$ term we have that
\begin{align*}
II_1 = & \int_{D_{u,v}} v^3 \langle \beta , 2 \hat{\chi} \cdot \underline{\beta} \rangle \\ \lesssim & \epsilon \int_{D_{u,v}} \frac{1}{\mathrm{w}^2 (u)} v^3 | \beta |^2 + \frac{1}{\epsilon} \int_{D_{u,v}} \mathrm{w}^2 (u) v^3 | \hat{\chi} |^2 | \underline{\beta} |^2  \\ \lesssim & \epsilon \sup_u \int_{H_u} v^3 | \beta |^2 + \frac{1}{\epsilon} \left( \int_v \sup_u v | \hat{\chi} |^2 \right) \times \sup_v \left( \int_u \int_{S_{u,v}} v^2 | \underline{\beta} |^2 \right) \\ \lesssim & \epsilon \sup_u \int_{H_u} v^3 | \beta |^2 + \frac{1}{\epsilon} + \frac{1}{\epsilon} \sum_{k=0}^2 \left\| \frac{1}{\sqrt{v}} ( v \slashed{\nabla} )^k \hat{\chi} \right\|_{L^2_v L^{\infty}_u L^2 (S_{u,v} )}^2 \| v \underline{\beta} \|_{L^{\infty}_v L^2 ( \underline{H}_v ) }^2 ,
\end{align*}
where we used Sobolev's inequality \eqref{sobolev3}, the integrability of $\frac{1}{\mathrm{w}^2 (u)}$, and then in the final expression the first term can be absorbed by the right hand side of \eqref{boot:riemann1}, while the rest of the terms can be bounded by the bootstrap assumptions.

For the $III_1$ term we have that
\begin{align*}
III_1 = & \int_{D_{u,v}} v^3 \left\langle \beta , -3 \eta \left( K - \frac{1}{v^2} \right)  +3^{*} \eta \check{\sigma} \right\rangle \\ \lesssim &  \epsilon \sup_u \int_{H_u} v^3 | \beta |^2 + \frac{1}{\epsilon} \int_{D_{u,v}} v^3 | \eta |^2 \left( \left| K - \frac{1}{v^2} \right|^2 + | \check{\sigma} |^2 \right) \\ \lesssim &  \epsilon \sup_u \int_{H_u} v^3 | \beta |^2 \\ & + \frac{1}{\epsilon} \left( \int_v \sum_{k=0}^2 \int_v \sup_u \frac{1}{v^2} \| ( v \slashed{\nabla} )^k \eta \|_{L^2 (S_{u,v} )}^2 \right) \times \left( \left\| v^{3/2} \left( K - \frac{1}{v^2} \right) \right\|_{L^{\infty}_v L^2 (\underline{H}_v )}^2 + \left\| \mathrm{w} (u) v^{3/2} \check{\sigma} \right\|_{L^{\infty}_v L^2 (\underline{H}_v )}^2 \right) ,
\end{align*}
and as before the first term of the last expression can be absorbed by the left-hand side of \eqref{boot:riemann1} and the remaining terms can be bounded by the bootstrap assumptions.

For the $IV_1$ term we have that
\begin{align*}
IV_1 = & \int_{D_{u,v}} v^3 \langle \beta , \frac{1}{2} \left( \slashed{\nabla} ( \hat{\chi} \cdot \hat{\underline{\chi}} )  +^{*} \slashed{\nabla} ( \hat{\chi} \wedge \hat{\underline{\chi}} ) \right) \rangle \\ \lesssim & \epsilon  \int_{D_{u,v}} \frac{1}{\mathrm{w}^2 (u)} v^3 | \beta |^2 + \frac{1}{\epsilon} \int_{D_{u,v}} \mathrm{w}^2 (u) v ( | ( v \slashed{\nabla} ) \hat{\chi} |^2 | \hat{\underline{\chi}} |^2 + | \hat{\chi} |^2 | ( v \slashed{\nabla} ) \hat{\underline{\chi}} |^2  \\ \lesssim & \epsilon \sup_u \int_{H_u} v^3 | \beta |^2 + \frac{1}{\epsilon} \| \mathrm{w} (u) ( v \slashed{\nabla} ) \underline{\hat{\chi}} \|_{L^2_u L^{\infty}_v L^2 (S_{u,v} )}^2 \left\|  \frac{1}{\sqrt{v}} \hat{\chi} \right\|_{L^2_v L^{\infty}_u L^2 (S_{u,v} )}^2 \\ & +\| \mathrm{w} (u) \underline{\hat{\chi}} \|_{L^2_u L^{\infty}_v L^2 (S_{u,v} )}^2 \left\|  \frac{1}{\sqrt{v}} (v \slashed{\nabla} ) \hat{\chi} \right\|_{L^2_v L^{\infty}_u L^2 (S_{u,v} )}^2 ,
\end{align*}
and the first term can be moved to the left-hand side of \eqref{boot:riemann1}, while the other terms are bounded by the bootstrap assumptions.

For the $V_1$ term we have that
\begin{align*}
V_1 = & \int_{D_{u,v}} v^3 \left\langle \beta , \frac{3}{2} ( \eta \hat{\chi} \cdot \hat{\underline{\chi}} +^{*} \eta \hat{\chi} \wedge \hat{\underline{\chi}} ) \right\rangle \\ \lesssim & \epsilon \int_{D_{u,v}} \frac{1}{\mathrm{w}^2 (u) } v^3 | \beta |^2 + \frac{1}{\epsilon} \int_{D_{u,v}} \mathrm{w}^2 (u) v^3 | \eta |^2 | \hat{\chi} |^2 | \hat{\underline{\chi}} |^2 \\ \lesssim & \epsilon \sup_u \int_{H_u} v^3 | \beta |^2 + \frac{1}{\epsilon} \left( \int_v \sup_u v^3 | \eta |^2 | \hat{\chi} |^2 \right) \times \left( \sup_v \int_u \int_{S_{u,v}} \mathrm{w}^2 (u) | \hat{\underline{\chi}} |^2 \right)  \\ \lesssim & \epsilon \sup_u \int_{H_u} v^3 | \beta |^2 + \frac{1}{\epsilon} \left( \| v ( v \slashed{\nabla} )^k \eta \|_{L^{\infty}_{u,v} L^2 (S_{u,v} )}^2 \left\| \frac{1}{\sqrt{v}} ( v \slashed{\nabla} )^k \hat{\chi} \right\|_{L^2_v L^{\infty}_u L^2 (S_{u,v} )}^2 \right) \times \|\mathrm{w} (u) \hat{\underline{\chi}} \|_{L^2_u L^{\infty}_v L^2 (S_{u,v} )}^2 ,
\end{align*}
where we used Sobolev's inequality twice for both $\eta$ and $\hat{\chi}$ (note that this use of Sobolev for these specific terms was not necessary and it could be used for another combination of terms), and in this way the terms involving $\eta$, $\hat{\chi}$ and $\hat{\underline{\chi}}$ of the final expression can be bounded by the bootstrap assumptions. Note that in the end we have more decay in $v$ than required.

For the $VI_1$ term we have that
\begin{align*}
VI_1 = & \int_{D_{u,v}} v^3 \langle \beta , - \frac{1}{4} \left( \slashed{\nabla} \left(  \slashed{\mathrm{tr}} \chi - \frac{2}{v} \right) \cdot \left( \slashed{\mathrm{tr}} \underline{\chi} + \frac{2}{v} \right) + \left(  \slashed{\mathrm{tr}} \chi - \frac{2}{v} \right) \cdot \slashed{\nabla} \left( \slashed{\mathrm{tr}} \underline{\chi} + \frac{2}{v} \right) \right\rangle \\ \lesssim & \epsilon \int_{D_{u,v}} \frac{1}{\mathrm{w}^2 (u)} v^3 | \beta |^2 + \frac{1}{\epsilon} \int_{D_{u,v}} \mathrm{w}^2 (u) v \left( \left| ( v \slashed{\nabla} ) \left( \slashed{\mathrm{tr}} \chi - \frac{2}{v} \right) \right|^2 \left|\slashed{\mathrm{tr}} \underline{\chi} + \frac{2}{v}  \right|^2 + \left| ( v \slashed{\nabla} ) \left( \slashed{\mathrm{tr}} \underline{\chi} + \frac{2}{v} \right) \right|^2 \left|\slashed{\mathrm{tr}} \chi - \frac{2}{v}  \right|^2 \right) \\ \lesssim & \epsilon \sup_u \int_{H_u} v^3 | \beta |^2 + \frac{1}{\epsilon} \Bigg[ \sum_{k=0}^2 \left\| \mathrm{w} (u) ( v \slashed{\nabla} )^k \left( \slashed{\mathrm{tr}} \underline{\chi} + \frac{2}{v} \right) \right\|_{L^2_u L^{\infty}_v L^2 (S_{u,v} )}^2 \left\| \frac{1}{\sqrt{v}} ( \slashed{\nabla} ) \left( \slashed{\mathrm{tr}} \chi - \frac{2}{v} \right) \right\|_{L^2_v L^{\infty}_u L^2 (S_{u,v} )}^2 \\ &+ \sum_{k=0}^2 \left\| \mathrm{w} (u) ( v \slashed{\nabla} )^k \left( \slashed{\mathrm{tr}} \chi - \frac{2}{v} \right) \right\|_{L^2_u L^{\infty}_v L^2 (S_{u,v} )}^2 \left\| \frac{1}{\sqrt{v}} ( \slashed{\nabla} ) \left( \slashed{\mathrm{tr}} \underline{\chi} + \frac{2}{v} \right) \right\|_{L^2_v L^{\infty}_u L^2 (S_{u,v} )}^2 \Bigg] , 
\end{align*}
where we used Sobolev's inequality \eqref{sobolev3} on the term with less angular derivatives, and we note that the last four terms can be bounded by the bootstrap assumptions.

For the $VII_1$ term we have that
\begin{align*}
VII_1 = & \int_{D_{u,v}} \int_{D_{u,v}} v^3 \langle \beta , \frac{1}{4} \left( ( v \slashed{\nabla} ) \left(  \slashed{\mathrm{tr}} \chi - \frac{2}{v} \right) \cdot \frac{2}{v^2}  +  ( v \slashed{\nabla} ) \left( \slashed{\mathrm{tr}} \underline{\chi} + \frac{2}{v} \right) \cdot \frac{2}{v} \right\rangle \\ \lesssim & \epsilon \sup_u \int_{H_u} v^3 |\beta|^2 + \frac{1}{\epsilon} \int_{D_{u,v}} \frac{1}{v} \left| ( v \slashed{\nabla} ) \left( \mathrm{tr} \chi - \frac{2}{v} \right) \right|^2 \\ & + \epsilon \int_{D_{u,v}} \frac{1}{\mathrm{w}^2 (u)} v^3 | \beta |^2 + \frac{1}{\epsilon} \int_{D_{u,v}} \mathrm{w}^2 (u) \frac{1}{v} \left| ( v \slashed{\nabla} ) \left( \mathrm{tr} \underline{\chi} + \frac{2}{v} \right) \right|^2 \\ \lesssim & \epsilon \sup_u \int_{H_u} v^3 |\beta|^2 + \epsilon \left\| \frac{1}{\sqrt{v}} \left( ( v \slashed{\nabla} ) \left( \mathrm{tr} \chi - \frac{2}{v} \right) \right) \right\|_{L^2_v L^{\infty}_u L^2 (S_{u,v} )}^2 + \frac{1}{\epsilon} \int_{v_0}^{\infty} \frac{1}{v^3} \, dv \left\| v ( v \slashed{\nabla} ) \left( \slashed{\mathrm{tr}} \underline{\chi} + \frac{2}{v} \right) \right\|_{L^2_u L^{\infty}_v L^2 (S_{u,v} )}^2 ,
\end{align*}
where in the second term of the last expression (which is linear in $\slashed{\mathrm{tr}} \chi - \frac{2}{v}$) we choose $U$ such that $U - u_0$ is appropriately small, while for the third and last term (which is linear in $\slashed{\mathrm{tr}} \underline{\chi} + \frac{2}{v}$) we can get smallness by an appropriate choice of $v_0$.

Terms $VIII_1$ and $IX_1$ can be treated similarly to the terms $VI_1$ and $VII_1$ with an appropriate use of Sobolev's inequality \eqref{sobolev3}.

For the higher derivatives 
$$ \int_{H_u} v^3 | (v \slashed{\nabla} ) \beta |^2 , \quad \int_{H_u} v^3 | (v \slashed{\nabla} )^2 \beta |^2 \mbox{  and  } \int_{H_u} v^3 | (v \slashed{\nabla} )^3 \beta |^2 , $$
we work similarly, and it can be easily seen by the way we treated $\int_{H_u} v^3 | \beta |^2$ that we can work similarly by appropriately using Sobolev's inequality \eqref{sobolev3} and arrive at the same estimates. First we should note that the extra terms that present higher order nonlinear estimates (that have additional $\eta$, $\underline{\eta}$ and $\underline{\beta}$ factors), have better decay, while in the case where all angular derivatives fall on $F_{\beta}$ we demonstrate for the sake of completeness the argument for three angular derivatives affecting the term $IV_1$ and we have that
\begin{align*}
\int_{D_{u,v}} v^3 \langle  (v \slashed{\nabla} )^3 \beta , v^3 \sum_{k_1 + k_2 = 4} & ( ( \slashed{\nabla} )^{k_1} \hat{\chi} )  \cdot  ( ( \slashed{\nabla} )^{k_2} \hat{\underline{\chi}} ) \rangle \lesssim  \epsilon \int_{D_{u,v}} \frac{1}{\mathrm{w}^2 (u) } v^3 | ( v \slashed{\nabla} )^2 \beta |^2 + \frac{1}{\epsilon} \sum_{k_1 + k_2 = 4} \int_{D_{u,v}} \mathrm{w}^2 (u)  v^9 | ( \slashed{\nabla} )^{k_1} \hat{\chi} |^2 | ( \slashed{\nabla} )^{k_2} \hat{\underline{\chi}} |^2 \\ \lesssim &  \epsilon \sup_u \int_{H_u} v^3 | ( v \slashed{\nabla} )^2 \beta |^2 + \frac{1}{\epsilon} \sum_{k_1 + k_2 = 4} \int_{D_{u,v}} \mathrm{w}^2 (u) v | ( v \slashed{\nabla} )^{k_1} \hat{\chi} |^2 | ( v \slashed{\nabla} )^{k_2} \hat{\underline{\chi}} |^2 \\ \lesssim &  \epsilon \sup_u \int_{H_u} v^3 | ( v \slashed{\nabla} )^2 \beta |^2 \\ & + \frac{1}{\epsilon} \sum_{l_1 + l_2 = 4, l_1 \leq l_2} \left( \int_v \sup_u v | ( v \slashed{\nabla} )^{l_1} \hat{\chi} |^2 \right) \times \sup_v \left( \int_u \int_{S_{u,v}} \mathrm{w}^2 (u) | ( v \slashed{\nabla} )^{l_2} \hat{\underline{\chi}} |^2 \right) \\ & + \frac{1}{\epsilon} \sum_{m_1 + m_2 = 4, m_1 \leq m_2} \left( \int_v \sup_u \int_{S_{u,v}} v | ( v \slashed{\nabla} )^{m_2} \hat{\chi} |^2 \right) \times \sup_v \left( \int_u \mathrm{w}^2 (u)  | ( v \slashed{\nabla} )^{m_1} \hat{\underline{\chi}} |^2 \right) \\ \lesssim &  \epsilon \sup_u \int_{H_u} v^3 | ( v \slashed{\nabla} )^2 \beta |^2 + \frac{1}{\epsilon} \sum_{k_1 \leq 4 , k_2 \leq 4} \left\| \frac{1}{\sqrt{v}} ( v \slashed{\nabla} )^{k_1} \hat{\chi} \right\|_{L^2_v L^{\infty}_u L^2 (S_{u,v} )}^2 \left\| \mathrm{w} (u) ( v \slashed{\nabla} )^{k_2} \hat{\underline{\chi}}  \right\|_{L^{\infty}_v L^2_u L^2 (S_{u,v} )}^2 ,
\end{align*} 
where we note that we used the same method as in the uncommuted case, using Sobolev's inequality \eqref{sobolev3} on the term with is hit by fewer angular derivatives.

Using equation \eqref{eq:rgauss4} we have that
$$ \int_{D_{u,v}} v^3 \left\langle K - \frac{1}{v^2} , F_{K - \frac{1}{v^2}} \right\rangle = I_2 + II_2 + III_2 + IV_2 + V_2 . $$
We examine again these terms one by one. For the $I_2$ term we have that
$$ I_1 =  \int_{D_{u,v}} v^3 \left\langle K - \frac{1}{v^2} , \frac{1}{2} \left( \slashed{\mathrm{tr}} \chi - \frac{2}{v} \right) \underline{\mu} + \frac{1}{v} \underline{\mu} \right\rangle , $$
where we recall that $\underline{\mu} = K - \frac{1}{v^2} - \slashed{\mathrm{div}} \underline{\eta}$. We have that
\begin{align*}
\int_{D_{u,v}} v^3 \left\langle K - \frac{1}{v^2} ,  \frac{1}{v} \underline{\mu} \right\rangle \lesssim & \int_{D_{u,v}} \frac{v^3}{v^{1+\delta}}  \left| K - \frac{1}{v^2} \right|^2 + \int_{D_{u,v}} v^{4+\delta} \frac{1}{v^2} | \underline{\mu} |^2 \\ \lesssim & \epsilon \sup_u \int_{H_u} \mathrm{w}^2 (u) v^2 \left| K - \frac{1}{v^2} \right|^2 + \epsilon \sup_v \int_{u} \int_{S_{u,v}} v^4  | \underline{\mu} |^2 , 
\end{align*}
where to gain the smallness in $\epsilon$ in both terms we used the appropriate choice of $v_0$ and both of these terms can be bounded by the bootstrap assumptions. The other part can be treated as follows: 
\begin{align*}
\int_{D_{u,v}} v^3 \left\langle K - \frac{1}{v^2} , \left( \slashed{\mathrm{tr}} \chi - \frac{2}{v} \right) \underline{\mu} \right\rangle \lesssim & \int_{D_{u,v}}  \frac{v^3}{v^2} \left| K - \frac{1}{v^2} \right|^2 + \int_{D_{u,v}}  v^5 \left| \slashed{\mathrm{tr}} \chi - \frac{2}{v} \right|^2 | \underline{\mu} |^2 \\ \lesssim & \epsilon \sup_v \int_{\underline{H}_v} v^3 \left| K - \frac{1}{v^2} \right|^2 + \int_{D_{u,v}} v^5 \left| \slashed{\mathrm{tr}} \chi - \frac{2}{v} \right|^2 | \underline{\mu} |^2 \\ \lesssim & \epsilon \sup_u \int_{H_u} \mathrm{w}^2 (u) v^2 \left| K - \frac{1}{v^2} \right|^2 \\ &+ \epsilon \sum_{k=0}^2 \left\| v ( v \slashed{\nabla} )^k \left( \slashed{\mathrm{tr}} \chi - \frac{2}{v} \right) \right\|_{L^{\infty}_{u,v} L^2 (S_{u,v} )}^2 \times \sup_v \left( \int_u \int_{S_{u,v}} v^4 | \underline{\mu} |^2 \right) ,
\end{align*}
where for the second term we used Sobolev's inequality \eqref{sobolev3} and the resulting term can be bounded by the bootstrap assumptions. Note that for higher derivatives we apply Sobolev's inequality \eqref{sobolev3} to the term with fewer derivatives and always evaluate $\underline{\mu}$ on an $H_u$ hypersurface. This is always possible as we can always make sure that no more than 3 angular derivatives are used on $\slashed{\mathrm{tr}} \chi - \frac{2}{v}$.

For the $II_2$ term we have that
\begin{align*}
II_2 = & \int_{D_{u,v}} v^3 \left\langle K - \frac{1}{v^2} , - ( \zeta + 2 \underline{\eta} ) \cdot \beta \right\rangle \lesssim   \int_{D_{u,v}} \frac{v^3}{v^2} \left| K - \frac{1}{v^2} \right|^2 + \int_{D_{u,v}} v^5 | \underline{\eta} |^2 | \beta |^2 \\ \lesssim & \epsilon \sup_v \int_{\underline{H}_v} v^3 \left| K - \frac{1}{v^2} \right|^2 + \sum_{k=0}^2  \|v (v \slashed{\nabla} )^k \underline{\eta} \|^2_{L^{\infty}_{u,v} L^2 (S_{u,v} )} \sup_u \int_{H_u} v | \beta |^2 ,
\end{align*}
where we used Sobolev's inequality \eqref{sobolev3}, and we note that in the last expression the first term can be absorbed by the left-hand side of \eqref{boot:riemann1} after achieving smallness in $\epsilon$ through an appropriate choice of $v_0$, while the last term can be bounded by the bootstrap assumptions.

For the $III_2$ term we have that
\begin{align*}
III_2 = & \int_{D_{u,v}}  v^3 \left\langle K - \frac{1}{v^2} , \frac{1}{2} \hat{\chi} \cdot ( \slashed{\nabla} \hat{\otimes} \underline{\eta} ) \right\rangle \lesssim   \int_{D_{u,v}} \frac{v^3}{v^2} \left| K - \frac{1}{v^2} \right|^2 + \int_{D_{u,v}} v^5 | \hat{\chi} |^2 | \slashed{\nabla} \underline{\eta} |^2 \\ \lesssim & \epsilon \sup_v \int_{\underline{H}_v} v^3 \left| K - \frac{1}{v^2} \right|^2 + \left( \int_v \sup_u v |\hat{\chi} |^2 \right) \times \left( \sup_v \int_u \int_{S_{u,v}} v^2 | ( v \slashed{\nabla} ) \underline{\eta} |^2  \right) \\ \lesssim & \epsilon \sup_v \int_{\underline{H}_v} v^3 \left| K - \frac{1}{v^2} \right|^2 +  \left\| \frac{1}{\sqrt{v}} ( v \slashed{\nabla} )^k \hat{\chi} \right\|_{L^2_v L^{\infty}_u L^2 (S_{u,v} )}^2  \times \left( \sup_v \int_u \int_{S_{u,v}} v^2 | ( v \slashed{\nabla} ) \underline{\eta} |^2 \right) ,
\end{align*}
where in the last expression the first term can be absorbed by the left-hand side of \eqref{boot:riemann1} after achieving smallness in $\epsilon$ through an appropriate choice of $v_0$, and the last term (which can be bounded by the bootstrap assumptions) was obtained after applying Sobolev's inequality \eqref{sobolev3}.

The term $IV_2$ is given by
$$ IV_2 = \int_{D_{u,v}}  v^3 \left\langle K - \frac{1}{v^2} , \frac{1}{2} \hat{\chi} \cdot ( \underline{\eta} \hat{\otimes} \underline{\eta} ) \right\rangle $$
and can be treated similar to $III_2$ (actually it has better decay in $v$).

Finally for the term $V_2$ we have that
$$ V_2 =  \int_{D_{u,v}} v^3 \left\langle K - \frac{1}{v^2} , - \frac{1}{2} \left( \slashed{\mathrm{tr}} \chi - \frac{2}{v} \right) | \underline{\eta} |^2 - \frac{1}{2v} | \underline{\eta} |^2 \right\rangle ,  $$
and since $\slashed{\mathrm{tr}} \chi - \frac{2}{v}$ decays better than $\frac{1}{v}$, we focus only on the last part, and we have that
\begin{align*}
\int_{D_{u,v}} v^3 \left\langle K - \frac{1}{v^2} ,- \frac{1}{2v} | \underline{\eta} |^2 \right\rangle \lesssim & \int_{D_{u,v}} \frac{v^3}{v^2} \left| K - \frac{1}{v^2} \right|^2 + \int_{D_{u,v}} v^5 \frac{1}{v^2} | \underline{\eta} |^4 \\ \lesssim & \epsilon \sup_v \int_{\underline{H}_v} v^3 \left| K - \frac{1}{v^2} \right|^2 + \sum_{k=0}^1 \int_{u} \int_v  \frac{1}{v^3} \| v ( v \slashed{\nabla} )^k \underline{\eta} \|_{L^4 (S_{u,v} )}^4  \\ \lesssim & \epsilon \sup_v \int_{\underline{H}_v} v^3 \left| K - \frac{1}{v^2} \right|^2 + \epsilon \sum_{k=0}^1  \| v ( v \slashed{\nabla} )^k \underline{\eta} \|_{L^{\infty}_{u,v} L^4 (S_{u,v} )}^4 ,
\end{align*}
where once again in the last expression the first term can be absorbed by the left-hand side of \eqref{boot:riemann1} after achieving smallness in $\epsilon$ through an appropriate choice of $v_0$, and for the last term we used Sobolev's inequality \eqref{sobolev1} and achieved smallness in $\epsilon$ by appropriate choices of $U$ and $v_0$.

From the way that we dealt with the terms $I_2$, $II_2$, $III_2$, $IV_2$ and $V_2$ it is clear that we can work in the same way to obtain the analogous results for 
$$\int_{\underline{H}_v} v^3 \left| ( v \slashed{\nabla} ) K - \frac{1}{v^2} \right|^2 , \quad  \int_{\underline{H}_v} v^3 \left| ( v \slashed{\nabla} )^2 K - \frac{1}{v^2} \right|^2 \mbox{  and  } \int_{\underline{H}_v} v^3 \left| ( v \slashed{\nabla} )^3 K - \frac{1}{v^2} \right|^2 . $$
Finally the terms of $F_{\check{\sigma}}$ are similar to the terms of $F_{K - \frac{1}{v^2}}$ and can be treated in the exact same way.

\textbf{Riemann curvature components, 2. $\underline{\beta}$, $K - \frac{1}{v^2}$, $\check{\sigma}$:} We use again the energy inequalities of Proposition \ref{intbyparts} where we use the $\slashed{\nabla}_4$ equation \eqref{eq:ubeta44} for $\underline{\beta}$ and the $\slashed{\nabla}_3$ equations \eqref{eq:gauss3}, \eqref{eq:csigma3} for $K-\frac{1}{v^2}$ and $\check{\sigma}$ and we have that (arguing as before):
\begin{align}\label{boot:riemann2}
\int_{\underline{H}_{v_2}} \mathrm{w}^2 (u) v_2^2 | \underline{\beta} |^2 + \int_{H_{u_2}} \mathrm{w}^2 (u_2 ) v^2 | \left( K - \frac{1}{v^2} , \check{\sigma} \right) |^2 \lesssim & \int_{\underline{H}_{v_1}} \mathrm{w}^2 (u) v_1^2 | \underline{\beta} |^2 + \int_{H_{u_1}} \mathrm{w}^2 (u_1 ) v^2 | \left( K - \frac{1}{v^2} , \check{\sigma} \right) |^2 \\ + &  \int_{D_{u,v}} \mathrm{w}^2 (u) v^2 \left( \langle \underline{\beta} , G_{\underline{\beta}} \rangle + \left\langle K - \frac{1}{v^2} , G_{K-\frac{1}{v^2}} \right\rangle + \langle \check{\sigma} , G_{\check{\sigma}} \rangle \right) ,
\end{align}
where 
$$ \slashed{\nabla}_4 \underline{\beta} + \slashed{\mathrm{tr}} \chi \underline{\beta} - \slashed{\nabla} \left( K - \frac{1}{v^2} \right) -^{*} \slashed{\nabla} \check{\sigma} = G_{\underline{\beta}} , \quad  \slashed{\nabla}_3 \left( K - \frac{1}{v^2} \right) + \slashed{\mathrm{tr}} \underline{\chi} \left( K - \frac{1}{v^2} \right) - \slashed{\mathrm{div}} \underline{\beta} = G_{K-\frac{1}{v^2}} , \quad \slashed{\nabla}_3 \check{\sigma}  +\frac{3}{2} \slashed{\mathrm{tr}} \chi \check{\sigma} + \slashed{\mathrm{div}}^{*} \underline{\beta}= G_{\check{\sigma}} . $$
We will examine all the terms coming from these inhomogeneities one by one. We start with $G_{\underline{\beta}}$ and we have that
$$ \int_{D_{u,v}} \mathrm{w}^2 (u) v^2 \langle \underline{\beta} , G_{\underline{\beta}} \rangle = B_1 + B_2 + B_3 + B_4 + B_5 + B_6 + B_7 + B_8 + B_9 + B_{10} , $$
where 
$$ \mbox{$B_1$ involves the term $\frac{6}{v^2} \underline{\eta}$,} $$ 
$$ \mbox{$B_2$ involves the term $2 \hat{\underline{\chi}} \cdot \beta$,} $$
$$ \mbox{ $B_3$ involves the terms $ \underline{\eta} \left( K - \frac{1}{v^2} \right)$ and $\eta \check{\sigma}$,} $$
$$ \mbox{ $B_4$ involves the terms $ \slashed{\nabla} ( \hat{\chi} \cdot \hat{\underline{\chi}} )$ and $ \slashed{\nabla} ( \hat{\chi} \wedge \hat{\underline{\chi}} )$,} $$
$$ \mbox{ $B_5$ involves the terms $\underline{\eta} \hat{\chi} \cdot \hat{\underline{\chi}}$ and $\underline{\eta} \hat{\chi} \wedge \hat{\underline{\chi}}$,} $$
$$ \mbox{ $B_6$ involves the terms $\slashed{\nabla} \left(  \slashed{\mathrm{tr}} \chi - \frac{2}{v} \right) \cdot \left( \slashed{\mathrm{tr}} \underline{\chi} + \frac{2}{v} \right)$ and $\slashed{\nabla} \left(  \slashed{\mathrm{tr}} \underline{\chi} + \frac{2}{v} \right) \cdot \left( \slashed{\mathrm{tr}} \chi - \frac{2}{v} \right)$,} $$
$$ \mbox{ $B_7$ involves the terms $\slashed{\nabla} \left(  \slashed{\mathrm{tr}} \chi - \frac{2}{v} \right) \cdot \frac{2}{v} $ and $\slashed{\nabla} \left(  \slashed{\mathrm{tr}} \underline{\chi} + \frac{2}{v} \right) \cdot \frac{2}{v}$,} $$
$$ \mbox{ $B_8$ involves the terms $\underline{\eta} \left(  \slashed{\mathrm{tr}} \chi - \frac{2}{v} \right) \cdot \frac{2}{v}$ and $\underline{\eta} \left(  \slashed{\mathrm{tr}} \underline{\chi} + \frac{2}{v} \right) \cdot \frac{2}{v}$,} $$
$$ \mbox{ $B_9$ involves the term $\underline{\eta} \left( \slashed{\mathrm{tr}} \chi - \frac{2}{v} \right) \cdot \left( \slashed{\mathrm{tr}} \underline{\chi} + \frac{2}{v} \right)$,} $$
$$ \mbox{  and $B_{10}$ involves the term $2 \omega \underline{\beta}$.} $$

For the $B_1$ term we have that
\begin{align*}
B_1 = & \int_{D_{u,v}} \mathrm{w}^2 (u) v^2 \left\langle K - \frac{1}{v^2} , \frac{6}{v^2} \underline{\eta} \right\rangle \lesssim \int_{D_{u,v}} \mathrm{w}^2 (u) \frac{v^2}{v^2} \left| K - \frac{1}{v^2} \right|^2 + \int_{D_{u,v}} \mathrm{w}^2 (u)  \frac{v^4}{v^4} | \underline{\eta} |^2 \\ \lesssim & \epsilon \sup_v \int_{\underline{H}_v} v^2 \left| K - \frac{1}{v^2} \right|^2 + \epsilon \int_v \frac{1}{v^2} \| v \eta \|_{L^{\infty}_u L^2 (S_{u,v})}^2 \\ \lesssim & \epsilon \sup_v \int_{\underline{H}_v} v^2 \left| K - \frac{1}{v^2} \right|^2 + \epsilon  \| v \eta \|_{L^{\infty}_{u,v} L^2 (S_{u,v})}^2 ,
\end{align*}
where we used the boundedness of $\mathrm{w}$, we obtained smallness in $\epsilon$ by an appropriate choice of $U$ and $v_0$, and finally we have that the first term of the last expression can be absorbed by the left-hand side of \eqref{boot:riemann2} while the second term can be bounded by the bootstrap assumptions.

For the term $B_2$ we have that:
\begin{align*}
B_2 = \int_{D_{u,v}} \mathrm{w}^2 (u) v^2 \langle \underline{\beta} , 2 \underline{\hat{\chi}} \cdot \beta \rangle \lesssim & \epsilon \int_{D_{u,v}} \mathrm{w}^2 (u) \frac{v^2}{v^2} | \underline{\beta} |^2 + \frac{1}{\epsilon} \int_{D_{u,v}} \mathrm{w}^2 (u) v^4 |\underline{\hat{\chi}} |^2 | \beta |^2 \\ \lesssim & \epsilon \sup_v \int_{\underline{H}_v} \mathrm{w}^2 (u) v^2 | \underline{\beta} |^2 + \frac{1}{\epsilon} \sum_{i=0}^2 \| \mathrm{w} (u) ( v \slashed{\nabla} )^i \underline{\hat{\chi}} \|_{L^2_u L^{\infty}_v L^2 (S_{u,v} )}^2 \sup_u \int_{H_u} v^2 | \beta |^2 . 
\end{align*}
Notice that in the last line the term
$$  \sup_u \int_{H_u} v^2 | \beta |^2 $$
is one power of $v$ better than the energy that $\beta$ is bounded by the bootstrap assumptions.

For the $B_3$ term we have that
\begin{align*}
B_3 = & \int_{D_{u,v}} \mathrm{w}^2 (u) v^2 \left\langle \underline{\beta} , 3 \underline{\eta} \left( K - \frac{1}{v^2} \right) + 3^{*} \underline{\eta} \check{\sigma} \right\rangle \\ \lesssim & \int_{D_{u,v}} \mathrm{w}^2 (u) \frac{v^2}{v^2} | \underline{\beta} |^2 +  \int_{D_{u,v}} \mathrm{w}^2 (u) v^4 | \underline{\eta} |^2 \left( \left| K - \frac{1}{v^2} \right|^2 + | \check{\sigma} |^2 \right) \\ \lesssim & \epsilon \sup_v \int_{\underline{H}_v} \mathrm{w}^2 (u) v^2 | \underline{\beta} |^2 + \sum_{k=0}^2 \int_{D_{u,v}} \| v ( v \slashed{\nabla} )^k \underline{\eta} \|_{L^2 (S_{u,v} )}^2 \left( \left| K - \frac{1}{v^2} \right|^2 + | \check{\sigma} |^2 \right) \\ \lesssim & \epsilon \sup_v \int_{\underline{H}_v} \mathrm{w}^2 (u) v^2 | \underline{\beta} |^2 + \epsilon \sum_{k=0}^2 \| v ( v \slashed{\nabla} )^k \underline{\eta} \|_{L^{\infty}_{u,v} L^2 (S_{u,v} )}^2 \sup_u \int_{H_u}\left( \left| K - \frac{1}{v^2} \right|^2 + | \check{\sigma} |^2 \right) ,
\end{align*}
where the smallness in $\epsilon$ was achieved by the choice of $v_0$ for the first term (which can then be absorbed by the left-hand side of \eqref{boot:riemann2}) and by the choice of $U$ for the second term (where we used Sobolev's inequality \eqref{sobolev3}, and we note that it can be bounded by the bootstrap assumptions -- note that the very last term has a better weight in $v$ than required).

For the term $B_4$ we have that
$$ B_2 = \int_{D_{u,v}} \mathrm{w}^2 (u) v^2 \langle \underline{\beta} , - \frac{1}{2} ( \slashed{\nabla} ( \hat{\chi} \cdot \hat{\underline{\chi}} ) +^{*} \slashed{\nabla} ( \hat{\chi} \wedge \hat{\underline{\chi}} ) ) \rangle , $$
so we can just write that
\begin{align*}
B_4 \simeq & \int_{D_{u,v}} \mathrm{w}^2 (u) v^2 \langle \underline{\beta} ,   \hat{\chi} \cdot \slashed{\nabla} \underline{\chi} +  \slashed{\nabla} \hat{\chi} \cdot \underline{\chi} \rangle  \lesssim \int_{D_{u,v}} v^2 | \underline{\beta} | ( | \hat{\chi} | | \slashed{\nabla} \underline{\chi} | + | \slashed{\nabla} \hat{\chi} | | \underline{\hat{\chi}} | ) \\ \lesssim & \| \hat{\chi} \|_{L^1_v L^{\infty}_u L^{\infty} (S_{u,v} )} \sup_v \int_{\underline{H}_v}  \mathrm{w}^2 (u) v^2 | \underline{\beta} | | \slashed{\nabla} \underline{\hat{\chi}} | + \| ( v \slashed{\nabla} ) \hat{\chi} \|_{L^1_v L^{\infty}_u L^{\infty} (S_{u,v} )} \sup_v \int_{\underline{H}_v} \mathrm{w}^2 (u) v | \underline{\beta} | | \underline{\hat{\chi}} | \\ \lesssim & \| \hat{\chi} \|_{L^1_v L^{\infty}_u L^{\infty} (S_{u,v} )} \sup_v \left( \int_{\underline{H}_v} \mathrm{w}^2 (u) v^2 | \underline{\beta} |^2 \right)^{1/2}  \left( \int_{\underline{H}_v} \mathrm{w}^2 (u) v^2 | (v \slashed{\nabla} ) \underline{\hat{\chi}} |^2 \right)^{1/2} \\ & + \| ( v \slashed{\nabla} ) \hat{\chi} \|_{L^1_v L^{\infty}_u L^{\infty} (S_{u,v} )} \sup_v \left( \int_{\underline{H}_v} \mathrm{w}^2 (u) v^2 | \underline{\beta} |^2 \right)^{1/2}  \left( \int_{\underline{H}_v} \mathrm{w}^2 (u) v^2 | \underline{\hat{\chi}} |^2 \right)^{1/2} ,
\end{align*}
 and all the terms in the last expression can be bounded by the bootstrap assumptions.
 
For the $B_5$ term we have that
\begin{align*}
B_5 = & \int_{D_{u,v}} \mathrm{w}^2 (u) v^2 \left\langle \underline{\beta} , -\frac{3}{2} [ \underline{\eta} ( \hat{\chi} \cdot \hat{\underline{\chi}} ) -^{*} \underline{\eta} ( \hat{\chi} \wedge \hat{\underline{\chi}} ) ] \right\rangle \\ \lesssim & \int_{D_{u,v}} \mathrm{w}^2 (u) \frac{v^2}{v^2} | \underline{\beta} |^2 +  \int_{D_{u,v}} \mathrm{w}^2 (u) v^4 | \underline{\eta} |^2 | \hat{\chi} |^2 | \hat{\underline{\chi}} |^2 \\ \lesssim & \epsilon \sup_v \int_{\underline{H}_v} | \underline{\eta} |^2 + \left( \int_v \sup_u v^4 | \hat{\chi} |^2 | \underline{\eta} |^2 \right) \times \sup_v\left( \int_u \int_{S_{u,v}} \mathrm{w}^2 (u) | \hat{\underline{\chi}} |^2 \right) \\ \lesssim & \epsilon \sup_v \int_{\underline{H}_v} | \underline{\eta} |^2 + \sum_{k=0}^2 \left( \int_v \sup_u  | \hat{\chi} |^2 \| v \underline{\eta} \|_{L^2 (S_{u,v} )}^2 \right) \times \sup_v\left( \int_u \int_{S_{u,v}} \mathrm{w}^2 (u) | \hat{\underline{\chi}} |^2 \right) \\ \lesssim & \epsilon \sup_v \int_{\underline{H}_v} | \underline{\eta} |^2 + \sum_{k=0}^2   \left\| \frac{1}{\sqrt{v}} \hat{\chi} \right\|_{L^2_v L^{\infty}_u L^2 (S_{u,v} )}^2 \| v \underline{\eta} \|_{L^{\infty}_{u,v} L^2 (S_{u,v} )}^2 \times \sup_v\left( \int_u \int_{S_{u,v}} \mathrm{w}^2 (u) | \hat{\underline{\chi}} |^2 \right) ,
\end{align*} 
and the first term of the last expression can be absorbed by the left-hand side of \eqref{boot:riemann2} while the second one can be bounded by the bootstrap assumptions.

For the $B_7$ term we have that
\begin{align*}
B_7 = & \int_{D_{u,v}} \mathrm{w}^2 (u) v^2 \left\langle \underline{\beta} , -\frac{1}{4} \left[ \slashed{\nabla} \left( \slashed{\mathrm{tr}} \chi - \frac{2}{v} \right) \frac{2}{v} + \slashed{\nabla} \left( \slashed{\mathrm{tr}} \underline{\chi} + \frac{2}{v} \right) \frac{2}{v} \right] \right\rangle \\ \lesssim & \int_{D_{u,v}} \mathrm{w}^2 (u) \frac{v^2}{v^2} | \underline{\beta} |^2 +  \int_{D_{u,v}} \mathrm{w}^2 (u) \left| ( v \slashed{\nabla} ) \left( \slashed{\mathrm{tr}} \chi - \frac{2}{v} \right) \right|^2 + \int_{D_{u,v}} \mathrm{w}^2 (u) \left| ( v \slashed{\nabla} ) \left( \slashed{\mathrm{tr}} \underline{\chi} + \frac{2}{v} \right) \right|^2 \\ \lesssim & \epsilon \sup_v \int_{\underline{H}_v} \mathrm{w}^2 (u) v^2 | \underline{\beta} |^2 + \sup_u \int_v \int_{S_{u,v}} \left| ( v \slashed{\nabla} ) \left( \slashed{\mathrm{tr}} \chi - \frac{2}{v} \right) \right|^2  + \epsilon \sup_v \int_u \int_{S_{u,v}} \mathrm{w}^2 (u) v^2 \left| ( v \slashed{\nabla} ) \left( \slashed{\mathrm{tr}} \underline{\chi} + \frac{2}{v} \right) \right|^2\\ \lesssim & \epsilon \sup_v \int_{\underline{H}_v} \mathrm{w}^2 (u) v^2 | \underline{\beta} |^2 + \epsilon \left\| v ( v \slashed{\nabla} ) \left( \slashed{\mathrm{tr}} \chi - \frac{2}{v} \right) \right\|_{L^2 (S_{u,v} )}^2  + \epsilon \sup_v \int_u \int_{S_{u,v}} \mathrm{w}^2 (u) v^2 \left| ( v \slashed{\nabla} ) \left( \slashed{\mathrm{tr}} \underline{\chi} + \frac{2}{v} \right) \right|^2,  
\end{align*}
where the first term of the last expression can be absorbed by the left-hand side of \eqref{boot:riemann2} while the second one can be bounded by the bootstrap assumptions. Note that we used here the improved decay of $\slashed{\mathrm{tr}} \chi - \frac{2}{v}$ compared to $\hat{\chi}$. The weighted $L^1 L^{\infty}_u L^2 (S_{u,v} )$ estimates analogous to the ones for $\hat{\chi}$, although they hold true in the case of $\slashed{\mathrm{tr}} \chi - \frac{2}{v}$, cannot be used in this situation as the term is linear in $\slashed{\mathrm{tr}} \chi - \frac{2}{v}$ and for the higher 2nd and 3rd derivatives on $\underline{\beta}$ the Sobolev inequalities present an irreparable loss of derivatives. The $B_6$ term can be treated similarly and it has better decay in $v$, and so are the $B_8$ and $B_9$ term.

For the $B_{10}$ term we have that
\begin{align*}
B_{10} = &  \int_{D_{u,v}} \mathrm{w}^2 (u) v^2 \langle \underline{\beta} , 2 \omega \underline{\beta} \rangle  \\ \lesssim & \int_{D_{u,v}} \mathrm{w}^2 (u) \frac{v^2}{v^2} | \underline{\beta} |^2 + \int_{D_{u,v}} \mathrm{w}^2 (u) v^4 |\omega |^2 | \underline{\beta} |^2  \\ \lesssim & \sup_v \int_{\underline{H}_v} \mathrm{w}^2 (u) v^2 | \underline{\beta} |^2 + \left( \int_v \sup_u v^2 | \omega |^2 \right) \times \sup_v \left( \int_u \int_{S_{u,v}} \mathrm{w}^2 (u) v^2 |\underline{\beta} |^2 \right)  \\ \lesssim & \sup_v \int_{\underline{H}_v} \mathrm{w}^2 (u) v^2 | \underline{\beta} |^2 + \sum_{k=0}^2 \| ( v \slashed{\nabla} )^k \omega \|_{L^2_v L^{\infty}_u L^2 (S_{u,v} )}^2 \sup_v \left( \int_u \int_{S_{u,v}} \mathrm{w}^2 (u) v^2 |\underline{\beta} |^2 \right) ,  
\end{align*}
where we used Sobolev's inequality \eqref{sobolev3}, and as before the first term of the last expression can be absorbed by the left-hand side of \eqref{boot:riemann2} while the second one can be bounded by the bootstrap assumptions. Note that the $v$-weight in $\omega$ is one power better than required.

For the higher derivatives
$$ \int_{\underline{H}_v} \mathrm{w}^2 (u) v^2 | (v \slashed{\nabla} ) \underline{\beta} |^2 , \quad \int_{\underline{H}_v} \mathrm{w}^2 (u) v^2 | ( v \slashed{\nabla} )^2 \underline{\beta} |^2 , \mbox{  and  }  \int_{\underline{H}_v} \mathrm{w}^2 (u) v^2 | ( v \slashed{\nabla} )^3 \underline{\beta} |^2 , $$
we can work similarly as above using Sobolev's inequality \eqref{sobolev3} when needed, apart from the case of the term $B_4$. For $B_4$ the case of one angular derivative can be treated similarly to the uncommuted case, but for the second and third derivatives cannot be treated in the same way. We demonstrate how to deal with the case of the highest 3rd derivative:
\begin{align*}
\int_{D_{u,v}} \mathrm{w}^2 (u) v^2 \langle & (v \slashed{\nabla} )^3 \underline{\beta} ,  \frac{1}{v} (v \slashed{\nabla} )^4 ( \hat{\chi} \cdot \hat{\underline{\chi}} ) \rangle \\ \simeq  & \int_{D_{u,v}} \mathrm{w}^2 (u) v^2 \langle (v \slashed{\nabla} )^3 \underline{\beta} , \frac{1}{v} \left( [ (v \slashed{\nabla} )^4  \hat{\chi} ]  \hat{\underline{\chi}} +[ (v \slashed{\nabla} )^3  \hat{\chi} ]  [ ( v \slashed{\nabla} )\hat{\underline{\chi}} ] +  [ (v \slashed{\nabla} )^2  \hat{\chi} ]  [ ( v \slashed{\nabla} )^2\hat{\underline{\chi}} ] + [ (v \slashed{\nabla} )  \hat{\chi} ]  [ ( v \slashed{\nabla} )^3 \hat{\underline{\chi}} ] + \hat{\chi} [ ( v \slashed{\nabla} )^4 ] \hat{\underline{\chi}} \right) \rangle .
\end{align*}
The last three terms of the last expression can be treated as in the uncommuted case. From the other two terms we show how to deal with the very first one (the other one is similar), we have that
\begin{align*}
\int_{D_{u,v}} \mathrm{w}^2 (u) v^2 \langle (v \slashed{\nabla} )^3 \underline{\beta} , \frac{1}{v} \left( [ (v \slashed{\nabla} )^4  \hat{\chi} ]  \hat{\underline{\chi}} \right) \rangle \lesssim & \int_{D_{u,v}} \mathrm{w}^2 (u) v | ( v \slashed{\nabla} )^3 \underline{\beta} | | ( v \slashed{\nabla} )^4 \hat{\chi} | | \hat{\underline{\chi}} | \\ \lesssim & \int_{D_{u,v}} \mathrm{w}^2 (u) v | ( v \slashed{\nabla} )^3 \underline{\beta} | \left| ( v \slashed{\nabla} )^4 \left( \slashed{\mathrm{tr}} \chi - \frac{2}{v} \right) \right| | \hat{\underline{\chi}} |\\ & + \int_{D_{u,v}} \mathrm{w}^2 (u) v^2 | ( v \slashed{\nabla} )^3 \underline{\beta} | | ( v \slashed{\nabla} )^3 ( \eta \hat{\chi} ) | | \hat{\underline{\chi}} |   \\ & + \int_{D_{u,v}} \mathrm{w}^2 (u) v^2 | ( v \slashed{\nabla} )^3 \underline{\beta} | | ( v \slashed{\nabla} )^3 ( \underline{\eta} \hat{\chi} ) | | \hat{\underline{\chi}} | \\ & + \int_{D_{u,v}} \mathrm{w}^2 (u) v^2 | ( v \slashed{\nabla} )^3 \underline{\beta} | \left| ( v \slashed{\nabla} )^3 \left( \eta \left( \slashed{\mathrm{tr}} \chi - \frac{2}{v} \right) \right) \right| | \hat{\underline{\chi}} |  \\ & + \int_{D_{u,v}} \mathrm{w}^2 (u) v^2 | ( v \slashed{\nabla} )^3 \underline{\beta} | \left| ( v \slashed{\nabla} )^3 \left( \underline{\eta} \left( \slashed{\mathrm{tr}} \chi - \frac{2}{v} \right) \right) \right| | \hat{\underline{\chi}} | \\ & + \int_{D_{u,v}} \mathrm{w}^2 (u) v | ( v \slashed{\nabla} )^3 \underline{\beta} | \left( | (v \slashed{\nabla} )^3 \eta | + | ( v \slashed{\nabla} )^3 \underline{\eta} | \right)  | \hat{\underline{\chi}} |  \\ & + \int_{D_{u,v}} \mathrm{w}^2 (u) v^2 | ( v \slashed{\nabla} )^3 \underline{\beta} | | ( v \slashed{\nabla} )^3 \beta |  | \hat{\underline{\chi}} | \\ \doteq & b1 + b2 + b3 + b4 + b5 + b6 + b7 ,
\end{align*}
where we used the elliptic equation \eqref{eq:codchi}. Let us consider these terms one by one again. We have that
\begin{align*}
b1 = & \int_{D_{u,v}} \mathrm{w}^2 (u) v | ( v \slashed{\nabla} )^4 \underline{\beta} | \left| ( v \slashed{\nabla} )^4 \left( \slashed{\mathrm{tr}} \chi - \frac{2}{v} \right) \right| | \hat{\underline{\chi}} | \\ \lesssim & \int_{D_{u,v}} \mathrm{w}^2 (u) \frac{v^2}{v^2} | ( v \slashed{\nabla} )^3 \underline{\beta} |^2 + \int_{D_{u,v}} \mathrm{w}^2 (u) v^2 | \left| ( v \slashed{\nabla} )^4 \left( \slashed{\mathrm{tr}} \chi - \frac{2}{v} \right) \right|^2 | \hat{\underline{\chi}} |^2 \\ \lesssim & \epsilon \sup_v \int_{\underline{H}_v} \mathrm{w}^2 (u) v^2 | ( v \slashed{\nabla} )^3 \underline{\beta} |^2 + \left( \int_u \sup_v \mathrm{w}^2 (u) v^2 | \hat{\underline{\chi}} |^2 \right) \cdot \sup_u \left( \int_v \int_{S_{u,v}} \left| ( v \slashed{\nabla} )^4 \left( \slashed{\mathrm{tr}} \chi - \frac{2}{v} \right) \right|^2 \right)  \\ \lesssim & \epsilon \sup_v \int_{\underline{H}_v} \mathrm{w}^2 (u) v^2 | ( v \slashed{\nabla} )^3 \underline{\beta} |^2 + \sum_{k=0}^2 \left( \int_u \sup_v \| \mathrm{w} (u)   \hat{\underline{\chi}} \|_{L^2 (S_{u,v})}^2 \right) \cdot \sup_u \left( \int_v \int_{S_{u,v}} \left| ( v \slashed{\nabla} )^4 \left( \slashed{\mathrm{tr}} \chi - \frac{2}{v} \right) \right|^2 \right)  ,
\end{align*}
and now the first term can be absorbed by the left-hand side of \eqref{boot:riemann2} (where we take 3 angular derivatives), while the last one can be bounded by the bootstrap assumptions.

For the $b2$ term we have that
\begin{align*}
b2 = & \int_{D_{u,v}} \mathrm{w}^2 (u) v^2 | ( v \slashed{\nabla} )^3 \underline{\beta} | | (v \slashed{\nabla} ) (\eta \hat{\chi} ) | |\hat{\underline{\chi}} | \\ \lesssim & \int_{D_{u,v}} \mathrm{w}^2 (u) \frac{v^2}{v^2} | ( v \slashed{\nabla} )^3 \underline{\beta} |^2 + \int_{D_{u,v}} \mathrm{w}^2 (u) v^4| (v \slashed{\nabla} ) (\eta \hat{\chi} ) |^2 |\hat{\underline{\chi}} |^2  \\ \lesssim & \int_{D_{u,v}} \mathrm{w}^2 (u) \frac{v^2}{v^2} | ( v \slashed{\nabla} )^3 \underline{\beta} |^2 + \int_{D_{u,v}} \mathrm{w}^2 (u) v^4 \left( | (v \slashed{\nabla} )^3 (\eta ) |^2 | \hat{\chi}  |^2 + | ( v \slashed{\nabla} )^2 \eta |^2 | ( v \slashed{\nabla} ) \hat{\chi} |^2 +   | ( v \slashed{\nabla} ) \eta |^2 | ( v \slashed{\nabla} )^2 \hat{\chi} |^2 + |  \eta |^2 | ( v \slashed{\nabla} )^3 \hat{\chi} |^2 \right)  |\hat{\underline{\chi}} |^2 \\ \lesssim & \epsilon \sup_v \int_{\underline{H}_v} \mathrm{w}^2 (u) v^2 | ( v \slashed{\nabla} )^3 \underline{\beta} |^2 + \sum_{k_1 + k_2 = 3} \int_{D_{u,v}} \mathrm{w}^2 (u) v^4 | ( v \slashed{\nabla} )^{k_1} \eta |^2 | ( v \slashed{\nabla} )^{k_2} \hat{\chi} |^2 | \hat{\underline{\chi}} |^2  \\ \lesssim & \epsilon \sup_v \int_{\underline{H}_v} \mathrm{w}^2 (u) v^2 | ( v \slashed{\nabla} )^3 \underline{\beta} |^2 + \sum_{k_1 + k_2 = 3} \left( \int_u \sup_v \mathrm{w}^2 (u) v^2 | \hat{\underline{\chi}} |^2 \right) \times \sup_u \left( \int_v \int_{S_{u,v}} v^2 | ( v \slashed{\nabla} )^{k_1} \eta |^2 | ( v \slashed{\nabla} )^{k_2} \hat{\chi} |^2 \right) \\ \lesssim & \epsilon \sup_v \int_{\underline{H}_v} \mathrm{w}^2 (u) v^2 | ( v \slashed{\nabla} )^3 \underline{\beta} |^2 \\ & + \sum_{0 \leq k \leq 2 , 0 \leq m_1 , m_2 \leq 3 } \left( \int_u \sup_v \| \mathrm{w} (u)  ( v \slashed{\nabla} )^k \hat{\underline{\chi}} \|_{L^2 (S_{u,v} )}^2 \right) \times \| v ( v \slashed{\nabla} )^{m_1} \eta \|_{L^{\infty}_{u,v} L^2 (S_{u,v} )}^2 \sup_u \left( \int_v \int_{S_{u,v}} \frac{1}{v^2} | ( v \slashed{\nabla} )^{m_2} \hat{\chi} |^2 \right) ,
\end{align*}
where once again the first term can be absorbed by the left-hand side of \eqref{boot:riemann2} (with 3 angular derivatives) and the second term can be absorbed by the boundedness assumptions, where we used Sobolev's inequality on both the $\eta$ and $\hat{\chi}$ terms if they were carrying only one angular derivative. The $b3$ term can be treated similarly by treating $\underline{\eta}$ in the same way that $\eta$ was treated above.

The terms $b4$ and $b5$ are similar to $b1$ but with better decay in $v$, requiring the use of Sobolev's inequality \eqref{sobolev3} as in $b2$. For the $b6$ term we examine the part that contains $\eta$ (the other part that contains $\underline{\eta}$ can be treated similarly) and we have that
\begin{align*}
\int_{D_{u,v}} \mathrm{w}^2 (u) v | ( v \slashed{\nabla} )^3 \underline{\beta} | | (v \slashed{\nabla} )^3 \eta| |\hat{\underline{\chi}} | \lesssim & \int_{D_{u,v}} \mathrm{w}^2 (u) \frac{v^2}{v^2} | ( v \slashed{\nabla} )^3 \underline{\beta} |^2 + \int_{D_{u,v}} \mathrm{w}^2 (u) v^2 | ( v \slashed{\nabla} )^3 \eta |^2 | \hat{\underline{\chi}} |^2 \\ \lesssim & \epsilon \sup_v \int_{\underline{H}_v} \mathrm{w}^2 (u) v^2 | ( v \slashed{\nabla} )^3 \underline{\beta} |^2 + \left( \int_u \sup_v v^2 |\hat{\underline{\chi}} |^2 \right) \times \sup_u \left( \int_v \int_{S_{u,v}} | ( v \slashed{\eta} )^3 \eta |^2 \right) \\ \lesssim & \epsilon \sup_v \int_{\underline{H}_v} \mathrm{w}^2 (u) v^2 | ( v \slashed{\nabla} )^3 \underline{\beta} |^2 + \sum_{k=0}^2 \left( \int_u \sup_v  \|( v \slashed{\nabla} )^k \hat{\underline{\chi}} \|_{L^2 (S_{u,v} )}^2 \right) \times \sup_u \left( \int_v \int_{S_{u,v}} | ( v \slashed{\eta} )^3 \eta |^2 \right) ,
\end{align*}
where we used Sobolev's inequality \eqref{sobolev3}, and again the first term can be absorbed by the left-hand side of \eqref{boot:riemann2} while the second term can be bounded by the bootstrap assumptions.

Finally for the term $b7$ we have that
\begin{align*}
b7 = & \int_{D_{u,v}} \mathrm{w}^2 (u) v^2 | ( v \slashed{\nabla} )^3 \underline{\beta} | | (v \slashed{\nabla} )^3 \beta| |\hat{\underline{\chi}} | \\ \lesssim & \int_{D_{u,v}} \mathrm{w}^2 (u) \frac{v^2}{v^2} | ( v \slashed{\nabla} )^3 \underline{\beta} |^2 + \int_{D_{u,v}} \mathrm{w}^2 (u) v^4 | ( v \slashed{\nabla} )^3 \beta |^2 | \hat{\underline{\chi}} |^2 \\ \lesssim & \sup_v \int_{\underline{H}_v} \mathrm{w}^2 (u) v^2 | ( v \slashed{\nabla} )^3 \underline{\beta} |^2 + \left( \int_u \sup_v v^2 | \hat{\underline{\chi}} |^2 \right) \times \sup_u \left( \int_v \int_{S_{u,v}} v^2 | ( v \slashed{\nabla} )^3 \beta |^2 \right)  \\ \lesssim & \sup_v \int_{\underline{H}_v} \mathrm{w}^2 (u) v^2 | ( v \slashed{\nabla} )^3 \underline{\beta} |^2 + \sum_{k=0}^2 \left( \int_u \sup_v \| ( v \slashed{\nabla} )^k \hat{\underline{\chi}} \|_{L^2 (S_{u,v} )}^2 \right) \times \sup_u \| v ( v \slashed{\nabla} )^3 \beta \|_{L^2 (H_u)}^2 ,
\end{align*} 
where we note that not only the very last term can be bounded by the bootstrap assumptions but the weight in $v$ in the term involving $\beta$ is one power better than required.

We turn to $G_{K-\frac{1}{v^2}}$ and we have that
$$ \int_{D_{u,v}} \mathrm{w}^2 (u) v^2 \left\langle K - \frac{1}{v^2} , G_{K-\frac{1}{v^2}} \right\rangle = C_1 + C_2 + C_3 + C_4 + C_5 + C_6 + C_7 + C_8 . $$

For the $C_1$ term we have that:
\begin{align*}
C_1 = &\int_{D_{u,v}} \mathrm{w}^2 (u) v^2 \left\langle K - \frac{1}{v^2} , \frac{1}{v^2} \mathrm{tr} \underline{\chi} \right\rangle \\ \lesssim & \int_{D_{u,v}} \mathrm{w}^2 (u) v^2 \left| K - \frac{1}{v^2} \right|^2 + \int_{D_{u,v}} \mathrm{w}^2 (u) v^2 \frac{1}{v^4} | \mathrm{tr} \underline{\chi} |^2 \\ \lesssim & \epsilon \sup_u \int_{H_u} \mathrm{w}^2 (u) v^2 \left| K - \frac{1}{v^2} \right|^2 +  \epsilon \sup_v \int_{\underline{H}_v} \mathrm{w}^2 (u) | \mathrm{tr} \underline{\chi} |^2 ,
\end{align*}
where in the last term we achieved smallness in $\epsilon$ by an appropriate choice of $v_0$ and we note that now this term can be bounded by the bootstrap assumption for $\slashed{\mathrm{tr}} \underline{\chi} + \frac{2}{v}$. 

For the $C_2$ term we have that
\begin{align*}
C_2 = & \int_{D_{u,v}} \mathrm{w}^2 (u) v^2 \left\langle K - \frac{1}{v^2} , - ( \zeta - 2 \eta ) \underline{\beta} \right\rangle \\ \lesssim & \int_{D_{u,v}} \mathrm{w}^2 (u) v^2 \left| K - \frac{1}{v^2} \right|^2 + \int_{D_{u,v}} \mathrm{w}^2 (u) v^2 ( | \eta |^2 + | \underline{\eta} |^2 ) | \underline{\beta} |^2 \\ \lesssim & \epsilon \sup_u \int_{H_u} \mathrm{w}^2 (u) v^2 \left| K - \frac{1}{v^2} \right|^2 + \left( \int_v \sup_u ( | \eta |^2 + | \underline{\eta} |^2 ) \right) \times  \sup_v \left( \int_u \int_{S_{u,v}} \mathrm{w}^2 (u) v^2 | \underline{\beta} |^2 \right)  \\ \lesssim & \epsilon \sup_u \int_{H_u} \mathrm{w}^2 (u) v^2 \left| K - \frac{1}{v^2} \right|^2 +   \epsilon \sum_{k=0}^2 ( \| v (v \slashed{\nabla} )^k \eta \|_{L^{\infty}_{u,v} L^2 (S_{u,v} )}^2  + \| v (v \slashed{\nabla} )^k \underline{\eta} \|_{L^{\infty}_{u,v} L^2 (S_{u,v} )}^2 ) \times \sup_v \left( \int_u \int_{S_{u,v}} \mathrm{w}^2 (u) v^2 | \underline{\beta} |^2 \right) ,
\end{align*}
where for the last term (which can be bounded by the bootstrap assumptions) we used Sobolev's inequality \eqref{sobolev3} and the smallness in $\epsilon$ was achieved by the choice of $v_0$.

For the $C_3$ term we have that
\begin{align*}
C_3 = & \int_{D_{u,v}} \mathrm{w}^2 (u) v^2 \left\langle K - \frac{1}{v^2} , \frac{1}{2} \hat{\underline{\chi}} \cdot ( \slashed{\nabla} \hat{\otimes} \eta ) \right\rangle \\ \lesssim &\int_{D_{u,v}} \frac{v^3}{v^2} \left| K - \frac{1}{v^2} \right|^2 + \int_{D_{u,v}} \mathrm{w}^4 (u) v  | \hat{\underline{\chi}} |^2 | | ( v \slashed{\nabla} ) \eta |^2 \\ \lesssim & \epsilon \sup_v \int_{\underline{H}_v}  v^3 \left| K - \frac{1}{v^2} \right|^2 + \sum_{k=0}^2 \int_{D_{u,v}}  \mathrm{w}^4 (u) \frac{1}{v} \|  ( v \slashed{\nabla} )^k \hat{\underline{\chi}} \|_{L^2 (S_{u,v} )}^2 | ( v \slashed{\nabla} ) \eta |^2 \\ \lesssim & \epsilon \sup_v \int_{\underline{H}_v}  v^3 \left| K - \frac{1}{v^2} \right|^2 + \left( \sup_u \int_v \int_{S_{u,v}}   \mathrm{w}^2 (u) |  ( v \slashed{\nabla} ) \eta |^2 \right) \times  \sum_{k=0}^2 \left( \int_u \mathrm{w}^2 (u) \sup_v \| ( v \slashed{\nabla} )^k \hat{\underline{\chi}} \|_{L^2 (S_{u,v} )}^2 \right) ,
\end{align*}
where we used Sobolev's inequality \eqref{sobolev3} on $\hat{\underline{\chi}}$ and the final term can be bounded by the bootstrap assumptions. Note that for higher derivatives we will need to use Sobolev's inequality on either $\eta$ or $\hat{\underline{\chi}}$ depending on which part is carrying the fewer number of derivatives. The $C_4$ term can be treated similarly and it has better decay in $v$ as it is quadratic in $\eta$.

For the $C_5$ term we have that
\begin{align*}
C_5 = & \int_{D_{u,v}} \mathrm{w}^2 (u) v^2 \left\langle K - \frac{1}{v^2} , -\frac{1}{2} \left( \slashed{\mathrm{tr}} \underline{\chi} + \frac{2}{v} \right) \slashed{\mathrm{div}} \eta \right\rangle \\ \lesssim & \int_{D_{u,v}} \frac{v^3}{v^2} \left| K - \frac{1}{v^2} \right|^2 + \int_{D_{u,v}} \mathrm{w}^4 (u) v \left| \slashed{\mathrm{tr}} \underline{\chi} + \frac{2}{v} \right|^2 | | ( v \slashed{\nabla} ) \eta |^2 \\ \lesssim & \epsilon \sup_v \int_{\underline{H}_v} v^3 \left| K - \frac{1}{v^2} \right|^2 + \sum_{k=0}^2 \int_{D_{u,v}} \mathrm{w}^4 (u) \frac{1}{v} \left\| ( v \slashed{\nabla} )^k \left( \slashed{\mathrm{tr}} \underline{\chi} + \frac{2}{v} \right) \right\|_{L^2 (S_{u,v})}^2 | ( v \slashed{\nabla} ) \eta |^2  \\ \lesssim & \epsilon \sup_v \int_{\underline{H}_v} v^3 \left| K - \frac{1}{v^2} \right|^2 + \sum_{k=0}^2 \left( \sup_v \int_u \mathrm{w}^2 (u) \left\| v (v \slashed{\nabla} )^k \slashed{\mathrm{tr}} \underline{\chi} + \frac{2}{v} \right\|_{L^2 (S_{u,v} )}^2 \right) \times \left( \sup_u \int_v \int_{S_{u,v}} \mathrm{w}^2 (u)  |  ( v \slashed{\nabla} ) \eta |^2 \right)  ,
\end{align*}
where we used Sobolev's inequality \eqref{sobolev3} on $\slashed{\mathrm{tr}} \underline{\chi} + \frac{2}{v}$. Note that for higher derivatives we need to use Sobolev's inequality \eqref{sobolev3} on either $\eta$ or $\slashed{\mathrm{tr}} \underline{\chi} + \frac{2}{v}$ depending on which part has fewer angular derivatives. The $C_7$ term can be treated similarly and it has better decay in $v$ as it is quadratic in $\eta$. 

For the $C_6$ term we have that
\begin{align*}
C_6 = & \int_{D_{u,v}} \mathrm{w}^2 (u) v^2 \left\langle K - \frac{1}{v^2} , \frac{1}{v} \slashed{\mathrm{div}} \eta \right\rangle \\ \lesssim & \int_{D_{u,v}} \mathrm{w}^2 (u) v^2 \left| K - \frac{1}{v^2} \right|^2 + \int_{D_{u,v}} \mathrm{w}^2 (u) | \slashed{\mathrm{div}} \eta |^2 \\ \lesssim & \epsilon \sup_u \int_{H_u} \mathrm{w}^2 (u) v^2 \left| K - \frac{1}{v^2} \right|^2 + \epsilon \sup_u \int_v \int_{S_{u,v}} \mathrm{w}^2 (u) | ( v \slashed{\nabla} ) \eta |^2 ,
\end{align*}
where for both terms the smallness in $\epsilon$ can be achieved by the choice of $U$. The term $C_8$ can be treated similarly and it has better decay in $v$ as it is quadratic in $\eta$. 

The higher derivatives for $K-\frac{1}{v^2}$ can be treated similarly to the terms above using Sobolev's inequality appropriately. Moreover, we can deal in a similar way with $G_{\check{\sigma}}$.

We note that the above estimates along with the boundedness of
$$ \int_{\underline{H}_v} \mathrm{w}^2 (u) v^2 | ( v \slashed{\nabla} )^k \underline{\beta} |^2 , $$
for $k \in \{0,1,2,3 \}$ that was shown before, imply also (through the use of equation \eqref{eq:rgauss3}) that 
\begin{equation}\label{bound:gauss}
v^2 \left\| ( v \slashed{\nabla} )^m \left( K - \frac{1}{v^2} \right) \right\|_{L^2 (S_{u_2 ,v} )} \lesssim v^2 \left\| ( v \slashed{\nabla} )^m \left( K - \frac{1}{v^2} \right) \right\|_{L^2 (S_{u_1 ,v} )}+ \epsilon^{1/2} G_I ,
\end{equation} 
for any $u_0 \leq u_1 < u_2 \leq U$, where $G_I$ consists of terms bounded by the bootstrap assumptions.

Finally (for this section) we derive estimates for the renormalized quantities:
$$ \mu = K - \frac{1}{v^2} + \slashed{\mathrm{div}} \eta , \quad \underline{\mu} = K - \frac{1}{v^2} - \slashed{\mathrm{div}} \underline{\eta} . $$
For both of them we will use the $\slashed{\nabla}_3$ equations, as we note that from both $\slashed{\nabla}_3 \mu$ and $\slashed{\nabla}_3 \underline{\mu}$ the term $\slashed{\mathrm{div}} \underline{\beta}$ (coming from $\slashed{\nabla}_3 \left( K - \frac{1}{v^2} \right)$) is absent. Schematically we have the following equations:
\begin{equation}\label{eq:mu3}
\slashed{\nabla}_3 \mu = G_{K-\frac{1}{v^2}} + G_{\mu} ,
\end{equation}
\begin{equation}\label{eq:umu3}
\slashed{\nabla}_3 \mu = G_{K-\frac{1}{v^2}} + G_{\underline{\mu}} ,
\end{equation} 
and for 
$$ \mathrm{\mu} \in \{ \mu , \underline{\mu} \} , $$
the term $G_{\mathrm{\mu}}$ schematically has the form 
$$ G_{\mathrm{\mu}} = - \slashed{\mathrm{tr}} \underline{\chi} \left( K - \frac{1}{v^2} \right) + \psi_{\underline{H}} \slashed{\nabla} \psi + \psi \underline{\beta} + \psi \psi \psi_{\underline{H}} ,$$
where
$$ \psi \in \{ \eta , \underline{\eta} \} , \quad \psi_{\underline{H}} \in \{  \hat{\underline{\chi}} , \slashed{\mathrm{tr}} \underline{\chi} \} . $$
So for $\mathrm{\mu}$ we have that
\begin{align*}
\| ( v \slashed{\nabla} )^k \mathrm{\mu} \|_{L^2 (S_{u,v})} \lesssim & \| ( v \slashed{\nabla} )^k \mathrm{\mu} \|_{L^2 (S_{u_0 ,v})} + \int_{u_0}^u \left( \| G_{K-\frac{1}{v^2}} \|_{L^2 (S_{u' , v} )} +  \| G_{\mu} \|_{L^2 (S_{u' , v} )} \right)  \\ \lesssim & \| ( v \slashed{\nabla} )^k \mathrm{\mu} \|_{L^2 (S_{u_0 ,v})} + \int_{u_0}^u ( m_1 + m_2 + m_3 + m_4 + m_5 + m_6 ) , 
\end{align*} 
for $k \in \{0,1,2,3\}$, where
$$ m_1 \simeq \left\| ( v \slashed{\nabla} )^k \left[ \slashed{\mathrm{tr}} \underline{\chi} \left( K - \frac{1}{v^2} \right) \right] \right\|_{L^2 (S_{u' , v})} , \quad m_2 \simeq \left\| ( v \slashed{\nabla} )^k \left( \frac{1}{v^2} \slashed{\mathrm{tr}} \underline{\chi} \right) \right\|_{L^2 (S_{u' , v} )} , \quad m_3 \simeq \| ( v \slashed{\nabla} )^k ( \psi_{\underline{H}} \slashed{\nabla} \psi ) \|_{L^2 (S_{u' , v})} , $$ $$  m_4 \simeq \| ( v \slashed{\nabla} )^k ( \psi \psi \psi_{\underline{H}} ) \|_{L^2 (S_{u' , v} )} , \quad m_5 \simeq \| ( v \slashed{\nabla} )^k ( \psi \underline{\beta} ) \|_{L^2 (S_{u' , v} )} , $$
and $m_6$ is the $L^2 (S)$ norm of higher order terms. We examine the above terms one by one and we have that
\begin{align*}
\int_{u_0}^u m_1 \simeq  & \sum_{k_1 + k_2 = k}  \int_{u_0}^u \left\| [ ( v \slashed{\nabla} )^{k_1} ( \slashed{\mathrm{tr}} \underline{\chi} ) ] \left[ ( v \slashed{\nabla} )^{k_2} \left( K - \frac{1}{v^2} \right) \right] \right\|_{L^2 (S_{u' , v})} \\ \lesssim & \sum_{k_1 + k_2 = k , k_1 < k_2 , 0 \leq m \leq 2}  \int_{u_0}^u \frac{1}{v^2} \left\|  ( v \slashed{\nabla} )^{k_1+m} ( \slashed{\mathrm{tr}} \underline{\chi} ) \right\|_{L^2 (S_{u' , v} )} \left\| v ( v \slashed{\nabla} )^{k_2} \left( K - \frac{1}{v^2} \right)  \right\|_{L^2 (S_{u' , v})} \\ & + \sum_{k_1 + k_2 = k ,  k_2 < k_1, 0\leq m \leq 2}  \int_{u_0}^u \frac{1}{v^2} \left\|  ( v \slashed{\nabla} )^{k_1+2} ( \slashed{\mathrm{tr}} \underline{\chi} ) \right\|_{L^2 (S_{u' , v} )} \left\| v ( v \slashed{\nabla} )^{k_2+m} \left( K - \frac{1}{v^2} \right)  \right\|_{L^2 (S_{u' , v})} \\ \lesssim & \frac{1}{v^2} \sum_{k_1 + k_2 = k ,  k_1 < k_2 , 0 \leq m \leq 2}  \left( \int_{u_0}^u  \mathrm{w}^2 (u' ) \left\|  ( v \slashed{\nabla} )^{k_1+m} ( \slashed{\mathrm{tr}} \underline{\chi} ) \right\|_{L^2 (S_{u' , v} )}^2 \right)^{1/2} \left( \int_{u_0}^u \frac{1}{\mathrm{w}^2 (u' )} \left\| v ( v \slashed{\nabla} )^{k_2} \left( K - \frac{1}{v^2} \right)  \right\|_{L^2 (S_{u' , v})}^2 \right)^{1/2} \\ & + \frac{1}{v^2} \sum_{k_1 + k_2 = k , k_2 < k_1, 0\leq m \leq 2}  \left( \int_{u_0}^u \mathrm{w}^2 (u' ) \left\|  ( v \slashed{\nabla} )^{k_1+2} ( \slashed{\mathrm{tr}} \underline{\chi} ) \right\|_{L^2 (S_{u' , v} )}^2 \right)^{1/2} \left( \int_{u_0}^u \frac{1}{\mathrm{w}^2 (u' )} \left\| v ( v \slashed{\nabla} )^{k_2+m} \left( K - \frac{1}{v^2} \right)  \right\|_{L^2 (S_{u' , v})}^2 \right)^{1/2} ,
\end{align*}
after applying Sobolev's inequality \eqref{sobolev3}. For the $m_2$ term we have that
$$ \int_{u_0}^u m_2 \simeq \frac{1}{v^2} \left( \int_{u_0}^u \mathrm{w}^2 (u' ) \| ( v \slashed{\nabla} )^k \slashed{\mathrm{tr}} \underline{\chi} \|_{L^2 (S_{u' , v} )}^2 \right)^{1/2} \times ( \left( \int_{u_0}^u \frac{1}{\mathrm{w}^2 (u' )} \right)^{1/2}  \lesssim \epsilon^{1/2} \frac{1}{v^2} \left( \int_{u_0}^u \mathrm{w}^2 (u' ) \| ( v \slashed{\nabla} )^k \slashed{\mathrm{tr}} \underline{\chi} \|_{L^2 (S_{u' , v} )}^2 \right)^{1/2} , $$
where for the smallness in $\epsilon$ we used the choice of $U$. For the $m_3$ term we have that
 \begin{align*}
 \int_{u_0}^u m_2 \simeq & \sum_{k_1 + k_2 = k} \int_{u_0}^u \| ( v \slashed{\nabla} )^k ( \psi_{\underline{H}} \slashed{\nabla} \psi ) \|_{L^2 (S_{u' , v} )} \\ \lesssim & \sum_{k_1 + k_2 = k}  \frac{1}{v} \int_{u_0}^u \left\| [ ( v \slashed{\nabla} )^{k_1} \psi_{\underline{H}} ] \left[ ( v \slashed{\nabla} )^{k_2 +1}\psi  \right] \right\|_{L^2 (S_{u' , v})} \\ \lesssim & \\ \lesssim & \sum_{k_1 + k_2 = k , k_1 \leq 1 , 0 \leq m \leq 2}  \int_{u_0}^u \frac{1}{v^3} \left\|  ( v \slashed{\nabla} )^{k_1+m} \psi_{\underline{H}} \right\|_{L^2 (S_{u' , v} )} \left\| v ( v \slashed{\nabla} )^{k_2+1} \psi  \right\|_{L^2 (S_{u' , v})} \\ & + \sum_{k_1 + k_2 = k , k_2 \leq 1, 0\leq m \leq 2}  \int_{u_0}^u \frac{1}{v^3} \left\|  ( v \slashed{\nabla} )^{k_1+2} \psi_{\underline{H}} \right\|_{L^2 (S_{u' , v} )} \left\| v ( v \slashed{\nabla} )^{k_2+m+1} \psi   \right\|_{L^2 (S_{u' , v})} \\ \lesssim & \frac{1}{v^3} \sum_{k_1 + k_2 = k , k_1 < k_2 , 0 \leq m \leq 2}  \left( \int_{u_0}^u  \mathrm{w}^2 (u' ) \left\|  ( v \slashed{\nabla} )^{k_1+m} \psi_{\underline{H}} \right\|_{L^2 (S_{u' , v} )}^2 \right)^{1/2} \left( \int_{u_0}^u \frac{1}{\mathrm{w}^2 (u' )} \left\| v ( v \slashed{\nabla} )^{k_2+1} \psi  \right\|_{L^2 (S_{u' , v})}^2 \right)^{1/2} \\ & + \frac{1}{v^3} \sum_{k_1 + k_2 = k , k_2 < k_1, 0\leq m \leq 2}  \left( \int_{u_0}^u \mathrm{w}^2 (u' ) \left\|  ( v \slashed{\nabla} )^{k_1+2} \psi_{\underline{H}} \right\|_{L^2 (S_{u' , v} )}^2 \right)^{1/2} \left( \int_{u_0}^u \frac{1}{\mathrm{w}^2 (u' )} \left\| v ( v \slashed{\nabla} )^{k_2+m+1} \psi \right\|_{L^2 (S_{u' , v})}^2 \right)^{1/2} ,
 \end{align*}
where once again we used Sobolev's inequality \eqref{sobolev3} (noting that the numerology works out in compatibility with our bootstrap assumptions as $k \in \{ 0,1,2,3 \}$). The $m_4$ term can be treated similarly and we can get that (for any $k \in \{0,1,2,3\}$):
\begin{align*}
\int_{u_0}^u m_4  \lesssim  \frac{1}{v^4} & \left( \sum_{m_1 \leq 3 , m_2 , m_3 \leq 4} \sup_{u' } \| v ( v \slashed{\nabla} )^{m_1} \psi \|_{L^2 (S_{u' , v} )} \right) \times \left(  \int_{u_0}^u \frac{1}{\mathrm{w}^2 (u' )} \| v ( v\slashed{\nabla} )^{m_2} \psi \|_{L^2 (S_{u' , v} )}^2 \right)^{1/2} \times \\ & \times \left( \int_{u_0}^u \mathrm{w}^2 (u' ) \| ( v \slashed{\nabla} )^{m_3} \psi_{\underline{H}} \|_{L^2 (S_{u' , v})}^2 \right)^{1/2} . 
\end{align*}
For the $m_5$ term we have that
\begin{align*}
\int_{u_0}^u m_5 \simeq & \int_{u_0}^u \| ( v \slashed{\nabla} )^k ( \psi \underline{\beta} ) \|_{L^2 (S_{u' , v })} \\ \lesssim & \sum_{k_1 + k_2 = k} \int_{u_0}^u \|[ ( v \slashed{\nabla} )^{k_1}  \psi ] [ (v \slashed{\nabla} )^{k_2}  \underline{\beta} ] \|_{L^2 (S_{u' , v })} \\ \lesssim & \sum_{k_1 + k_2 = k , k_1 \leq 1 , 0 \leq m \leq 2} \int_{u_0}^u \| ( v \slashed{\nabla} )^{k_1+m}  \psi \|_{L^2 (S_{u' , v})} \|  (v \slashed{\nabla} )^{k_2}  \underline{\beta}  \|_{L^2 (S_{u' , v })} \\ & + \sum_{k_1 + k_2 = k , k_2 \leq 1 , 0 \leq m \leq 2} \int_{u_0}^u \| ( v \slashed{\nabla} )^{k_1}  \psi \|_{L^2 (S_{u' , v})} \|  (v \slashed{\nabla} )^{k_2 + m}  \underline{\beta}  \|_{L^2 (S_{u' , v })} \\ \lesssim & \frac{1}{v^3} \sum_{k_1 + k_2 = k , k_1 \leq 1 , 0 \leq m \leq 2} \left( \int_{u_0}^u \frac{1}{\mathrm{w}^2 (u' )} \| v ( v \slashed{\nabla} )^{k_1 +m } \psi \|_{L^2 (S_{u' , v} )}^2 \right) \times \left( \int_{u_0}^u \mathrm{w}^2 (u' ) \| v ( v \slashed{\nabla}  )^{k_2} \underline{\beta} \|_{L^2 (S_{u' , v })}^2 \right) \\ & + \frac{1}{v^3} \sum_{k_1 + k_2 = k , k_2 \leq 1 , 0 \leq m \leq 2} \left( \int_{u_0}^u \frac{1}{\mathrm{w}^2 (u' )} \| v ( v \slashed{\nabla} )^{k_1  } \psi \|_{L^2 (S_{u' , v} )}^2 \right) \times \left( \int_{u_0}^u \mathrm{w}^2 (u' ) \| v ( v \slashed{\nabla}  )^{k_2 + m} \underline{\beta} \|_{L^2 (S_{u' , v })}^2 \right) ,
\end{align*}
where we applied Sobolev's inequality \eqref{sobolev3} noticing that the exponent for $\underline{\beta}$ is always less or equal than 3. The $m_6$ term can be bounded by the same estimates as the previous terms. Gathering together all the above estimates, we can now show the following two estimates:
\begin{equation}\label{mu1}
 \sup_u \int_v \mathrm{w}^2 (u) \| v ( v \slashed{\nabla} )^k \mathrm{\mu} \|_{L^2 (S_{u,v} )}^2 \lesssim\int_v \mathrm{w}^2 (u) \| v ( v \slashed{\nabla} )^k \mathrm{\mu} \|_{L^2 (S_{u_0 ,v} )}^2 + \epsilon M_{b1} , 
 \end{equation}
\begin{equation}\label{mu2}
\sup_v \int_u v^4 \| ( v \slashed{\nabla} )^k \mathrm{\mu} \|_{L^2 (S_{u,v} )}^2 \lesssim \epsilon C +  \epsilon M_{b2} , 
\end{equation}
where $M_{b1}$ consists of terms bounded by the bootstrap assumptions, and $C$ is constant depending on the initial data. To demonstrate estimate \eqref{mu1} we show how we can deal with the worst term (in terms of decay in $v$ and integration in $u$) coming from $m_1$, which we square and integrate in $v$ with a $v^2$ weight and we get that
\begin{align*}
\int_v v^2 \mathrm{w}^2 (u) & \left( \int_{u_0}^u m_1 \right)^2 \\ \lesssim & \sum_{m_1 \leq 3 , m_2 \leq 2} \int_v \frac{1}{v^2} \mathrm{w}^2 (u)  \left[ \left( \int_{u_0}^u \mathrm{w}^2 (u' ) \| ( v \slashed{\nabla} )^{m_1} ( \slashed{\mathrm{tr}} \underline{\chi} ) \|_{L^2 (S_{u' ,v} )}^2 \right) \cdot \left( \int_{u_0}^u  \frac{1}{v^2} \left\| v^{3/2} ( v \slashed{\nabla} )^{m_2} \left( K - \frac{1}{v^2} \right) \right\|_{L^2 (S_{u' ,v} )}^2 \right) \right] \\ \lesssim & \sum_{m_1 , m_2 \leq 3} \sup_v \left( \int_{u_0}^u \mathrm{w}^2 (u' ) \| ( v \slashed{\nabla} )^{m_1} ( \slashed{\mathrm{tr}} \underline{\chi} ) \|_{L^2 (S_{u' ,v} )}^2 \right) \int_{D_{u' ,v}} \frac{1}{\mathrm{w}^2 (u' )} \left| ( v \slashed{\nabla} )^{m_2} \left( K - \frac{1}{v^2} \right) \right|^2 \\ \lesssim & \sum_{m_1 , m_2 \leq 3} \sup_v \left( \int_{u_0}^u \mathrm{w}^2 (u' ) \| ( v \slashed{\nabla} )^{m_1} ( \slashed{\mathrm{tr}} \underline{\chi} ) \sup_v \|_{L^2 (S_{u' ,v} )}^2 \right) \left\| v^2 ( v \slashed{\nabla} )^{m_2} \left( K - \frac{1}{v^2} \right) \right\|_{L^2 (\underline{H}_v )}^2 ,
\end{align*} 
where we used the bound \eqref{bound:gauss} and we can get smallness in $\epsilon$ by the choice of $v_0$. Note that we can always replace $\slashed{\mathrm{tr}} \underline{\chi}$ with $\slashed{\mathrm{tr}} \underline{\chi} + \frac{2}{v}$ when this term is hit by an angular derivative, so the resulting term can be bounded by the bootstrap assumptions (although as it has been noted before bounds for $\slashed{\mathrm{tr}} \underline{\chi} + \frac{2}{v}$ imply bounds for $\slashed{\mathrm{tr}} \underline{\chi}$, the bounds for the latter are of fixed size and not dependent on the initial data though). Finally note that we used the fact that $\mathrm{w}$ is decreasing so that
$$ \sup_{u' \in [u_0 , u]} \frac{1}{\mathrm{w}^2 (u')} \leq \frac{1}{\mathrm{w}^2 (u)} . $$
Note that the rest of the terms are easier to deal with. Moreover it should be noted that from $m_1$ the worst term comes from $\slashed{\mathrm{tr}} \underline{\chi} ( v \slashed{\nabla} )^3 \left( K - \frac{1}{v^2} \right)$. Working similarly we also get the estimate \eqref{mu2}. Finally we note that for $l \in \{0,1,2\}$ we have the pointwise estimates:
$$ \| v^2 ( v \slashed{\nabla} )^l \mathrm{\mu} \|_{L^{\infty}_{u,v} L^2 (S_{u,v} )} \lesssim C + \epsilon M_l , $$
for $C$ a constant depending on the initial data and $M_l$ a term bounded by the bootstrap assumptions.

\textbf{Ricci coefficients, 1. $\hat{\chi}$, $\slashed{\mathrm{tr}} \chi - \frac{2}{v}$ and $\omega$:} For all these quantities we use the $\slashed{\nabla}_3$ equations \eqref{eq:chi3}, \eqref{eq:trchi3} and \eqref{eq:omega3}, for up to 3 angular derivatives. 

We start with $\hat{\chi}$ and we have that for the $L^1$ norm:
\begin{align*}
\| \hat{\chi} \|_{L^2 (S_{u_2 ,v} )} \lesssim & \| \hat{\chi} \|_{L^2 (S_{u_1 , v} )} + \int_{u_1}^{u_2} \| \slashed{\mathrm{tr}} \underline{\chi} \hat{\chi} \|_{L^2 (S_{u,v} )} \, du + \int_{u_1}^{u_2} \| \slashed{\nabla} \eta \|_{L^2 (S_{u,v} )} \, du  \\ & + \int_{u_1}^{u_2} \| \slashed{\mathrm{tr}} \chi \underline{\hat{\chi}} \|_{L^2 (S_{u,v} )} \, du + \int_{u_1}^{u_2} \| \eta \hat{\otimes} \eta \|_{L^2 (S_{u,v} )} \, du \\ = & \| \hat{\chi} \|_{L^2 (S_{u_1 , v} )} + R_0 + R_1 + R_2 + R_3 .
\end{align*}
We examine the terms one by one. For the $R_0$ term we have that
\begin{align*}
\int_{u_1}^{u_2} \| \slashed{\mathrm{tr}} \underline{\chi} \hat{\chi} \|_{L^2 (S_{u,v} )} \, du \lesssim & \sum_{k=0}^2 \int_{u_1}^{u_2} \frac{1}{v} \| ( v \slashed{\nabla} )^k \slashed{\mathrm{tr}} \underline{\chi} \|_{L^2 (S_{u,v} )}  \| \hat{\chi} \|_{L^2 (S_{u,v} )} \, du \\ \lesssim & \frac{1}{v} \sum_{k=0}^2 \left( \int_{u_1}^{u_2} \mathrm{w}^2 (u) \| ( v \slashed{\nabla} )^k \slashed{\mathrm{tr}} \underline{\chi} \|_{L^2 (S_{u,v} )}^2 \,  du \right)^{1/2} \left( \int_{u_1}^{u_2} \frac{1}{\mathrm{w}^2 (u)} \| \hat{\chi} \|_{L^2 (S_{u,v} )}^2 \,  du \right)^{1/2} \\ \lesssim & \epsilon^{1/2} \frac{1}{v} \sup_u \| \hat{\chi} \|_{L^2 (S_{u,v} )} \sum_{k=0}^2 \left( \int_{u_1}^{u_2} \mathrm{w}^2 (u) \| ( v \slashed{\nabla} )^k \slashed{\mathrm{tr}} \underline{\chi} \|_{L^2 (S_{u,v} )}^2 \,  du \right)^{1/2} ,
\end{align*}
which implies that $\int_v \frac{1}{v} R_0$ and $\int_v \frac{1}{v} ( R_0 )^2$ can be bounded by the bootstrap assumptions. For the $R_1$ term we have that
\begin{align*}
R_1 = & \int_{u_1}^{u_2} \|  \slashed{\nabla}  \eta \|_{L^2 (S_{u,v} )} \, du =  \int_{u_1}^{u_2} \frac{1}{v^2} \| v (v \slashed{\nabla} ) \eta \|_{L^2 (S_{u,v} )} \, du \\ \lesssim &  \frac{1}{v^2} \left( \int_{u_1}^{u_2} \frac{1}{\mathrm{w}^2 (u)} \right)^{1/2} \left( \int_{u_1}^{u_2} \mathrm{w}^2 (u) \| v (v \slashed{\nabla} ) \eta \|_{L^2 (S_{u,v} )}^2 \right)^{1/2} \lesssim \epsilon^{1/2}\frac{1}{v^2} \left( \int_{u_1}^{u_2} \mathrm{w}^2 (u) \| v (v \slashed{\nabla} ) \eta \|_{L^2 (S_{u,v} )}^2 \right)^{1/2}  .
\end{align*}
For the $L^1_v L^2 (S)$ norm of $\hat{\chi}$ we have that
\begin{align*}
\int_v \frac{1}{v} R_1 \lesssim &  \epsilon^{1/2} \int_v \frac{1}{v^3} \left( \int_{u_1}^{u_2} \| v (v \slashed{\nabla} ) \eta \|_{L^2 (S_{u,v} )}^2 \right)^{1/2} \\ \lesssim & \epsilon^{1/2} \left( \int_v \frac{1}{v^4} \right)^{1/2} \left( \int_{u_1}^{u_2} \int_v \| (v \slashed{\nabla} ) \eta \|_{L^2 (S_{u,v})}^2 \right)^{1/2} ,
\end{align*}
and the last term can be bounded by the bootstrap assumption. The case of the $L^2_v L^2 (S)$ norm of $\hat{\chi}$ can be treated in the same way. Note that the above estimate works for all derivatives up to 3 for $R_1$, as the derivatives of $\eta$ are estimated on an $H_u$ hypersurface. For the $R_2$ term involving $ \slashed{\mathrm{tr}}\chi \underline{\hat{\chi}}$ we break it into
$$ \slashed{\mathrm{tr}} \chi \underline{\hat{\chi}} = \left( \slashed{\mathrm{tr}} \chi - \frac{2}{v} \right) \underline{\hat{\chi}} + \frac{2}{v} \underline{\hat{\chi}} ,$$
and after breaking $R_2$ correspondingly as
$$R_2 \doteq R_2^1 + R_2^2 , $$ 
we have that
$$ R_2^1 \lesssim \frac{1}{v} \left( \int_{u_1}^{u_2} \frac{1}{\mathrm{w}^2 (u)} \right)^{1/2} \times \left( \int_{u_1}^{u_2} \mathrm{w}^2 (u)  \| \underline{\hat{\chi}} \|_{L^2 (S_{u,v} )}^2 \, du \right)^{1/2} \lesssim \epsilon^{1/2} \frac{1}{v} \left( \int_{u_1}^{u_2} \mathrm{w}^2 (u)  \| \underline{\hat{\chi}} \|_{L^2 (S_{u,v} )}^2 \, du \right)^{1/2} , $$
which is integrable in $v$ after we integrate $\| \hat{\chi} \|_{L^2 (S_{u_2 ,v} )}$ in $v$ and multiply it by $\frac{1}{v}$. Similarly we can square the $L^2 (S)$ norm of $\hat{\chi}$ first and then integrate in $v$ and the right hand side is finite by the bootstrap assumptions. Note that the smallness in $\epsilon$ comes from the choice of $U$. Moreover, the same argument works for all derivatives up to 3. On the other hand we have that
\begin{align*}
R_2^2 \lesssim & \int_{u_1}^{u_2} \left\| \left( \slashed{\mathrm{tr}} \chi - \frac{2}{v} \right) \hat{\underline{\chi}} \right\|_{L^2 (S_{u,v})} \, du \\ \lesssim & \frac{1}{v^2} \sum_{k=0}^2 \left( \int_{u_1}^{u_2} \frac{1}{\mathrm{w}^2 (u)} \left\| v (v \slashed{\nabla} )^k \left( \slashed{\mathrm{tr}} \chi - \frac{2}{v} \right)  \right\|_{L^2 (S_{u,v})}^2 \, du \right)^{1/2} \cdot \left( \int_{u_1}^{u_2} \mathrm{w}^2 (u) \| \hat{\underline{\chi}} \|_{L^2 (S_{u,v} )}^2 \, du \right)^{1/2} \\ \lesssim & \epsilon^{1/2} \frac{1}{v^2}  \sum_{k=0}^2 \left\| v (v \slashed{\nabla} )^k \left( \slashed{\mathrm{tr}} \chi - \frac{2}{v} \right)  \right\|_{L^{\infty}_u L^2 (S_{u,v})}  \left( \int_{u_1}^{u_2} \mathrm{w}^2 (u) \| \hat{\underline{\chi}} \|_{L^2 (S_{u,v} )}^2 \, du \right)^{1/2}  ,
\end{align*}
and then we have that
\begin{align*}
\int_v \frac{1}{v} R_2^2 \lesssim & \epsilon^{1/2} \left( \int_{u_1}^{u_2} \mathrm{w}^2 (u) \| \hat{\underline{\chi}} \|_{L^2 (S_{u,v} )}^2 \, du \right)^{1/2} \cdot \left( \sum_{k=0}^2 \int_v \frac{1}{v^3} \left\| v (v \slashed{\nabla} )^k \left( \slashed{\mathrm{tr}} \chi - \frac{2}{v} \right)  \right\|_{L^{\infty}_u L^2 (S_{u,v})} \, dv \right) \\ \lesssim &  \epsilon^{1/2} \left( \int_{u_1}^{u_2} \mathrm{w}^2 (u) \| \hat{\underline{\chi}} \|_{L^2 (S_{u,v} )}^2 \, du \right)^{1/2} \cdot \left( \int_v \frac{1}{v^4} \, dv \right)^{1/2} \cdot \left[ \sum_{k=0}^2 \left( \int_v  \left\| (v \slashed{\nabla} )^k \left( \slashed{\mathrm{tr}} \chi - \frac{2}{v} \right)  \right\|_{L^{\infty}_u L^2 (S_{u,v})} \, dv \right)^{1/2} \right] ,
\end{align*} 
and the last expression can be bounded by the bootstrap assumptions. We can deal similarly to deal with $\int_v \frac{1}{v} ( R_2^2 )^2$. For the $R_3$ term we have that
\begin{align*}
R_3 \lesssim & \int_{u_1}^{u_2} \| | \eta |^2 \|_{L^2 (S_{u,v} )} \, du \\ \lesssim & \sum_{k=0}^2 \left( \int_{u_1}^{u_2} \frac{1}{v} \| ( v \slashed{\nabla} )^k \eta \|_{L^2 (S_{u,v} )} \, du \right)^{1/2} \cdot  \left( \int_{u_1}^{u_2}  \|  \eta \|_{L^2 (S_{u,v} )} \, du \right)^{1/2} \\ \lesssim & \sum_{k=0}^2  \frac{1}{v^3} \| v ( v \slashed{\nabla} )^k \eta \|_{L^{\infty}_u L^2 (S_{u,v} )}  \cdot   \| v \eta \|_{L^{\infty}_u L^2 (S_{u,v} )} ,
\end{align*}
and the last term can be integrated in $v$ with $\frac{1}{v}$ as a weight either as it is or after being squared. Note that the same process works for all derivatives up to 3.

We turn to $\slashed{\mathrm{tr}} \chi - \frac{2}{v}$ and we have that
\begin{align*}
\left\| \slashed{\mathrm{tr}} \chi - \frac{2}{v} \right\|_{L^2 (S_{u_2 ,v} )} \lesssim & \left\| \slashed{\mathrm{tr}} \chi - \frac{2}{v} \right\|_{L^2 (S_{u_1 , v} )} + \int_{u_1}^{u_2} \left\| \slashed{\mathrm{tr}} \underline{\chi} \left( \slashed{\mathrm{tr}} \chi - \frac{2}{v} \right) \right\|_{L^2 (S_{u,v} )} \, du + \int_{u_1}^{u_2} \| \slashed{\nabla} \eta \|_{L^2 (S_{u,v} )} \, du  \\ + & \int_{u_1}^{u_2} \left\| K - \frac{1}{v^2} \right\|_{L^2 (S_{u,v} )} \, du + \int_{u_1}^{u_2} \left\| \frac{1}{v} \left( \slashed{\mathrm{tr}} \underline{\chi} + \frac{2}{v} \right) \right\|_{L^2 (S_{u,v} )} \, du + \int_{u_1}^{u_2} \| |\eta |^2 \|_{L^2 (S_{u,v} )} \, du  \\ = &  \left\| \slashed{\mathrm{tr}} \chi - \frac{2}{v} \right\|_{L^2 (S_{u_1 , v} )} + r_0 + r_1 + r_2 + r_3 + r_4 .
\end{align*}

For the $r_1$ term we work as for $R_1$ above in the case of $\hat{\chi}$. For the $r_2$ term we have that
$$ r_2 =  \int_{u_1}^{u_2} \left\| K - \frac{1}{v^2} \right\|_{L^2 (S_{u,v} )} \, du \lesssim \epsilon^{1/2} \frac{1}{v^{3/2}} \left( \int_{u_1}^{u_2} v^3 \left\| K - \frac{1}{v^2} \right\|_{L^2 (S_{u,v} )}^2 \, du \right)^{1/2} . $$
For the $r_3$ term we have that
\begin{align*}
 r_3 = & \int_{u_1}^{u_2} \left\| \frac{1}{v} \left( \slashed{\mathrm{tr}} \underline{\chi} + \frac{2}{v} \right) \right\|_{L^2 (S_{u,v} )} \, du \lesssim \frac{1}{v} \left( \int_{u_1}^{u_2} \frac{1}{\mathrm{w}^2 (u)} \, du \right)^{1/2} \cdot \left( \int_{u_1}^{u_2} \mathrm{w}^2 (u)  \left\| \slashed{\mathrm{tr}}\underline{\chi} + \frac{2}{v}  \right\|_{L^2 (S_{u,v} )}^2 \, du \right)^{1/2} \\ \lesssim & \epsilon^{1/2} \frac{1}{v} \left( \int_{u_1}^{u_2} \mathrm{w}^2 (u)  \left\| \slashed{\mathrm{tr}}\underline{\chi} + \frac{2}{v}  \right\|_{L^2 (S_{u,v} )}^2 \, du \right)^{1/2} .
\end{align*} 
For the $r_4$ term we work in the same way as for $R_3$ above in the case of $\hat{\chi}$. Gathering together all the above estimates we have that
$$ \left\| v \left( \slashed{\mathrm{tr}} \chi - \frac{2}{v} \right) \right\|_{L^2 (S_{u_2 ,v})} \lesssim \left\| v \left( \slashed{\mathrm{tr}} \chi - \frac{2}{v} \right) \right\|_{L^2 (S_{u_2 ,v})} + \epsilon^{1/2} r , $$
where $r$ can be bounded by the bootstrap assumptions. The same method (giving a similar estimate) works for all derivatives up to 3. 

For $\omega$ we have that
\begin{align*}
\| \omega \|_{L^2 (S_{u_2 , v} )} \lesssim & \| \omega \|_{L^2 (S_{u_1 , v} )} + \int_{u_1}^{u_2} \| \zeta ( \eta - \underline{\eta} ) \|_{L^2 (S_{u,v} )} \, du + \int_{u_1}^{u_2} \| \eta \cdot \underline{\eta} ) \|_{L^2 (S_{u,v} )} \, du \\ + & \int_{u_1}^{u_2} \left\| K - \frac{1}{v^2} \right\|_{L^2 (S_{u,v} )} \, du + \int_{u_1}^{u_2} \| \hat{\chi} \cdot \underline{\hat{\chi}} \|_{L^2 (S_{u,v} )} du + \int_{u_1}^{u_2} \left\| \frac{1}{v} \left( \slashed{\mathrm{tr}} \chi - \frac{2}{v} \right) \right\|_{L^2 (S_{u,v} )} \, du \\ + & \int_{u_1}^{u_2} \left\| \frac{1}{v} \left( \slashed{\mathrm{tr}} \underline{\chi} +  \frac{2}{v} \right) \right\|_{L^2 (S_{u,v} )} du + \int_{u_1}^{u_2} \left\| \left( \slashed{\mathrm{tr}} \chi - \frac{2}{v} \right) \left( \mathrm{tr}\underline{\chi} + \frac{2}{v} \right) \right\|_{L^2 (S_{u,v} )} \, du \\ = & \| \omega \|_{L^2 (S_{u_1 , v} )} + w_1 + w_2 + w_3 + w_4 + w_5 + w_6 + w_7 .
\end{align*}
For the $w_1$ term we have that
\begin{align*}
w_1 \lesssim & \int_{u_1}^{u_2} \left( \| | \underline{\eta} | | \eta |  \|_{L^2 (S_{u,v} )} + \| | \underline{\eta} |^2   \|_{L^2 (S_{u,v} )} \right) \, du \\ \lesssim & \sum_{k=0}^2 \left( \int_{u_1}^{u_2} \frac{1}{v} \| (v \slashed{\nabla} )^k \underline{\eta} \|_{L^2 (S_{u,v} )}^2 \, du \right)^{1/2} \cdot \left[ \left( \int_{u_1}^{u_2}  \| \underline{\eta} \|_{L^2 (S_{u,v} )}^2 \, du \right)^{1/2} + \left( \int_{u_1}^{u_2}  \| \eta \|_{L^2 (S_{u,v} )}^2 \, du \right)^{1/2} \right] \\ \lesssim & \epsilon \frac{1}{v^2} \left( \sum_{k=0}^2 \| v (v \slashed{\nabla} )^k \underline{\eta} \|_{L^2 (S_{u,v} )} \right) \left( \| v \underline{\eta} \|_{L^2 (S_{u,v} )} +\| v \eta \|_{L^2 (S_{u,v} )} \right) ,
\end{align*} 
and the last term can be bounded by the bootstrap assumptions after integrating it in $v$, and the same is true after squaring the last obtained expression and integrating it in $v$ with a $v$-weight. The $w_2$ term can be treated in the exact same way. Note that the same method holds for all derivatives up to 3 by using Sobolev's inequality and making sure that no term (either $\eta$ or $\underline{\eta}$) has more than 3 derivatives. The $w_3$ term can be treated similarly to the $r_2$ term above in the case of $\slashed{\mathrm{tr}} \chi - \frac{2}{v}$. For the $w_4$ term we have that
\begin{align*}
w_4 = & \int_{u_1}^{u_2} \| \hat{\chi} \cdot \hat{\underline{\chi}} \|_{L^2 (S_{u,v})} \, du \\ \lesssim & \sum_{k=0}^2 \left( \int_{u_1}^{u_2} \frac{1}{\mathrm{w}^2 (u) v^2} \| ( v \slashed{\nabla} )^k \hat{\chi} \|_{L^2 (S_{u,v} )}^2 \, du \right)^{1/2} \cdot \left( \int_{u_1}^{u_2} \mathrm{w}^2 (u) \| \hat{\underline{\chi}} \|_{L^2 (S_{u,v} )}^2 \, du \right)^{1/2} \\ \lesssim & \epsilon^{1/2} \frac{1}{v} \left( \sum_{k=0}^2 \| ( v \slashed{\nabla} )^k \hat{\chi} \|_{L^{\infty}_u L^2 (S_{u,v})} \right) \cdot \left( \int_{u_1}^{u_2} \mathrm{w}^2 (u) \| \hat{\underline{\chi}} \|_{L^2 (S_{u,v} )}^2 \, du \right)^{1/2} ,
\end{align*}
and the last term is integrable in $v$ by the bootstrap assumptions, and the same is true after squaring the last term and then integrating it in $v$ with a $v$-weight. For the $w_5$ term we have that
\begin{align*}
w_5 = & \int_{u_1}^{u_2} \left\| \frac{1}{v} \left( \slashed{\mathrm{tr}} \chi - \frac{2}{v} \right) \right\|_{L^2 (S_{u,v} )} \, du \\ \lesssim & \epsilon^{1/2} \left( \int_{u_1}^{u_2} \frac{1}{v^3} \left\| v \left( \slashed{\mathrm{tr}} \chi - \frac{2}{v} \right) \right\|_{L^2 (S_{u,v})}^2 \, du \right)^{1/2} \\ \lesssim & \epsilon^{1/2} \frac{1}{v^{3/2}} \left\| v \left( \slashed{\mathrm{tr}} \chi - \frac{2}{v} \right) \right\|_{L^{\infty}_u L^2 (S_{u,v})} , 
\end{align*}
and the last term is integrable in $v$ by the bootstrap assumptions, and the same is true after squaring the last term and then integrating it in $v$ with a $v$-weight. For the $w_6$ term we have that
\begin{align*}
w_6 = & \int_{u_1}^{u_2} \left\| \frac{1}{v} \left( \slashed{\mathrm{tr}} \underline{\chi} + \frac{2}{v} \right) \right\|_{L^2 (S_{u,v} )} \, du \\ \lesssim & \frac{1}{v^{3/2}} \left( \int_{u_1}^{u_2} \frac{1}{\mathrm{w}^2 (u)} \, du \right)^{1/2} \cdot \left( \int_{u_1}^{u_2} \left\| v \left( \slashed{\mathrm{tr}} \underline{\chi} + \frac{2}{v} \right) \right\|_{L^2 (S_{u,v})}^2 \, du \right)^{1/2} \\ \lesssim & \epsilon^{1/2} frac{1}{v^{3/2}} \left( \int_{u_1}^{u_2} \left\| v \left( \slashed{\mathrm{tr}} \underline{\chi} + \frac{2}{v} \right) \right\|_{L^2 (S_{u,v})}^2 \, du \right)^{1/2} , 
\end{align*}
and the same holds as in the previous cases. For the $w_7$ term we have that
\begin{align*}
w_7 = & \int_{u_1}^{u_2} \left\| \left( \slashed{\mathrm{tr}} \chi - \frac{2}{v} \right) \left( \mathrm{tr}\underline{\chi} + \frac{2}{v} \right) \right\|_{L^2 (S_{u,v} )} \, du \\ \lesssim & \sum_{k=0}^2 \left( \int_{u_1}^{u_2} \left\| ( v \slashed{\nabla} )^k  \left( \slashed{\mathrm{tr}} \underline{\chi} + \frac{2}{v} \right) \right\|_{L^2 (S_{u,v})}^2 \, du \right)^{1/2} \cdot \left( \int_{u_1}^{u_2}  \left\|  \slashed{\mathrm{tr}} \underline{\chi} - \frac{2}{v}  \right\|_{L^2 (S_{u,v})}^2 \, du \right)^{1/2} \\ \lesssim & \frac{1}{v^2}\sum_{k=0}^2 \left( \int_{u_1}^{u_2} \left\| v ( v \slashed{\nabla} )^k  \left( \slashed{\mathrm{tr}} \underline{\chi} + \frac{2}{v} \right) \right\|_{L^2 (S_{u,v})}^2 \, du \right)^{1/2} \cdot \left\|  v \left( \slashed{\mathrm{tr}} \underline{\chi} - \frac{2}{v} \right)  \right\|_{L^{\infty}_u L^2 (S_{u,v})}  ,
\end{align*}
and once again the last term is integrable in $v$ by the bootstrap assumptions, and the same is true after squaring the last term and then integrating it in $v$ with a $v$-weight.  Gathering together all the above estimates we have that
$$ \int_v \| \omega \|_{L^2 (S_{u_2 , v } )} \, dv \lesssim \int_v \| \omega \|_{L^2 (S_{u_1 , v } )} \, dv + \epsilon^{1/2} W_1 , $$
and 
$$ \int_v v \| \omega \|_{L^2 (S_{u_2 , v } )}^2 \, dv \lesssim \int_v v \| \omega \|_{L^2 (S_{u_1 , v } )}^2 \, dv + \epsilon W_2 , $$
where both $W_1$ and $W_2$ can be bounded by the bootstrap assumptions. Note that the same method used above works for all derivatives up to 3 by appropriately using Sobolev's inequality \eqref{sobolev3}.

For the top order angular derivative we use the $\slashed{\nabla}_4$ equation \eqref{eq:trchi4} for $\slashed{\mathrm{tr}} \chi - \frac{2}{v}$, and the elliptic equation \eqref{eq:codchi}. As opposed to previous work, in this situation we actually need to use them simultaneously. We have that
$$ v_2 \left\| ( v_2\slashed{\nabla} )^4 \left( \slashed{\mathrm{tr}}\chi - \frac{2}{v} \right) \right\|_{L^2 (S_{u , v_2} )} \lesssim  v_1 \left\|  ( v_1 \slashed{\nabla} )^4 \left(\slashed{\mathrm{tr}} \chi - \frac{2}{v_1} \right) \right\|_{L^2 (S_{u , v_1} )} + \int_{v_1}^{v_2} v \| F^4_{\slashed{\mathrm{tr}} \chi - \frac{2}{v}} \|_{L^2 (S_{u,v} )} \, dv $$
where
$$ \snabla_4 \left[ ( \slashed{\nabla}^4 ) \left( \slashed{\mathrm{tr}} \chi - \frac{2}{v} \right) \right] + 3 \slashed{\mathrm{tr}} \chi \cdot \left[ (  \slashed{\nabla}^4 ) \left( \slashed{\mathrm{tr}} \chi - \frac{2}{v} \right) \right] = \frac{1}{v^4} F^4_{\slashed{\mathrm{tr}} \chi - \frac{2}{v}} . $$
The trickiest terms are the one that is linear in $\omega$, that can be treated directly though as the norm
$$ \int_{v_1}^{v_2} \| ( v \slashed{\nabla} )^4 \omega \|_{L^{\infty}_u L^2 (S_{u,v} )} \, dv $$
is assumed to be finite for all $v_1$, $v_2$ by the bootstrap assumptions, and the one involving $| \hat{\chi} |^2$. For this last term we look at the case where no additional terms are introduced by the commutations, hence we have that
\begin{align*}
\sum_{k_1 + k_2 = 4} \int_{v_1}^{v_2} v \| ( ( v \slashed{\nabla} )^{k_1} \hat{\chi} ) \cdot &   ( ( v \slashed{\nabla} )^{k_2} \hat{\chi}  ) \|_{L^2 (S_{u,v} )} \, dv \lesssim   \sum_{k_1 , k_2 = 0 , k_1 + k_2 > 0}^4  \int_{v_1}^{v_2} \| ( v \slashed{\nabla} )^{k_1} \hat{\chi} \|_{L^2 (S_{u,v} )} \| ( v \slashed{\nabla} )^{k_2} \hat{\chi} \|_{L^2 (S_{u,v} )} \, dv \\ \lesssim & \sum_{0 \leq k_1 \leq 4 , 1 \leq k_2 \leq 4 } \int_{v_1}^{v_2} \| ( v \slashed{\nabla} )^{k_1} \hat{\chi} \|_{L^2 (S_{u,v} )} \left\| ( v \slashed{\nabla} )^{k_2} \left( \slashed{\mathrm{tr}}  \chi - \frac{2}{v} \right)  \right\|_{L^2 (S_{u,v} )} \, dv  \\ + &  \sum_{0 \leq k_1 \leq 4 , 1 \leq k_2 \leq 4 } \int_{v_1}^{v_2} v \| ( v \slashed{\nabla} )^{k_1} \hat{\chi} \|_{L^2 (S_{u,v} )} \| ( v \slashed{\nabla} )^{k_2} \beta \|_{L^2 (S_{u,v} )} \, dv \\ + & \sum_{0 \leq k_1 \leq 4 , 1 \leq k_2 \leq 4 } \int_{v_1}^{v_2} v \| ( v \slashed{\nabla} )^{k_1} \hat{\chi} \|_{L^2 (S_{u,v} )} \| ( v \slashed{\nabla} )^{k_2} F^4_{ell} \|_{L^2 (S_{u,v} )} \, dv ,
\end{align*}
where 
$$ v^{k+1} \slashed{\nabla}^k \slashed{\mathrm{div}} \hat{\chi} =  (v \slashed{\nabla} )^{k+1} \left( \mathrm{tr} \chi - \frac{2}{v} \right) - v^{k+1} \slashed{\nabla}^k \beta + v^{k+1} \slashed{\nabla}^k F^4_{ell} . $$ 
Note that we used Sobolev's inequality \eqref{sobolev3}, and then the elliptic estimate \eqref{eq:codchi}. The terms included in $F^4_{ell}$ are better, so we consider the first two terms of the last expression, and we have that:
\begin{align*}
\sum_{0 \leq k_1 \leq 4 , 1 \leq k_2 \leq 4 } & \int_{v_1}^{v_2} \| ( v \slashed{\nabla} )^{k_1} \hat{\chi} \|_{L^2 (S_{u,v} )} \left\| ( v \slashed{\nabla} )^{k_2}  \left( \slashed{\mathrm{tr}}  \chi - \frac{2}{v} \right)  \right\|_{L^2 (S_{u,v} )} \, dv  \\ \lesssim &  \sum_{0 \leq k_1 \leq 4 , 1 \leq k_2 \leq 4 } \sup_v \left( v   \left\| ( v \slashed{\nabla} )^{k_2}  \left( \slashed{\mathrm{tr}}  \chi - \frac{2}{v} \right)  \right\|_{L^2 (S_{u,v} )} \right) \cdot \int_{v_1}^{v_2} \frac{1}{v}   \| ( v \slashed{\nabla} )^{k_1} \hat{\chi} \|_{L^2 (S_{u,v} )} \, dv ,
\end{align*}
and by repeating the same process for all angular derivatives of $\slashed{\mathrm{tr}} \chi - \frac{2}{v}$ and taking the $L^1_v L^{\infty}_u L^2 (S_{u,v} )$ of $\hat{\chi}$ to be appropriately small we can absorb the above term in the left hand side. On the other hand we have that:
\begin{align*}
 \sum_{0 \leq k_1 \leq 4 , 0 \leq k_2 \leq 3 } & \int_{v_1}^{v_2} v \| ( v \slashed{\nabla} )^{k_1} \hat{\chi} \|_{L^2 (S_{u,v} )} \| ( v \slashed{\nabla} )^{k_2} \beta \|_{L^2 (S_{u,v} )} \, dv\\ \lesssim & \sum_{0 \leq k_1 \leq 4 , 0 \leq k_2 \leq 3 } \left( \int_{v_1}^{v_2} \frac{1}{v} \| ( v \slashed{\nabla} )^{k_1} \hat{\chi} \|_{L^2 (S_{u,v} )}^2 \, dv \right)^{1/2} \cdot  \left( \int_{v_1}^{v_2} v^3 \| ( v \slashed{\nabla} )^{k_2} \beta \|_{L^2 (S_{u,v} )} \, dv \right)^{1/2} ,
\end{align*}
and both of the last terms are bounded by the bootstrap assumptions.

Finally in this section, we deal with the fourth top-order derivative for $\omega$. We introduce the auxiliary quantity $\omega^{\dagger}$ through the equation
\begin{equation}\label{eq:omegadag}
\slashed{\nabla}_3 \omega^{\dagger} = \frac{1}{2} \check{\sigma} .
\end{equation}
Let
\begin{equation}\label{eq:kappa}
\kappa \doteq \snabla \omega +^{*} \snabla \omega^{\dagger} - \frac{1}{2} \beta .
\end{equation}
Note that we have that
\begin{equation}\label{eq:kappa3}
\begin{split}
\snabla_3 \kappa \sim &  - \snabla G^1_{\kappa} - \frac{1}{2} \snabla G^2_{\kappa} + \sum_{\psi \in \{ \eta , \underline{\eta} \}} \psi  ( G^1_{\kappa}  + G^2_{\kappa} ) \\ & + \sum_{\psi_{\underline{H}} \in \{ \underline{\hat{\chi}} , \slashed{\mathrm{tr}} \underline{\chi} \} } [  \psi_{\underline{H}} ( \snabla \kappa) + ( \snabla \psi_{\underline{H}} ) \kappa + \psi \psi_{\underline{H}} \kappa ] ,
\end{split}  
\end{equation}
where
\begin{align*}
G^1_{\kappa} \doteq &  \zeta \cdot ( \eta - \underline{\eta} ) - \eta \cdot \underline{\eta} - \frac{1}{2} \left( K - \frac{1}{v^2} \right) \\ & + \frac{1}{4} \hat{\chi} \cdot \underline{\hat{\chi}}  -\frac{1}{8} \left( \slashed{\mathrm{tr}} \chi - \frac{2}{v} \right) \left( \slashed{\mathrm{tr}} \underline{\chi} + \frac{2}{v} \right) + \frac{1}{2v} \left( \slashed{\mathrm{tr}} \chi - \frac{2}{v} \right) - \frac{1}{2v} \left( \slashed{\mathrm{tr}} \underline{\chi} + \frac{2}{v} \right) ,
\end{align*}
and 
$$ G^2_{\kappa} \doteq \snabla \left( K - \frac{1}{v^2} \right) - F_{\beta} , $$
which implies that
\begin{align*}
\snabla_3 [ ( v \snabla )^3 \kappa ] \sim &  - \frac{1}{v^4} ( v \snabla )^3 G^1_{\kappa}  - \frac{1}{2v^4} ( v \snabla )^3 G^2_{\kappa} + \frac{1}{v} \sum_{\psi \in \{ \eta , \underline{\eta}\} } \sum_{i_1 + i_2 + i_3 = 3, i_3 \geq 1} [ ( v \snabla )^{i_1} \psi ]^{i_2} ( v \snabla )^{i_3} ( G^1_{\kappa} + G^2_{\kappa} ) \\ & + \frac{1}{v} \sum_{\psi \in \{ \eta , \underline{\eta}\} , \psi_{\underline{H}} \in \{ \hat{\underline{\chi}} , \slashed{\mathrm{tr}} \underline{\chi} \} } \sum_{i_1 + i_2 + i_3 + i_4 = 3} [ ( v \snabla )^{i_1} \psi ]^{i_2} [ ( v \snabla )^{i_3} \psi_{\underline{H}} ] [  ( v \snabla )^{i_4} \kappa ] .
\end{align*}
From the definition of $\omega^{\dagger}$ we have that:
\begin{align*}
\| \omega^{\dagger} \|_{L^2 (S_{u_2 , v}  )} \lesssim & \| \omega^{\dagger} \|_{L^2 (S_{u_2 , v}  )} + \int_{u_1}^{u_2} \| \check{\sigma} \|_{L^2 (S_{u,v} ) } \, du \\ \lesssim & \| \omega^{\dagger} \|_{L^2 (S_{u_2 , v}  )} +  \epsilon^{1/2} \frac{1}{v^{3/2}} \left( \int_{u_1}^{u_2} v^3 \| \check{\sigma} \|_{L^2 (S_{u,v} ) }^2 \, du \right)^{1/2} , 
\end{align*}
which implies that 
$$ \sum_{i \leq 3} \| v^{1/2} (v \snabla )^i \omega^{\dagger} \|_{L^2_v L^{\infty}_u L^2 (S_{u,v} ) }^2 \lesssim  \sum_{i \leq 3} \left( \| v^{1/2} (v \snabla )^i \omega^{\dagger} \|_{L^2_v  L^2 (S_{u_1 ,v} ) }^2  + \epsilon \int_{u_1}^{u_2} v^3 \| \check{\sigma} \|_{L^2 (S_{u,v} ) }^2 \, du \right) , $$
and
$$ \sum_{i \leq 3} \| (v \snabla )^i \omega^{\dagger} \|_{L^1_v L^{\infty}_u L^2 (S_{u,v} ) } \lesssim  \sum_{i \leq 3} \left[ \| (v \snabla )^i \omega^{\dagger} \|_{L^1_v  L^2 (S_{u_1 ,v} ) }  + \epsilon^{1/2} \left( \int_{u_1}^{u_2} v^3 \| \check{\sigma} \|_{L^2 (S_{u,v} ) }^2 \, du \right)^{1/2} \right] . $$
Using the last two estimates, the equation for $\snabla_3 [ (v \snabla )^3 \kappa ]$, and the estimates for $\eta$, $\underline{\eta}$, $\hat{\underline{\chi}}$, $\slashed{\mathrm{tr}} \underline{\chi}$, $K - \frac{1}{v^2}$ and $\beta$, we get that
$$ \| v^{1/2} ( v \snabla )^4 \kappa \|_{L^2_v L^{\infty}_u L^2 (S_{u,v} )}^2 \lesssim \| v^{1/2} ( v \snabla )^4 \kappa \|_{L^2_v L^2 (S_{u_1 ,v} )}^2 + \epsilon \bar{\Omega} , $$
and
$$ \| ( v \snabla )^4 \kappa \|_{L^1_v L^{\infty}_u L^2 (S_{u,v} )} \lesssim \| v^{1/2} ( v \snabla )^4 \kappa \|_{L^1_v  L^2 (S_{u_1 ,v} )} + \epsilon^{1/2} \bar{\Omega} , $$
where $\bar{\Omega}$ can be bounded by the bootstrap assumptions. Now we note that we have the following div-curl system:
$$ \slashed{\mathrm{div}} ( \snabla \omega^{\dagger} ) =  \slashed{\mathrm{div}} \kappa + \frac{1}{2} \slashed{\mathrm{div}} \beta , $$
$$ \slashed{\mathrm{div}} ( \snabla \omega^{\dagger} ) =  \slashed{\mathrm{curl}} \kappa + \frac{1}{2} \slashed{\mathrm{curl}} \beta , $$
$$ \slashed{\mathrm{curl}} ( \snabla \omega ) = \slashed{\mathrm{curl}} ( \snabla \omega ) = 0 , $$
and by Proposition \ref{ellipticcurl} and estimate \eqref{bound:gauss}, we have that
\begin{align*}
\| v^{1/2} ( v \snabla )^4 \omega \|_{L^2_v L^{\infty}_u L^2 (S_{u,v} )} + &\| v^{1/2} ( v \snabla )^4 \omega^{\dagger} \|_{L^2_v L^{\infty}_u L^2 (S_{u,v} )} \\ &+ \| ( v \snabla )^4 \omega \|_{L^1_v L^{\infty}_u L^2 (S_{u,v} )} + \| ( v \snabla )^4 \omega^{\dagger} \|_{L^1_v L^{\infty}_u L^2 (S_{u,v} )} \lesssim \bar{\Omega}_i + \epsilon^{1/2} \widetilde{\Omega} , 
\end{align*}
where $\bar{\Omega}_i$ is bounded by the initial data while $\widetilde{\Omega}$ is bounded by the bootstrap assumptions.

\textbf{Ricci coefficients, 2. $\hat{\underline{\chi}}$ and $\slashed{\mathrm{tr}} \underline{\chi}$:} We use the $\slashed{\nabla}_4$ equation \eqref{eq:uchi4} for $\hat{\underline{\chi}}$ and the $\slashed{\nabla}_3$ equations \eqref{eq:truchir3}, \eqref{eq:truchi3} for $\mathrm{tr} \underline{\chi}$ and $\mathrm{tr} \underline{\chi} + \frac{2}{v}$ for up to 3 derivatives. For $\hat{\underline{\chi}}$ we have that
\begin{align*}
\| \hat{\underline{\chi}} \|_{L^2 (S_{u,v_2})} \lesssim & \| \hat{\underline{\chi}} \|_{L^2 (S_{u,v_2})} + \int_{v_1}^{v_2} \| \slashed{\nabla} \underline{\eta} \|_{L^2 (S_{u,v} )} \, dv \\ & + \int_{v_1}^{v_2} \| \omega \hat{\underline{\chi}} \|_{L^2 (S_{u,v} )} \, dv+ \int_{v_1}^{v_2} \frac{1}{v} \| \hat{\chi} \|_{L^2 (S_{u,v} )} \, dv \\ & + \int_{v_1}^{v_2} \left\| \left( \slashed{\mathrm{tr}} \underline{\chi} + \frac{2}{v} \right) \hat{\chi} \right\|_{L^2 (S_{u,v} )} \, dv +  \int_{v_1}^{v_2} \| \underline{\eta} \otimes \underline{\eta} \|_{L^2 (S_{u,v} )} \, dv \\ = & \| \hat{\underline{\chi}} \|_{L^2 (S_{u,v_2})} + d_1 + d_2 + d_3 + d_4 + d_5 .
\end{align*}
For $d_1$ we have that
$$ d_1 \doteq \int_{v_1}^{v_2} \frac{1}{v} \| ( v \slashed{\nabla} ) \underline{\eta} \|_{L^2 (S_{u,v} )} \, dv \lesssim \epsilon^{1/2} \left( \int_{v_1}^{v_2} \| v ( v \slashed{\nabla} ) \underline{\eta} \|_{L^2 (S_{u,v} )}^2 \, dv \right)^{1/2} . $$
For $d_2$ we have that
\begin{align*}
d_2 \doteq & \int_{v_1}^{v_2} \| \omega \hat{\underline{\chi}} \|_{L^2 (S_{u,v} )} \, dv \lesssim \sum_{k=0}^2  \int_{v_1}^{v_2} \frac{1}{v} \| ( v \slashed{\nabla} )^k \omega \|_{L^2 (S_{u,v} )} \| \hat{\underline{\chi}} \|_{L^2 (S_{u,v} )} \, dv \\ \lesssim & \frac{1}{\mathrm{w} (u)} \sum_{k=0}^2 \left( \int_{v_1}^{v_2} \frac{1}{v^3} \mathrm{w}^2 (u) \| \hat{\underline{\chi}} \|_{L^2 (S_{u,v} )}^2 \, dv \right)^{1/2} \left( \int_{v_1}^{v_2} v \| ( v \slashed{\nabla} )^k \omega \|_{L^2 (S_{u,v} )}^2 \, dv \right)^{1/2} . 
\end{align*}
For $d_4$ we have that
\begin{align*}
d_4 \doteq & \int_{v_1}^{v_2} \left\| \left( \slashed{\mathrm{tr}} \underline{\chi} + \frac{2}{v} \right) \hat{\chi} \right\|_{L^2 (S_{u,v} )} \, dv \lesssim \sum_{k=0}^2 \int_{v_1}^{v_2} \frac{1}{v^2} \left\| v ( v \slashed{\nabla} )^k \left( \slashed{\mathrm{tr}} \underline{\chi} + \frac{2}{v} \right) \right\|_{L^2 (S_{u,v} )}  \| \hat{\chi} \|_{L^2 (S_{u,v} )} \, dv \\ \lesssim & \frac{1}{\mathrm{w} (u)} \sum_{k=0}^2 \left( \int_{v_1}^{v_2} \frac{1}{v^3} \mathrm{w}^2 (u) \left\| v ( v \slashed{\nabla} )^k \left( \slashed{\mathrm{tr}} \underline{\chi} + \frac{2}{v} \right) \right\|_{L^2 (S_{u,v} )}^2 \, dv \right)^{1/2} \left( \int_{v_1}^{v_2} \frac{1}{v} \| \hat{\chi} \|_{L^2 (S_{u,v} )}^2 \, dv \right)^{1/2} .
\end{align*}
For $d_5$ we have that
\begin{align*}
d_5 \doteq & \int_{v_1}^{v_2} \| \underline{\eta} \otimes \underline{\eta} \|_{L^2 (S_{u,v} )} \, dv \lesssim \sum_{k=0}^2 \int_{v_1}^{v_2} \frac{1}{v^3} \| v ( v \slashed{\nabla} )^k \underline{\eta} \|_{L^2 (S_{u,v} )} \| v \underline{\eta} \|_{L^2 (S_{u,v} )} \, dv \\ \lesssim & \sum_{k=0}^2 \left( \int_{v_1}^{v_2} \frac{1}{v^2} \| v ( v \slashed{\nabla} )^k \underline{\eta} \|_{L^2 (S_{u,v} )}^2 \, dv \right)^{1/2} \left( \int_{v_1}^{v_2} \frac{1}{v^2} \| v \underline{\eta} \|_{L^2 (S_{u,v} )}^2 \, dv \right)^{1/2} .
\end{align*}
Gathering all the above estimates we have that
\begin{align*}
\int_{u_1}^{u_2} & \mathrm{w}^2 (u) \|  \hat{\underline{\chi}} \|_{L^2 (S_{u,v_2} )}^2 \, du \lesssim  \int_{u_1}^{u_2} \mathrm{w}^2 (u) \| \hat{\underline{\chi}} \|_{L^2 (S_{u,v_1} )}^2 \, du + \epsilon \int_{u_1}^{u_2} \int_{v_1}^{v_2} \mathrm{w}^2 (u) \| v ( v \slashed{\nabla} ) \underline{\eta} \|_{L^2 (S_{u,v} )}^2 \, dudv \\ & + \sum_{k=0}^2 \left( \int_{v_1}^{v_2} v \| ( v \slashed{\nabla} )^k \omega \|_{L^2 (S_{u,v} )}^2 \, dv \right) \times \int_{u_1}^{u_2} \int_{v_1}^{v_2} \frac{1}{v^3} \mathrm{w}^2 (u) \| \hat{\underline{\chi}} \|_{L^2 (S_{u,v} )}^2 \, dudv \\ & + \epsilon \left( \int_{v_1}^{v_2} \frac{1}{v} \| \hat{\chi} \|_{L^2 (S_{u,v} )} \, dv \right)^2  + \sum_{k=0}^2 \left( \int_{v_1}^{v_2} \frac{1}{v} \| \hat{\chi} \|_{L^2 (S_{u,v} )}^2 \, dv \right) \times \int_{u_1}^{u_2} \int_{v_1}^{v_2} \frac{1}{v^3} \mathrm{w}^2 (u) \left\| v ( v \slashed{\nabla} )^k \left( \slashed{\mathrm{tr}} \underline{\chi} + \frac{2}{v} \right) \right\|_{L^2 (S_{u,v} )}^2 \, du dv \\ & + \epsilon \| v \underline{\eta} \|_{L^{\infty}_{u,v} L^2 (S_{u,v})}^2 \left( \sup_u \int_{v_1}^{v_2} \mathrm{w}^2 (u) \| ( v \slashed{\nabla} )^k \underline{\eta} \|_{L^2 (S_{u,v} )}^2 \, dv \right) . 
\end{align*}
Note that all the terms in the right hand side can be bounded by the bootstrap assumptions. A similar estimate holds for all derivatives up to 3. Note also that in the very last term we need the integral in $v$ and the weight with $\mathrm{w}$ when dealing with the fourth top order derivative of $\underline{\eta}$.

Now we turn to $\slashed{\mathrm{tr}} \underline{\chi}$. We use equation \eqref{eq:truchi3} and we have that
\begin{align*}
\| \slashed{\mathrm{tr}} \underline{\chi} \|_{L^2 (S_{u,v} )} \lesssim & \| \slashed{\mathrm{tr}} \underline{\chi} \|_{L^2 (S_{u_0 ,v} )} + \int_{u_0}^u \| ( \slashed{\mathrm{tr}} \underline{\chi} )^2 \|_{L^2 (S_{u' , v})} \, du' \\ & + \int_{u_0}^u \| | \hat{\underline{\chi}} |^2 \|_{L^2 (S_{u' , v})} \, du' \\ \lesssim & \| \slashed{\mathrm{tr}} \underline{\chi} \|_{L^2 (S_{u_0 ,v} )} + \sum_{k=0}^2 \int_{u_0}^u \frac{1}{v} \| (v \slashed{\nabla} )^k \slashed{\mathrm{tr}} \underline{\chi} \|_{L^2 (S_{u' , v})} \| \slashed{\mathrm{tr}} \underline{\chi} \|_{L^2 (S_{u' , v} )} \, du' \\ & +\sum_{k=0}^2 \int_{u_0}^u \frac{1}{v} \| (v \slashed{\nabla} )^k \hat{\underline{\chi}} \|_{L^2 (S_{u' , v})} \| \slashed{\mathrm{tr}} \underline{\chi} \|_{L^2 (S_{u' , v} )} \, du' \\ \lesssim & \| \slashed{\mathrm{tr}} \underline{\chi} \|_{L^2 (S_{u_0 ,v} )} + \frac{1}{v} \sum_{k=0}^2 \left( \int_{u_0}^u  \mathrm{w}^2 (u' ) \| (v \slashed{\nabla} )^k \slashed{\mathrm{tr}} \underline{\chi} \|_{L^2 (S_{u' , v})}^2 \, du' \right)^{1/2} \left( \int_{u_0}^u \frac{\mathrm{w}^2 (u')}{\mathrm{w}^4 (u' )} \| \slashed{\mathrm{tr}} \underline{\chi} \|_{L^2 (S_{u' , v} )}^2 \, du' \right)^{1/2}  \\ & +\frac{1}{v} \sum_{k=0}^2 \left( \int_{u_0}^u  \mathrm{w}^2 (u' ) \| (v \slashed{\nabla} )^k \hat{\underline{\chi}} \|_{L^2 (S_{u' , v})}^2 \, du' \right)^{1/2} \left( \int_{u_0}^u \frac{\mathrm{w}^2 (u')}{\mathrm{w}^4 (u' )} \| \hat{\underline{\chi}} \|_{L^2 (S_{u' , v} )}^2 \, du' \right)^{1/2} \\ \lesssim & \| \slashed{\mathrm{tr}} \underline{\chi} \|_{L^2 (S_{u_0 ,v} )} + \frac{1}{v} \sum_{k=0}^2 \left( \int_{u_0}^u  \mathrm{w}^2 (u' ) \| (v \slashed{\nabla} )^k \slashed{\mathrm{tr}} \underline{\chi} \|_{L^2 (S_{u' , v})}^2 \, du' \right)^{1/2} \left( \int_{u_0}^u \frac{\mathrm{w}^2 (u')}{\mathrm{w}^4 (u' )} \| \slashed{\mathrm{tr}} \underline{\chi} \|_{L^2 (S_{u' , v} )}^2 \, du' \right)^{1/2}  \\ & +\frac{1}{v} \frac{1}{\mathrm{w}^2 (u)} \sum_{k=0}^2 \left( \int_{u_0}^u  \mathrm{w}^2 (u' ) \| (v \slashed{\nabla} )^k \hat{\underline{\chi}} \|_{L^2 (S_{u' , v})}^2 \, du' \right)^{1/2} \left( \int_{u_0}^u \mathrm{w}^2 (u') \| \hat{\underline{\chi}} \|_{L^2 (S_{u' , v} )}^2 \, du' \right)^{1/2} ,
\end{align*}
where we used that
$$ \sup_{u' \in [u_0 , u]} \frac{1}{\mathrm{w}^2 (u' )} \leq \frac{1}{\mathrm{w}^2 (u)} , $$
and now from the previous estimate we see that we have 
$$ \int_{u_0}^{u''} \mathrm{w}^2 (u) \| \slashed{\mathrm{tr}} \underline{\chi} \|_{L^2 (S_{u,v} )}^2 \, du \lesssim  \| \slashed{\mathrm{tr}} \underline{\chi} \|_{L^{\infty}_v L^2 (S_{u_0 ,v} )}^2 + \epsilon D_{tr} , $$
where $D_{tr}$ can be bounded by the bootstrap assumptions. For $\slashed{\mathrm{tr}} \underline{\chi} + \frac{2}{v}$ we use equation \eqref{eq:truchir3} and we have for $m \in \{0,1,2,3\}$ that 
\begin{align*} 
v \left\| ( v \slashed{\nabla} )^m \left( \slashed{\mathrm{tr}} \underline{\chi} + \frac{2}{v} \right) \right\|_{L^2 (S_{u,v} )} & \lesssim  v \left\| ( v \slashed{\nabla} )^m \left( \slashed{\mathrm{tr}} \underline{\chi} + \frac{2}{v} \right) \right\|_{L^2 (S_{u_0 ,v} )} + \int_{u_0}^u v \left\| ( \slashed{\mathrm{tr}} \underline{\chi} ) \cdot ( v \slashed{\nabla} )^m \left( \slashed{\mathrm{tr}} \underline{\chi}  + \frac{2}{v} \right) \right\|_{L^2 (S_{u' , v})} \, du' \\ & + \sum_{m_1 + m_2 = m} \int_{u_0}^u v \left\| (v \slashed{\nabla} )^{m_1} \left( \slashed{\mathrm{tr}} \underline{\chi} + \frac{2}{v} \right)  \cdot ( v \slashed{\nabla} )^{m_2} \left( \slashed{\mathrm{tr}} \underline{\chi}  + \frac{2}{v} \right) \right\|_{L^2 (S_{u' , v})} \, du'  \\ & + \int_{u_0}^u \left\| (v \slashed{\nabla} )^m \left( \slashed{\mathrm{tr}} \underline{\chi} + \frac{2}{v} \right) \right\|_{L^2 (S_{u' , v})} \, du' + \sum_{m_1 + m_2 = m} \int_{u_0}^u v \| [ (v \slashed{\nabla} )^{m_1} \hat{\underline{\chi}} ] [ (v \slashed{\nabla} )^{m_2} \hat{\underline{\chi}} ] \|_{L^2 (S_{u' , v})} \, du' \\ & + \int_{u_0}^u v \| F_{\slashed{\mathrm{tr}} \underline{\chi} + \frac{2}{v} , m} \|_{L^2 (S_{u' , v} )} \, du' \\ \doteq  & v \left\| ( v \slashed{\nabla} )^m \left( \slashed{\mathrm{tr}} \underline{\chi} + \frac{2}{v} \right) \right\|_{L^2 (S_{u_0 ,v} )} + z_1 + z_2 + z_3 + z_4 + z_5 , 
\end{align*}
where $F_{\slashed{\mathrm{tr}} \underline{\chi} + \frac{2}{v} , m}$ contains quadratic and cubic terms of the form 
$$ \sum_{m_1 + m_2 = m-1} [ ( v \slashed{\nabla})^{m_1} \psi ] [ ( v \slashed{\nabla} )^{m_2} \dot{F}_{\slashed{\mathrm{tr}} \underline{\chi} + \frac{2}{v} } ] , \quad    \sum_{m_1 + m_2 = m} [ ( v \slashed{\nabla})^{m_1} ( \hat{\underline{\chi}} ,\slashed{\mathrm{tr}} \underline{\chi} )] [ ( v \slashed{\nabla} )^{m_2} \dot{F}_{\slashed{\mathrm{tr}} \underline{\chi} + \frac{2}{v} } ] ,$$ $$  \sum_{m_1 + m_2 + m_3 = m , \psi \in \{ \eta , \underline{\eta} \} , \psi_{\underline{H}} \in \{ \hat{\underline{\chi}} , \slashed{\mathrm{tr}} \underline{\chi} \}} [ ( v \slashed{\nabla} )^{m_1} \psi ] [ ( v \slashed{\nabla})^{m_2} \psi_{\underline{H}} ] [ ( v \slashed{\nabla} )^{m_3} \dot{F}_{\slashed{\mathrm{tr}} \underline{\chi} + \frac{2}{v} } ] , $$
where
$$ \slashed{\nabla}_3 \left( \slashed{\mathrm{tr}} \underline{\chi} + \frac{2}{v} \right) = \dot{F}_{\slashed{\mathrm{tr}} \underline{\chi} + \frac{2}{v} } . $$
We examine the $z_4$ term we have that
\begin{align*}
z_4 = & \sum_{m_1 + m_2 = m , m_1 \leq m_2} \int_{u_0}^u v \| [ (v \slashed{\nabla} )^{m_1} \hat{\underline{\chi}} ] [ (v \slashed{\nabla} )^{m_2} \hat{\underline{\chi}} ] \|_{L^2 (S_{u' , v})} \, du' \\ \lesssim &  \sum_{m_1 + m_2 = m , m_1 \leq m_2 , k \in \{0,1,2\} } \int_{u_0}^u  \|  (v \slashed{\nabla} )^{m_1 + k} \hat{\underline{\chi}} \|_{L^2 (S_{u' , v})}  \| (v \slashed{\nabla} )^{m_2} \hat{\underline{\chi}}  \|_{L^2 (S_{u' , v})} \, du' \\ \lesssim & \sum_{m_1 + m_2 = m , m_1 \leq m_2 , k \in \{0,1,2\} } \left( \int_{u_0}^u \mathrm{w}^2 (u' ) \| (v \slashed{\nabla} )^{m_1 + k} \hat{\underline{\chi}}  \|_{L^2 (S_{u' , v})}^2 \, du' \right)^{1/2} \left( \int_{u_0}^u  \frac{\mathrm{w}^2 (u' )}{\mathrm{w}^4 (u' )} \| (v \slashed{\nabla} )^{m_2} \hat{\underline{\chi}}  \|_{L^2 (S_{u' , v})}^2 \, du' \right)^{1/2}  \\ \lesssim & \frac{1}{\mathrm{w}^2 (u)} \sum_{m_1 + m_2 = m , m_1 \leq m_2 , k \in \{0,1,2\} } \left( \int_{u_0}^u \mathrm{w}^2 (u' ) \| (v \slashed{\nabla} )^{m_1 + k} \hat{\underline{\chi}}  \|_{L^2 (S_{u' , v})}^2 \, du' \right)^{1/2} \left( \int_{u_0}^u  \mathrm{w}^2 (u' ) \| (v \slashed{\nabla} )^{m_2} \hat{\underline{\chi}}  \|_{L^2 (S_{u' , v})}^2 \, du' \right)^{1/2} ,
\end{align*}
and we can now note that this term can be bounded by $\epsilon Z_4$ where $Z_4$ can be bounded by the bootstrap assumptions, after being squared and integrated in $u$ with $\mathrm{w}^2 (u)$ as a weight. The other terms can be treated similarly, and in the end we get the estimate
\begin{equation}\label{est:truchi}
\int_{u_0}^{u''} \mathrm{w}^2 (u) \left\| v ( v \slashed{\nabla} )^m \left( \slashed{\mathrm{tr}} \underline{\chi} + \frac{2}{v} \right) \right\|_{L^2 (S_{u,v} )} \, du \lesssim  \left\| v ( v \slashed{\nabla} )^m \left( \slashed{\mathrm{tr}} \underline{\chi} + \frac{2}{v} \right) \right\|_{L^{\infty}_v L^2 (S_{u_0 ,v} )} + \epsilon Z^m , 
\end{equation}
where $Z^m$ can be bounded by the bootstrap assumptions. 

For the fourth top-order derivative we start with $\slashed{\mathrm{tr}} \underline{\chi} + \frac{2}{v}$ and we use equation \eqref{eq:truchir3} which gives the estimate \eqref{est:truchi} for $m=4$ (the proof is identical to the one given above for $m \leq 3$). Now we can estimate the fourth top-order derivative for $\hat{\underline{\chi}}$ for which we need to control the quantity
$$ \int_{u_0}^{u''} \mathrm{w}^2 (u) \| ( v \slashed{\nabla} )^4 \hat{\underline{\chi}} \|_{L^2 (S_{u ,v})}^2 \, du , $$
for any $u'' \leq U$, by using the equation \eqref{eq:coduchi} and Proposition \ref{elliptic}. The terms involving $\slashed{\mathrm{tr}} \underline{\chi}$, $\hat{\underline{\chi}}$, $\eta$ and $\underline{\eta}$ can be treated similarly as above (in the estimates for $\slashed{\mathrm{tr}} \underline{\chi} + \frac{2}{v}$) while for the term involving $\underline{\beta}$ we get the term
$$ \int_{u_0}^u v \| ( v \slashed{\nabla} )^3 \underline{\beta} \|_{L^2 (S_{u' , v} )} \, du' \leq \left(  \int_{u_0}^u \frac{1}{\mathrm{w}^2 (u' )} \, du' \right)^{1/2} \cdot \left(  \int_{u_0}^u \mathrm{w}^2 (u' ) v^2 \| ( v \slashed{\nabla} )^3 \underline{\beta} \|_{L^2 (S_{u' , v} )}^2  \, du' \right)^{1/2} , $$
which after being squared and integrated in $u$ with the weight $\mathrm{w}^2 (u)$ can be bounded by the bootstrap assumptions.


\textbf{Ricci coefficients, 3. $\eta$ and $\underline{\eta}$:} We use the $\slashed{\nabla}_3$ equations \eqref{eq:eta3}, \eqref{eq:ueta3} for both $\eta$ and $\underline{\eta}$ up to 3 derivatives. The reason for not using the $\slashed{\nabla}_4$ equation \eqref{eq:eta4} for $\eta$ is a logarithmic divergence coming from $\beta$. For $\psi \in \{\eta , \underline{\eta} \}$ we have that:
\begin{align*}
v \|  \eta \|_{L^2 (S_{u_2 ,v} )} \lesssim & v \|  \psi \|_{L^2 (S_{u_1 ,v} )} + \int_{u_1}^{u_2} v \| \slashed{\mathrm{tr}}\underline{\chi} \cdot \psi \|_{L^2 (S_{u,v})} \, du \\& + \int_{u_1}^{u_2} v \| \hat{\underline{\chi}} \cdot \psi \|_{L^2 (S_{u,v})} \, du + \int_{u_1}^{u_2} v \|  \underline{\beta} \|_{L^2 (S_{u,v})} \, du \\ \doteq &  v \|  \psi \|_{L^2 (S_{u_1 ,v} )} + e_1 + e_2 +e_3  ,
\end{align*}
For the $e_1$ term we have that
\begin{align*}
e_1 = & \int_{u_1}^{u_2} v \| \slashed{\mathrm{tr}}\underline{\chi} \cdot \psi \|_{L^2 (S_{u,v})} \, du \\ \lesssim & \sum_{k=0}^2 \int_{u_1}^{u_2}  \| ( v \slashed{\nabla} )^k \slashed{\mathrm{tr}}\underline{\chi} \|_{L^2 (S_{u,v} )} \| \psi \|_{L^2 (S_{u,v})} \, du \\ \lesssim & \sum_{k=0}^2 \left( \int_{u_1}^{u_2} \mathrm{w}^2 (u) \| ( v \slashed{\nabla} )^k \slashed{\mathrm{tr}}\underline{\chi} \|_{L^2 (S_{u,v} )}^2 \, du \right)^{1/2} \left( \int_{u_1}^{u_2} \frac{1}{\mathrm{w}^2 (u)} \| ( v \slashed{\nabla} )^k \psi\|_{L^2 (S_{u,v} )}^2 \, du  \right)^{1/2} \\ \lesssim & \epsilon \| ( v \slashed{\nabla} )^k \psi\|_{L^{\infty}_u L^2 (S_{u,v} )}  \sum_{k=0}^2 \left( \int_{u_1}^{u_2} \mathrm{w}^2 (u) \| ( v \slashed{\nabla} )^k \slashed{\mathrm{tr}}\underline{\chi} \|_{L^2 (S_{u,v} )}^2 \, du \right)^{1/2} .
\end{align*}
Note that a similar estimate can be obtained once we apply up to 3 angular derivatives to it, by using Sobolev on the either $\slashed{\mathrm{tr}} \underline{\chi}$ or $\psi$, making sure that $\psi$ is never hit by more than 3 angular derivatives (which is always possible). The $e_2$ term can be treated similarly. For the $e_3$ term we can apply Cauchy-Schwarz and we get that
$$ e_3 = \int_{u_1}^{u_2} v \|  \underline{\beta} \|_{L^2 (S_{u,v})} \, du \lesssim \left( \int_{u_1}^{u_2} \frac{1}{\mathrm{w}^2 (u)} \, du \right)^{1/2} \left( \int_{u_1}^{u_2} \mathrm{w}^2 (u) v^2 \|  \underline{\beta} \|_{L^2 (S_{u,v})}^2 \, du \right)^{1/2} , $$
which is bounded by the bootstrap assumptions, and a similar estimate can be obtained once the same term is hit by up to 3 angular derivatives.

Finally (as far as $\eta$ and $\underline{\eta}$ are concerned) for the top order fourth derivatives, by the estimates \eqref{mu1}, \eqref{mu2} we have that
$$ \sup_u \int_v \mathrm{w}^2 (u) \| ( v \slashed{\nabla} )^4 \psi \|_{L^2 (S_{u,v} )}^2 \lesssim \int_v \mathrm{w}^2 (u) \| ( v \slashed{\nabla} )^4 \psi \|_{L^2 (S_{u_0 ,v} )}^2 + \epsilon \dot{Z}_1 , $$
$$ \sup_v \int_u v^2 \| ( v \slashed{\nabla} )^4 \psi \|_{L^2 (S_{u,v} )}^2 \lesssim \epsilon C_{\psi} + \epsilon \dot{Z}_2 , $$
where $\dot{Z}_1$, $\dot{Z}_2$ are bounded by the bootstrap assumptions, and $C_{\psi}$ is quantity depending on the initial data.

\textbf{Closing the bootstrap assumptions:} We gather together all the previous estimates and we get in the end that:
$$ \sum_{i=0}^4 \mathcal{O}_i +  v^2 \left\| ( v \slashed{\nabla} )^m \left( K - \frac{1}{v^2} \right) \right\|_{L^2 (S_{u ,v} )} \leq  \sum_{i=1}^4 \epsilon_i + \bar{C} \epsilon^{1/2}\sum_{i=1}^4 \epsilon_i  ,$$
for some constant $\bar{C}$, and we can finish our bootstrap argument by choosing $\epsilon$ to be small enough.

\begin{remark}
As in the local case in \cite{weaknull}, $H_U$ is a weak null singularity. This can be shown exactly in the same way as in section 9 of \cite{weaknull} and we refer the reader there for details.
\end{remark}

\section{Proof of Theorem \ref{thm:main2}}\label{p2}
We will only give an outline of the proof of Theorem \ref{thm:main2}. The main differences with the previous proof (as outlined before) is that we work with $\alpha$, $\check{\rho}$ and $\check{\sigma}$ instead of avoiding $\alpha$ and working with $K$ and $\check{\sigma}$, and we treat $\hat{\chi}$ using equation \eqref{eq:chi4} instead of \eqref{eq:chi3}. We will show in detail only the estimates for $\chi$ and $\alpha$ as the rest of the proof works in the case of Theorem \ref{thm:main}. Again we use a bootstrap argument where our bootstrap assumptions are the following:
\begin{equation}\label{bootstraps2}
\sum_{i=1}^4 \mathcal{O}_i^2 + v^{1+\delta} \left\| ( v \slashed{\nabla} )^m \left( K - \frac{1}{v^2} \right) \right\|_{L^2 (S_{u ,v} )} \leq \widetilde{C} \sum_{i=1}^4 \epsilon_i ,
\end{equation}

\paragraph{Estimates for $\hat{\chi}$:} We make use of inequality \eqref{est:b4p} , and we have that for $i \in \{ 0,1,2,3\}$:
\begin{align*}
v^{\delta} \left\| ( v \snabla )^i \hat{\chi} \right\|_{L^2 (S_{u,v} )} \lesssim &  v_0^{\delta} \left\| ( v_0 \snabla )^i \hat{\chi} \right\|_{L^2 (S_{u,v_0 } )} + \int_{v_0}^v (v')^{\delta} \| F^i_{\hat{\chi}} \|_{L^2 (S_{u,v'} )} ,
\end{align*}
where 
$$ \snabla_4 [ ( \snabla )^i \hat{\chi} ) ] + \frac{i+2}{2} \slashed{\mathrm{tr}}\chi [ ( \snabla )^i \hat{\chi} ) ] = \frac{1}{v^i} F^i_{\hat{\chi}} , $$
and so we have that
$$ F^i_{\hat{\chi}} \sim \sum_{\psi \in \{\eta , \underline{\eta} \}} \sum_{k_1 + k_2 + k_3 = i} ( \snabla^{k_2} \psi )^{k_2} ( \snabla^{k_3} F^0_{\hat{\chi}} ) + \sum_{\psi \in \{\eta , \underline{\eta}\} } \sum_{\psi_H \in \{ \hat{\chi} , \slashed{\mathrm{tr}} \chi - \frac{2}{v} \}} \sum_{k_1 + k_2 + k_3 + k_4 = i} ( \snabla^{k_1} \psi )^{k_2} ( \snabla^{k_3} \psi_H ) ( \snabla^{k_4} \hat{\chi} ) , $$ 
for 
$$ F^0_{\hat{\chi}} \doteq -2\omega \hat{\chi} - \alpha . $$
We look only at the part that involves
$$ \snabla^i F^0_{\hat{\chi}} \sim \sum_{i_1 + i_2 = i} ( \snabla^{i_1} \omega ) ( \snabla^{i_2} \hat{\chi} ) + \snabla^i \alpha , $$
as the rest of $F^i_{\hat{\chi}}$ consists of terms that are of similar difficulty to deal with, or easier. We also consider the case for $i=3$ (the rest can be treated similarly). We have that
\begin{align*}
\int_{v_0}^v ( v' )^{\delta} & \| ( v \snabla )^3 F^0_{\hat{\chi}} \|_{L^2 (S_{u,v'} )} \, dv' \sim  \int_{v_0}^v ( v' )^{\delta} \left( \sum_{i_1 + i_2 = 3} \| [ ( v' \snabla )^{i_1} ) \omega ] [ ( v' \snabla )^{i_2} \hat{\chi} ] \|_{L^2 ( S_{u,v' } )} + \| ( v' \snabla )^i \alpha \|_{L^2 (S_{u,v' })} \right) \, dv'  \\ \lesssim & \sum_{k \leq 3}  \int_{v_0}^v ( v' )^{-1+\delta} \| ( v' \snabla )^k \omega \|_{L^2 (S_{u,v' })} \| ( v' \snabla )^k \hat{\chi} \|_{L^2 (S_{u,v' })} \, dv' + \left( \int_{v_0}^v \frac{1}{(v' )^{1+\delta' + \delta}} \, dv' \right)^{1/2} \left( \int_{v_0}^v (v' )^{1+\delta'} \| ( v \snabla )^3 \alpha \|^2_{L^2 (S_{u,v' })} \, dv' \right)^{1/2} \\ \lesssim &   \sum_{k_1 \leq 3 , k_2 \leq 3}  \int_{v_0}^v ( v' )^{-2-\delta} \| ( v' )^{1+\delta} ( v' \snabla )^{k_1} \omega \|_{L^2 (S_{u,v' })} \| (v' )^{\delta} ( v' \snabla )^{k_2} \hat{\chi} \|_{L^2 (S_{u,v' })} \, dv' + \epsilon^{1/2} \| v^{5/2} ( v \snabla )^3 \alpha \|_{L^2 (H_u )} \\ \lesssim & \epsilon^{1/2} \left[ \sup_{v'} \sum_{k_1 \leq 3 , k_2 \leq 3} \left( \| ( v' )^{1+\delta} ( v' \snabla )^{k_1} \omega \|_{L^2 (S_{u,v' })} \| (v' )^{\delta} ( v' \snabla )^{k_2} \hat{\chi} \|_{L^2 (S_{u,v' })} \right) + \| v^{3/2} ( v \snabla )^3 \alpha \|_{L^2 (H_u )} \right] ,
\end{align*}
using Sobolev's inequality \eqref{sobolev3}, where smallness in $\epsilon$ can be achieved by the choice of $v_0$, and the other terms can be bounded by the bootstrap assumptions.

\paragraph{Estimates for $\alpha$:} We use equations \eqref{eq:alpha3} and \eqref{eq:beta4} and by Propositions \ref{intbyparts} and \ref{intbyparts3} we have that for $i \in \{0,1,2,3\}$:
\begin{align*}
\int_{H_{u_2}} v^{3} | ( v \snabla )^i \alpha |^2 + & \int_{\underline{H}_{v_2}} v_2^{3} | ( v_2 \snabla )^i \beta |^2 \lesssim \int_{H_{u_1}} v^{5} | ( v \snabla )^i \alpha |^2 + \int_{\underline{H}_{v_1}} v_1^{3} | ( v_1 \snabla )^i \beta |^2 \\ & + \int_{D_{u,v}} v^{3} \left( | \langle ( v \snabla )^i \alpha , F^i_{\alpha} \rangle | + | \langle ( v \snabla )^i \beta , F^i_{\beta , 4} \rangle | \right) ,
\end{align*}
where
$$ \snabla_3 ( \snabla^i \alpha ) + \frac{i+1}{2} \slashed{\mathrm{tr}} \underline{\chi} ( \snabla^i \alpha ) - \snabla \hat{\otimes} ( \snabla^i \beta )= \frac{1}{v^i} F^i_{\alpha} , $$
and 
$$ \snabla_4 ( \snabla^i \beta ) + \frac{i+4}{2} \slashed{\mathrm{tr}} \chi ( \snabla^i \beta ) - \slashed{\mathrm{div}} ( \snabla^i \alpha ) = \frac{1}{v^i} F^i_{\beta , 4} . $$
Note that we have that
$$ F^i_{\alpha} \sim \sum_{\psi \in \{\eta , \underline{\eta} \}} \sum_{k_1 + k_2 + k_3 = i} ( \snabla^{k_2} \psi )^{k_2} ( \snabla^{k_3} F^0_{\alpha} ) + \sum_{\psi \in \{\eta , \underline{\eta}\} } \sum_{\psi_{\underline{H}} \in \{ \underline{\hat{\chi}} , \slashed{\mathrm{tr}} \underline{\chi}  \}} \sum_{k_1 + k_2 + k_3 + k_4 = i} ( \snabla^{k_1} \psi )^{k_2} ( \snabla^{k_3} \psi_{\underline{H}} ) ( \snabla^{k_4} \alpha) , $$
$$ F^i_{\beta , 4} \sim  \sum_{\psi \in \{\eta , \underline{\eta} \}} \sum_{k_1 + k_2 + k_3 = i} ( \snabla^{k_2} \psi )^{k_2} ( \snabla^{k_3} F^0_{\beta , 4} ) + \sum_{\psi \in \{\eta , \underline{\eta}\} } \sum_{\psi_H \in \{ \hat{\chi} , \slashed{\mathrm{tr}} \chi - \frac{2}{v} \}} \sum_{k_1 + k_2 + k_3 + k_4 = i} ( \snabla^{k_1} \psi )^{k_2} ( \snabla^{k_3} \psi_H ) ( \snabla^{k_4} \beta) , $$ 
for
$$ F^0_{\alpha} = - 3 ( \hat{\chi} \rho +^{*} \hat{\chi} \sigma ) + ( \zeta + 4 \eta ) \hat{\otimes} \beta , $$
and
$$ F^0_{\beta , 4} = -2\omega \beta + \eta \alpha . $$
Once again we will examine only the parts of the inhomogeneity that involve $\snabla^i F^0_{\alpha}$ and $\snabla^i F^0_{\beta , 4}$ as the rest can be treated either similarly or are easier to deal with, and we will take $i=3$ as the other derivatives can be dealt with similarly. First we have that
\begin{align*}
\int_{D_{u,v}} v^{3}  | \langle ( v \snabla )^3 \alpha , ( v \snabla )^3 F^0_{\alpha} \rangle | \lesssim & \sum_{i_1+i_2=3} \int_{D_{u,v}} v^{3} | ( v \snabla )^3 \alpha |  | ( v \snabla )^{i_1} \hat{\chi} | | ( v \snabla )^{i_2} \check{\rho} | +  \sum_{i_1+i_2=3} \int_{D_{u,v}} v^{3} | ( v \snabla )^3 \alpha |  | ( v \snabla )^{i_1} \hat{\chi} | | ( v \snabla )^{i_2} \check{\sigma} | \\ & +  \sum_{i_1+i_2+i_3=3} \int_{D_{u,v}} v^{3} | ( v \snabla )^3 \alpha |  | ( v \snabla )^{i_1} \hat{\chi} | | ( v \snabla )^{i_2} \hat{\chi} | | ( v \snabla )^{i_3} \underline{\hat{\chi}} | \\ &+  \sum_{i_1+i_2=3} \sum_{\psi \in \{ \eta , \underline{\eta} \}} \int_{D_{u,v}} v^{3} | ( v \snabla )^3 \alpha |  | ( v \snabla )^{i_1} \psi | | ( v \snabla )^{i_2} \beta | \\ \doteq & a_1 + a_2 + a_3 + a_4 .
\end{align*}
We then have that
\begin{align*}
a_1 \doteq & \sum_{i_1+i_2=3} \int_{D_{u,v}} v^{3} | ( v \snabla )^3 \alpha |  | ( v \snabla )^{i_1} \hat{\chi} | | ( v \snabla )^{i_2} \check{\rho} | \\ \lesssim & \epsilon \sup_{u \in [u_1 , u_2 ]} \int_{H_u} v^3 | ( v \snabla )^3 \alpha |^2 + \sum_{i_1 + i_2 = 3} \int_{D_{u,v}} v^3 | ( v \snabla )^{i_1} \hat{\chi} |^2 | ( v \snabla )^{i_2} \check{\rho} |^2 \\ \lesssim & \epsilon \sup_{u \in [u_1 , u_2 ]}\int_{H_u} v^3 | ( v \snabla )^3 \alpha |^2 + \sup_{v'} \sum_{k_1 \leq 3 , k_2 \leq 3} \| ( v' )^{\delta} ( v' \snabla )^3 \hat{\chi} \|_{L^2 (S_{u,v'} )}^2 \sup_{u \in [ u_1 , u_2 ]} \left( \int_{H_u} v^{1-2\delta} | ( v \snabla )^{k_2} \check{\rho} |^2 \right) ,  
\end{align*}
where we gained smallness in $\epsilon$ by the choice of $U$ and used Sobolev's inequality \eqref{sobolev3}. We note that the first term of the last expression can be absorbed by the left-hand side of the initial estimate, while the last term can be bounded by the bootstrap assumptions. The term $a_2$ can be treated similarly. For $a_3$ we have that
\begin{align*}
a_3 \doteq & \sum_{i_1+i_2+i_3=3} \int_{D_{u,v}} v^{3} | ( v \snabla )^3 \alpha |  | ( v \snabla )^{i_1} \hat{\chi} | | ( v \snabla )^{i_2} \hat{\chi} | | ( v \snabla )^{i_3} \underline{\hat{\chi}} | \\ \lesssim & \epsilon \sup_{u \in [u_1 , u_2 ]}\int_{H_u} v^3 | ( v \snabla )^3 \alpha |^2 + \sum_{i_1+i_2+i_3=3} \int_{D_{u,v}} v^{3}  | ( v \snabla )^{i_1} \hat{\chi} |^2 | ( v \snabla )^{i_2} \hat{\chi} |^2 | ( v \snabla )^{i_3} \underline{\hat{\chi}} |^2 \\ \lesssim &  \epsilon \sup_{u \in [u_1 , u_2 ]}\int_{H_u} v^3 | ( v \snabla )^3 \alpha |^2 \\ & + \sup_{v' }\sum_{k_1 , k_2  \leq 3} \left( \| (v' )^{\delta} ( v' \snabla )^{k_1} \hat{\chi} \|_{L^2 (S_{u,v' } )}^4 \left\| \frac{1}{(v' )^{1-\delta}} ( v' \snabla )^{k_2} \underline{\chi} \right\|_{L^2 (S_{u,v' } )}^2 \right) \left( \int_{u_1}^{u_2} \int_{v_1}^{v_2} \frac{1}{\widetilde{\mathrm{w}}^2 (u)} \frac{v^5}{v^{4+6\delta}} ,\ dv du   \right) ,
\end{align*}
and the first term of the last expression can be absorbed by the left-hand side of the initial estimate, while the second term can be bounded by the bootstrap assumptions due to the assumptions of $\widetilde{\mathrm{w}}$ and that $\delta > \frac{1}{3}$. Lastly for $a_4$ we consider only the term involving $\eta$ (as the other one can be treated similarly) and we have that
\begin{align*}
\sum_{i_1+i_2=3}  \int_{D_{u,v}} v^{3} & | ( v \snabla )^3 \alpha |  | ( v \snabla )^{i_1} \eta | | ( v \snabla )^{i_2} \beta | \lesssim  \epsilon \sup_{u \in [u_1 , u_2 ]}\int_{H_u} v^3 | ( v \snabla )^3 \alpha |^2 + \sum_{i_1+i_2=3}  \int_{D_{u,v}} v^{3} | ( v \snabla )^{i_1} \eta |^2 | ( v \snabla )^{i_2} \beta |^2 \\ \lesssim & \epsilon \sup_{u \in [u_1 , u_2 ]}\int_{H_u} v^3 | ( v \snabla )^3 \alpha |^2 + \sum_{k_1 , k_2 \leq 3}  \sup_{v'} \| ( v' )^{\delta} ( v' \snabla )^{k_1} \eta \|_{L^2 (S_{u,v' })}^2 \sup_{u \in [u_1 , u_2 ]} \int_{H_u} v^{3-2\delta} | ( v \snabla )^{k_2} \beta |^2  ,
\end{align*}
and once again the first term of the last expression can be absorbed by the left-hand side of the initial estimate, while the second term can be bounded by the bootstrap assumptions.

On the other hand for $F_{\beta , 4}^0$ we have that
\begin{align*}
\int_{D_{u,v}} v^3 | \langle ( v \snabla )^3 \beta , ( v \snabla )^3 F^0_{\beta , 4} \rangle | \lesssim & \sum_{i_1 + i_2 = 3} \int_{D_{u,v}} v^3 \left( | ( v \snabla )^3 \beta | | ( v \snabla )^{i_1} \omega | | ( v \snabla )^{i_2} \beta | + | ( v \snabla )^3 \beta | | ( v \snabla )^{i_1} \eta | | ( v \snabla )^{i_2} \alpha  | \right) \\ \doteq & b^4_1 + b^4_2 .
\end{align*}
Then we have that
\begin{align*}
b^4_1 \doteq & \sum_{i_1 + i_2 = 3} \int_{D_{u,v}} v^3 | ( v \snabla )^3 \beta | | ( v \snabla )^{i_1} \omega | | ( v \snabla )^{i_2} \beta | \\ \lesssim & \epsilon \sup_{u \in [u_1 , u_2 ]} \int_{H_u} v | ( v \snabla )^3 \beta |^2 + \sum_{i_1 + i_2 = 3} \int_{D_{u,v}} v^5  ( v \snabla )^{i_1} \omega |^2 | ( v \snabla )^{i_2} \beta |^2 \\ \lesssim & \epsilon \sup_{u \in [u_1 , u_2 ]} \int_{H_u} v | ( v \snabla )^3 \beta |^2 + \sum_{k_1 , k_2 \leq 3} \sup_{v'} \| (v' )^{1+\delta} ( v' \snabla )^{k_1} \omega \|_{L^2 (S_{u,v' } )}^2 \sup_{u \in [u_1 , u_2 ]} \int_{H_u} v^{1-2\delta} | ( v \snabla )^{k_2} \beta |^2 ,
\end{align*}
and the first term of the last expression can be absorbed by the left-hand side of the initial estimate while the rest can be bounded by the bootstrap assumptions. Finally we have that
\begin{align*}
b^4_2 \doteq & \sum_{i_1 + i_2 = 3} \int_{D_{u,v}} v^3  | ( v \snabla )^3 \beta | | ( v \snabla )^{i_1} \eta | | ( v \snabla )^{i_2} \alpha  | \\ \lesssim & \epsilon \sup_{u \in [u_1 , u_2 ]} \int_{H_u} v | ( v \snabla )^3 \beta |^2 + \sum_{i_1 + i_2 = 3} \int_{D_{u,v}} v^5 | ( v \snabla )^{i_1} \eta |^2 | ( v \snabla )^{i_2} \alpha  |^2  \\ \lesssim & \epsilon \sup_{u \in [u_1 , u_2 ]} \int_{H_u} v | ( v \snabla )^3 \beta |^2 + \sum_{k_1 , k_2 \leq 3} \sup_{v'} \| ( v' )^{\delta} ( v' \snabla )^{k_1} \eta \|_{L^2 (S_{u,v'} )}^2 \sup_{u \in [ u_1 , u_2 ]} \int_{H_u} v^{3-2\delta}  | ( v \snabla )^{k_2} \alpha  |^2  ,
\end{align*}
and once again the first term of the last expression can be absorbed by the left-hand side of the initial estimate while the rest can be bounded by the bootstrap assumptions.

\begin{remark}
If we assume that $\hat{\underline{\chi}}$ has a jump discontinuity on $\underline{H}_{v_0}$ across the sphere $S_{u_j , v_0}$, then working as in \cite{iwaves1} we can show that $\alpha$ is a measure supported on $H_{u_j}$ (by repeating the arguments of Proposition 75 of \cite{iwaves1}).

Moreover let us note that the estimate for the quantity
$$ \left\| \widetilde{\mathrm{w}} (u) \frac{1}{v^{1-\delta}} \hat{\underline{\chi}} \right\|_{L^2_u L_v^{\infty} L^2 (S_{u,v} )} , $$
gives us as a consequence that if $\hat{\underline{\chi}}$ has a jump discontinuity on the initial hypersurface $\underline{H}_{v_0}$ then the size of the jump remains controlled (and comparable with respect to the constant $C$) by the initial jump on every hypersurface $\underline{H}_v$. 

Our estimates allow us also to define the following quantities:
$$ \Sigma_{\delta} (u , \theta^1 , \theta^2 ) \doteq \lim_{v \rightarrow \infty} v^{1+\delta} \hat{\chi} , \quad \quad \Xi_{\delta} (u , \theta^1 , \theta^2 ) \doteq v^{\delta} \hat{\underline{\chi}} , $$
which are connected through the equation:
$$ \frac{\partial \Sigma_{\delta}}{\partial u} = -\frac{1}{2} \Xi_{\delta} . $$
By also defining the Hawking mass:
$$ m(u,v) \doteq \frac{v}{2} \left( 1+ \frac{1}{16\pi} \int_{S_{u,v}} \slashed{\mathrm{tr}} \chi \slashed{\mathrm{tr}} \underline{\chi} \right) , $$
our estimates also imply the existence of the limit
$$ M(u) \doteq \lim_{v \rightarrow \infty} m (u,v) , $$
which can also be seen to satisfy the equation
$$ \frac{\partial M}{\partial u} = \frac{1}{8\pi} \lim_{v \rightarrow \infty} \left( \int_{S_{u,v}} | \Xi_{\delta} |^2 \right) , $$
which can now give us a mass loss estimate at $\mathcal{I}^{+}$ using the estimate for $\left\| \widetilde{\mathrm{w}} (u) \frac{1}{v^{1-\delta}} \hat{\underline{\chi}} \right\|_{L^2_u L_v^{\infty} L^2 (S_{u,v} )}$.

\end{remark}

\section{Aside: some variations of the two main Theorems}\label{aside}
In this section we make some remarks in the case that we substitute the assumptions on $\hat{\chi}$ with the expected sharp decay in $v$.

\paragraph{Remarks on Theorem \ref{thm:main}:} In Theorem \ref{thm:main} we assume that
\begin{equation}\label{chisharp}
\sum_{i=0}^4 \| v ( v \snabla )^i \hat{\chi} \|_{L^{\infty}_v L^2 (S_{u_0,v} )} \leq \bar{\epsilon} , 
\end{equation}
for some $\bar{\epsilon} > 0$ small enough. This allows us also to assume the following for the Riemann curvature components:
\begin{align*}
\sum_{i=0}^3  \Bigg(  \| v^2 ( v \slashed{\nabla} )^i \beta \|_{L^2_v L^2 (S_{u_0 , v} ) } + &  \sum_{\Psi \in \{ K - \frac{1}{v^2} , \check{\sigma} \}} \| v_0^2 ( v_0 \slashed{\nabla} )^i \Psi_1 \|_{L^2_u L^2 (S_{u , v_0 } )}  \\ + & \| \mathrm{w} (u) v_0 ( v_0 \slashed{\nabla} )^i \underline{\beta} \|_{L^2_u L^2 (S_{u, v_0} )}   + \sum_{\Psi \in \{ K - \frac{1}{v^2} , \check{\sigma} \}} \| \mathrm{w} (u_0 ) v ( v \slashed{\nabla} )^i \Psi_2 \|_{L^2_v L^2 (S_{u_0 , v } ) } \Bigg) \leq \widetilde{\epsilon} ,
\end{align*}
for some $\widetilde{\epsilon} >0$ small enough. Then using the same proof as for Theorem \ref{thm:main}, we can show the following bounds:
$$ \sum_{i=0}^3 \| v ( v \snabla )^i \hat{\chi} \|_{L^{\infty}_{u,v} L^2 (S_{u,v} )} + \sup_u \| ( v \snabla )^4 \hat{\chi} \|_{L^2 (H_u )} \leq C , $$
and
\begin{align*}
\sum_{i=0}^3  \Bigg(  \| v^{2} ( v \slashed{\nabla} )^i \beta \|_{L^{\infty}_u L^2_v L^2 (S_{u , v} ) } + & \sum_{\Psi \in \{ K - \frac{1}{v^2} , \check{\sigma} \}} \| v^{2} ( v \slashed{\nabla} )^i \Psi_1 \|_{L^{\infty}_v L^2_u L^2 (S_{u , v } )}   \\ + & \| \mathrm{w} (u) v ( v \slashed{\nabla} )^i \underline{\beta} \|_{L^{\infty}_v L^2_u L^2 (S_{u, v} )}   + \sum_{\Psi \in \{ K - \frac{1}{v^2} , \check{\sigma} \}} \| \mathrm{w} (u ) v ( v \slashed{\nabla} )^i \Psi_2 \|_{L^{\infty}_u L^2_v L^2 (S_{u , v } ) } \Bigg) \leq C ,
\end{align*}
for $C$ depending on $\bar{\epsilon}$, $\widetilde{\epsilon}$ and the rest of the initial constants from Theorem \ref{thm:main}. Note that for the first bound we use again equation \eqref{eq:chi3} but we do not need to integrate in $v$ as we only care about $L^{\infty}_{u,v} L^2 (S_{u,v})$ norms.

\paragraph{Remarks on Theorem \ref{thm:main2}:} In this case we make the same assumption as \eqref{chisharp} by taking $\delta =1$ everywhere in the assumptions of Theorem \ref{thm:main2}. We can then assume the following for the Riemann curvature components:
\begin{align*}
\sum_{i=0}^3  \Bigg( \| v^{3} ( v \slashed{\nabla} )^i \alpha \|_{L^2_v L^2 (S_{u_0 , v} ) }  + & \| v^{2} ( v \slashed{\nabla} )^i \beta \|_{L^2_v L^2 (S_{u_0 , v} ) } \\ + & \| v_0^{3} ( v_0 \slashed{\nabla} )^i \beta \|_{L^2_u L^2 (S_{u , v_0 } )} +\sum_{\Psi \in \{ \check{\rho} , \check{\sigma} \}} \| v_0^{2} ( v_0 \slashed{\nabla} )^i \Psi_1 \|_{L^2_u L^2 (S_{u , v_0 } )}  \\ & + \| \widetilde{\mathrm{w}} (u) v_0 ( v_0 \slashed{\nabla} )^i \underline{\beta} \|_{L^2_u L^2 (S_{u, v_0} )}   + \sum_{\Psi \in \{ \check{\rho} , \check{\sigma} \}} \| \widetilde{\mathrm{w}} (u_0 ) v ( v \slashed{\nabla} )^i \Psi_2 \|_{L^2_v L^2 (S_{u_0 , v } ) } \Bigg) \leq \hat{\epsilon} ,
\end{align*}
for some $\hat{\epsilon} > 0$ small enough. The same method of proof allows us to obtain the following bounds (aside from the ones for the Ricci components given in Theorem \ref{thm:main2} for $\delta =1$):
\begin{align*}
\sum_{i=0}^3  \Bigg(  \| v^{3} ( v \slashed{\nabla} )^i \alpha \|_{L^{\infty}_u L^2_v L^2 (S_{u , v} )}+&  \| v^{2} ( v \slashed{\nabla} )^i \beta \|_{L^{\infty}_u L^2_v L^2 (S_{u , v} ) } \\ + &  \| v^{3} ( v \slashed{\nabla} )^i \beta \|_{L^{\infty}_v L^2_u L^2 (S_{u , v } )} + \sum_{\Psi \in \{ \check{\rho} , \check{\sigma} \}} \| v^2 ( v \slashed{\nabla} )^i \Psi_1 \|_{L^{\infty}_v L^2_u L^2 (S_{u , v } )}    \\ & + \| \widetilde{\mathrm{w}} (u) v ( v \slashed{\nabla} )^i \underline{\beta} \|_{L^{\infty}_v L^2_u L^2 (S_{u, v} )}   + \sum_{\Psi \in \{ \check{\rho} , \check{\sigma} \}} \| \widetilde{\mathrm{w}} (u ) v ( v \slashed{\nabla} )^i \Psi_2 \|_{L^{\infty}_u L^2_v L^2 (S_{u , v } ) } \Bigg) \leq C ,
\end{align*}
where $C$ depends on $\hat{\epsilon}$ and the other initial constants from the assumptions of Theorem \ref{thm:main2}.

\end{document}